\documentclass[10pt]{amsart}
\usepackage{amsxtra, amsfonts, amsmath, amsthm, amstext, amssymb, amscd, mathrsfs, verbatim, color}
\usepackage{threeparttable}
\usepackage{mathtools}
\usepackage[ansinew]{inputenc}\usepackage[T1]{fontenc}
\usepackage[all,cmtip]{xy}
\usepackage{tikz} 
\usetikzlibrary{shapes.geometric, arrows}  

\theoremstyle{plain}

\newtheorem{theorem}{Theorem}[section]
\newtheorem{lemma}[theorem]{Lemma}
\newtheorem{definition-theorem}[theorem]{Definition-Theorem}
\newtheorem{proposition}[theorem]{Proposition}
\newtheorem{corollary}[theorem]{Corollary}

\theoremstyle{definition}
\newtheorem{definition}[theorem]{Definition}
\newtheorem{example}[theorem]{Example}
\newtheorem{remark}[theorem]{Remark}
\newtheorem{notation}[theorem]{Notation}
\newcommand \bth[1] { \begin{theorem}\label{t#1} }
\newcommand \ble[1] { \begin{lemma}\label{l#1} }

\newcommand \bpr[1] { \begin{proposition}\label{p#1} }
\newcommand \bco[1] { \begin{corollary}\label{c#1} }
\newcommand \bde[1] { \begin{definition}\label{d#1}\rm }
\newcommand \bex[1] { \begin{example}\label{e#1}\rm }
\newcommand \bre[1] { \begin{remark}\label{r#1}\rm }

\newcommand \bnota[1] {\begin{notation}\label{n#1}\rm }
\newcommand {\ele} { \end{lemma} }

\newcommand {\epr} { \end{proposition} }
\newcommand {\eco} { \end{corollary} }
\newcommand {\ede} { \end{definition} }
\newcommand {\eex} { \end{example} }
\newcommand {\ere} { \end{remark} }
\newcommand {\enota} { \end{notation} }

\begin{document}
\title[Quotient branching laws]{Quotient branching law for $p$-adic $(\mathrm{GL}_{n+1}, \mathrm{GL}_n)$ I: generalized Gan-Gross-Prasad relevant pairs } 

\author[Kei Yuen Chan]{Kei Yuen Chan}
\address{Department of Mathematics, The University of Hong Kong}
\email{kychan1@hku.hk}
\maketitle
\setcounter{tocdepth}{1}

\begin{abstract}
Let $G_n=\mathrm{GL}_n(F)$ be the general linear group over a non-Archimedean local field $F$. We formulate and prove a necessary and sufficient condition on determining when 
\[  \mathrm{Hom}_{G_n}(\pi, \pi') \neq 0 
\]
for irreducible smooth representations $\pi$ and $\pi'$ of $G_{n+1}$ and $G_n$,  respectively. This resolves the problem of the quotient branching law. 

We also prove that any simple quotient of a Bernstein-Zelevinsky derivative of an irreducible representation can be constructed by a sequence of derivatives of essentially square-integrable representations. This result transferred to affine Hecke algebras of type A gives a generalization of the classical Pieri's rule of symmetric groups. 

One key new ingredient is a characterization of the layer in the Bernstein-Zelevinsky filtration that contributes to the branching law, obtained by the multiplicity one theorem for standard representations, which also gives a refinement of the branching law. Another key new ingredient is constructions of some branching laws and simple quotients of Bernstein-Zelevinsky derivatives by taking certain highest derivatives.
\end{abstract}

\tableofcontents
\part{Introduction}

\section{Background}

Let $G_n=\mathrm{GL}_n(F)$, the general linear group over a local non-Archimedean field $F$. All representations of $G_n$ are smooth and over $\mathbb C$. We regard $G_n$ as a subgroup of $G_{n+1}$ via the embedding $g \mapsto \begin{pmatrix} g & \\ & 1 \end{pmatrix}$. The quotient branching law asks: for an irreducible representation $\pi$ of $G_{n+1}$ and an irreducible representation $\pi'$ of $G_n$, determine the space $\mathrm{Hom}_{G_n}(\pi, \pi')$, where $\pi$ is viewed as $G_n$-representation via restriction to the subgroup. The multiplicity-at-most-one theorem, first fully established by Aizenbud-Gourevitch-Rallis-Schiffmann \cite{AGRS10} using distributions, asserts that:
\[  \mathrm{dim}~\mathrm{Hom}_{G_n}(\pi, \pi') \leq 1 .
\]
An alternative proof heavily using Bernstein-Zelevinsky (BZ) theory and the multiplicity-one theorem \cite{GK71, Sh74} of Whittaker models is given in \cite{Ch21+}. As mentioned in \cite[Page 1411]{AGRS10}, a more difficult question is to determine when the dimension of the Hom space is one. We give a practical solution to this question in terms of the Langlands correspondence and in full generality (for $p$-adic general linear groups). Along the way, we also solve the problem of an affine Pieri rule. Further applications of our results on other branching problems will also be considered elsewhere, see e.g. \cite{CW25}.

We review some previously known cases of the quotient branching law in the literature:
\begin{enumerate}
\item When $\pi$ is cuspidal and $\pi'$ is arbitrary, it is known by the work of Bernstein-Zelevinsky \cite{BZ77} by restricting to a mirabolic subgroup.
\item When both $\pi$ and $\pi'$ are generic, the Hom is always non-zero by the work of Jacquet--Piateski-Shapiro--Shalika \cite{JPSS83} using Rankin-Selberg integrals. We also point out a related study on the construction of Shintani functions for spherical representations by Murase-Sugano \cite{MS96}. 
\item When $\pi'$ is the trivial representation and $\pi$ is arbitrary, it is established for $n=2$ by D. Prasad \cite{Pr93} and completed for general $n$ by his student Venketasubramanian \cite{Ve13}.
\item When $\pi$ is a generalized Steinberg representation and $\pi'$ is arbitrary, it is known in the work of Chan-Savin \cite{CS21}, in which we use both left and right BZ theory.
\item When both $\pi$ and $\pi'$ are Arthur-type representations, a conjecture determining their quotient branching law is formulated by Gan-Gross-Prasad \cite{GGP20}. Earlier evidence dealing with one direction is established by M. Gurevich \cite{Gu22} using results in quantum groups. The conjecture is now settled by Chan \cite{Ch22}, in which we use some functorial properties of parabolic inductions \cite{Ch22+b}.
\end{enumerate}

We shall give a condition unifying all above cases. The condition for Arthur type representations in \cite{GGP20} is called {\it relevant}. Following their terminology, we shall call our condition (Definition \ref{def g relevant pair}) to be {\it generalized GGP relevant} (or simply relevant).

In a more general framework of the relative Langlands program developed by Sakellaridis-Venkatesh \cite{SV18}, our main result also solves the smooth distinction problem for the homogeneous spherical variety $\Delta G_n\setminus (G_{n+1}\times G_n)$. 

We remark that the unitary branching law from $G_{n+1}$ to $G_n$ could follow from the Kirillov conjecture \cite{Ki62} for restricting to the mirabolic subgroup, proved by Bernstein \cite{Be84}. On the other hand, the constituents in the restriction are also described as a special case of Clozel's conjecture \cite{Cl04}, proved by Venkatesh \cite{Ve05}, which is also a motivation for the non-tempered conjecture in \cite{GGP20}.

We now turn to discussions on BZ derivatives. For more discussions, see the introduction of \cite{Ch22+}. Here we only mention some classical important results: the highest BZ derivatives of irreducible representations by Zelevinsky \cite{Ze80}, partial BZ derivatives of ladder representations by Tadi\'c \cite{Ta87} and Lapid-M\'inguez \cite{LM14}, and an Hecke algebra approach of BZ derivatives for Speh representations in \cite{CS19}. As seen in \cite{Ch21} and \cite{Ch22+} as well as in this article, there is a large amount of interplay between BZ derivatives and branching laws, as suggested by their theory \cite{BZ77}. 

Our another main result asserts that any simple quotient of a BZ derivative of an irreducible representation can be obtained by a sequence of derivatives of essentially square-integrable representations. Translating this to the setup of affine Hecke algebras, this result essentially generalizes Pieri's rule \cite{Pi93} from symmetric groups to affine Hecke algebras of type A with generic parameters. We refer the reader to \cite{CS19, Ch22+} for the details of transferring those results, including the multiplicity-freeness for socles and cosocles in \cite{Ch22+}. We plan to discuss the combinatorics aspect of such generalized rule in the sequel.

A special case in this theme is obtained by Grojnowski-Vazirani \cite{GV01} and Vazirani \cite{Va02}. Indeed, our terminologies in defining the generalized relevance combinatorially (more precisely, Definition \ref{def comb commute}) and the use of derivatives also reflect our original viewpoint from using Hecke algebras, started from the work \cite{CS19, CS21, Ch21} (c.f. symplectic root number and $L$-function perspective in \cite{GGP12, GGP22+}). It seems that the proof of \cite{GV01} cannot be directly adapted to proving our generalized Pieri's rule (but see some related ideas for $\eta$-invariants), and our proof essentially uses some nature of branching laws specific to $p$-adic groups. 

As we are working representations over $\mathbb C$, the modular quotient branching laws for symmetric groups and finite Hecke algebras of type A by Kleshchev \cite{Kl95} and Brundan \cite{Br98} respectively are not in our scope at this point while one may hope that developments from Vigneras' school on mod $\ell$ representations \cite{Vi96} shed light on this aspect, see e.g. the work of S\'echerre-Venketasubramanian \cite{SV17} on the distinguished case.


Indeed, it is known in \cite{GGP20} that the original GGP relevance is not a necessary condition for the quotient branching law for other classical groups, and so we hope this work also sheds light on the quotient branching law outside GL. The work on generic case by M\oe glin-Waldspurger \cite{MW12} for orthogonal groups, adapted to unitary groups by Beuzart-Plessis \cite{BP16}, has similar spirit to the use of BZ filtrations. For more background and results on the GGP problems, see a survey \cite{Gr21+}. On the other hand, the works of Chan-Savin \cite{CS20} and Baki\'c-Savin \cite{BS22} have started the line of research for orthogonal groups along the GL ones from Hecke algebra viewpoint \cite{CS19, Ch21}.

An irreducible representation of $G_{n+1}$ restricting to a $G_n$-representation is far from being semisimple in general. In \cite{Pr18}, D. Prasad initiated to study some higher structure in the branching law. In particular, he conjectured the higher Ext groups for generic representations, which is now proved in \cite{CS21} and extended to standard representations in \cite{Ch21+}. In \cite{Ch21}, we determine all the indecomposable components of an irreducible representation restricted from $G_{n+1}$ to $G_n$, classify irreducible representations which are projective under restriction and determine the Ext-group at the cohomological degree in quotient branching law, marking some substantial progress in homological branching laws. However, a more difficult question, in determining when $\mathrm{Ext}^i_{G_n}(\pi, \pi')$ is non-zero or even determining the precise dimensions, remains largely open.

\section{Generalized GGP relevance}

\subsection{Basic notations} \label{ss basic notations}
All representations we consider are smooth and over $\mathbb C$, and we shall usually omit those descriptions. We sometimes do not distinguish representations in the same isomorphism class. Let $\mathrm{Alg}(G)$ be the category of smooth representations of a reductive group $G$ over $\mathbb C$. For $\pi \in \mathrm{Alg}(G)$, let $\pi^{\vee}$ be its smooth dual. Let $\mathrm{Irr}(G)$ be the set of irreducible representations of $G$.

Let $\mathrm{Irr}=\sqcup_n \mathrm{Irr}(G_n)$. Let $\mathrm{Irr}^c(G_n)$ be the set of irreducible cuspidal representations of $G_n$ and let $\mathrm{Irr}^c=\sqcup_n \mathrm{Irr}^c(G_n)$. Let $S_n$ be the permutation group on $n$ elements.

For $n_1+\ldots+n_r=n$, let $P_{n_1, \ldots, n_r}$ be the associated standard parabolic subgroup in $G_n$ i.e. the parabolic subgroup containing upper triangular matrices and matrices $\mathrm{diag}(g_1, \ldots, g_r)$ for $g_i \in G_{n_i}$. Denote by $\mathrm{Ind}_P^{G_n}$ the normalized parabolic induction from a parabolic subgroup $P$ of $G_n$. As usual, for $\pi_1 \in \mathrm{Alg}(G_{n_1})$ and $\pi_2 \in \mathrm{Alg}(G_{n_2})$, let $\pi_1 \times \pi_2 =\mathrm{Ind}_{P_{n_1,n_2}}^{G_n}\pi_1\boxtimes \pi_2$. 

For $n_1+\ldots+n_r=n$, let $N_{n_1, \ldots, n_r}$ be the unipotent radical in $P_{n_1, \ldots, n_r}$. For $\pi \in \mathrm{Alg}(G_n)$, denote by $\pi_{N_{n_1, \ldots, n_r}}$ its associated Jacquet module, viewed as a $G_{n_1}\times \ldots \times G_{n_r}$-representation. We shall sometimes abbreviate $N_{n_1, n_2}$ by $N_{n_2}$ to lighten notations. (The choice for $n_2$ is compatible with our convention to usually consider derivatives taken on the "right", see Section \ref{ss notion simplify} below.) Let $U_n=N_{1,\ldots, 1}$, the subgroup of all upper triangular unipotent matrices in $G_n$.

We now introduce some combinatorial objects following \cite{Ze80}. A {\it segment} $\Delta$ is a datum of the form $[a,b]_{\rho}$ for $a,b \in \mathbb Z$ and $\rho \in \mathrm{Irr}^c$. For $g \in G_n$, let $\nu(g)=|\mathrm{det}(g)|_F$ be the non-Archimedean absolute value of the determinant of $g$. We also write $a(\Delta)=\nu^a\rho$ and $b(\Delta)=\nu^b\rho$. We sometimes write $[a]_{\rho}$ for $[a,a]_{\rho}$, and write $[a,b]$ for $[a,b]_1$ (here $1$ is the trivial representation of $G_1$). As standard, we also consider $[a,a-1]_{\rho}$ as the empty set. We shall also regard a segment $[a,b]_{\rho}$ as a set $\left\{ \nu^a\rho, \ldots, \nu^b\rho \right\}$ so that one can consider the intersection and union of two segments. A {\it multisegment} is a multiset of non-empty segments. (For our convention and convenience, we also consider an empty set to be a segment and a multisegment.) For a segment, let $l_a(\Delta)=(b-a+1)n(\rho)$, called the {\it absolute length} of $\Delta$; and let $l_r(\Delta)=b-a+1$, called the {\it relative length} of $\Delta$. For a multisegment $\mathfrak m=\left\{ \Delta_1, \ldots, \Delta_r \right\}$, let $l_a(\mathfrak m)=\sum_i l_a(\Delta_i)$, and let $\mathrm{csupp}(\mathfrak m)=\cup_{\Delta} \Delta$. Let $\mathrm{Mult}$ be the set of multisegments.

Two segments $\Delta$ and $\Delta'$ are {\it linked} if $\Delta\cup \Delta'$ is still a segment. For two cuspidal representations $\rho_1$ and $\rho_2$, we write $\rho_1 <\rho_2$ if there exists a positive integer $c$ such that $\nu^c \rho_1\cong \rho_2$, and write $\rho_1 \leq \rho_2$ if $\rho_1<\rho_2$ or $\rho_1\cong \rho_2$. We write $\Delta_1<\Delta_2$ if $\Delta_1$ and $\Delta_2$ are linked and $b(\Delta_1) <b(\Delta_2)$, and write $\Delta_1 \leq \Delta_2$ if $\Delta_1<\Delta_2$ or $\Delta_1=\Delta_2$. 

For $\pi \in \mathrm{Irr}$, there exists a unique collection of cuspidal representations $\rho_1, \ldots, \rho_k$ such that $\pi$ is a simple composition factor $\rho_1\times \ldots \times \rho_k$. Let $\mathrm{csupp}(\pi)=\left\{ \rho_1, \ldots, \rho_k \right\}$ (as a multiset), called the {\it cuspidal support} of $\pi$. 

 For each segment $\Delta$, we denote by $\langle \Delta \rangle$ the Zelevinsky segment representation \cite[Section 3.1]{Ze80}. Denote by $\mathrm{St}(\Delta)$  the generalized Steinberg representation \cite[Section 9.1]{Ze80}. In particular, $\mathrm{csupp}(\langle \Delta \rangle)=\mathrm{csupp}(\mathrm{St}(\Delta))=\Delta$. 

For a multisegment $\mathfrak m=\left\{ \Delta_1, \ldots, \Delta_r\right\}$, we label the segments such that $\Delta_1 \not\leq \ldots \not\leq \Delta_r$. Let $\zeta(\mathfrak m)=\langle \Delta_1\rangle \times \ldots \times \langle \Delta_r \rangle$. Let $\lambda(\mathfrak m)=\mathrm{St}(\Delta_1)\times \ldots \times \mathrm{St}(\Delta_r)$, which we shall refer to a {\it standard module} (in the sense of Langlands). Let $\langle \mathfrak m \rangle$ be the unique submodule of $\zeta(\mathfrak m)$ \cite[Section 6.5]{Ze80} and let $\mathrm{St}(\mathfrak m)$ be the unique simple quotient of $\lambda(\mathfrak m)$. 

\subsection{Generalized GGP relevance} \label{ss gGGP relevance}

In this section, we define a notion of generalized GGP relevant pairs from a representation-theoretic viewpoint. For the practical purpose, the reader may first adapt Definition \ref{def integral derivative} and then go to Definition \ref{def combinatorial comm triple} for the notion of strongly commutative triples (also see Theorem \ref{thm combinatorial def}). We first define a notion of derivatives and integrals:

\begin{definition} \label{def integral derivative}
Let $\pi \in \mathrm{Irr}$. 
\begin{itemize}
\item[(1)] Denote by $I_{\Delta}^R(\pi)$ (resp. $I_{\Delta}^L(\pi)$) to be the unique submodule of $\pi \times \mathrm{St}(\Delta)$ (resp. $\mathrm{St}(\Delta)\times \pi$). We shall call $I_{\Delta}^L$ and $I^R_{\Delta}$ to be {\it integrals}.
\item[(2)] Suppose $l_a(\Delta)\leq n$. Let $N$ (resp. $N'$) be the unipotent subgroup of the standard parabolic subgroup corresponding to the partition $(n-l_a(\Delta), l_a(\Delta))$ (resp. ($l_a(\Delta), n-l_a(\Delta))$). Let $D^R_{\Delta}(\pi)$ (resp. $D^L_{\Delta}(\pi)$) be the unique irreducible representation (if exists) such that 
\[  D^R_{\Delta}(\pi) \boxtimes \mathrm{St}(\Delta) \hookrightarrow \pi_N, \quad (\mbox{resp. } \mathrm{St}(\Delta)\boxtimes D^L_{\Delta}(\pi) \hookrightarrow \pi_{N'} ).
\]
If such module does not exist, set $D^R_{\Delta}(\pi)=0$ (resp. $D^L_{\Delta}(\pi)=0$). We shall call those $D^R_{\Delta}$ and $D^L_{\Delta}$ to be {\it derivatives}.
\end{itemize}
\end{definition}

For the above uniqueness, see \cite{KKKO15, LM16} and also see \cite{Ch22+, Ch22+b}.

We define a notion of strongly commutative triples in terms of how a certain structure of Jacquet module is embedded in the layers of the geometric lemma of the induced module.

\begin{definition} \label{def strong commu intro} (c.f. \cite{Ch22+d})
Let $\Delta, \Delta'$ be two segments. Let $\pi \in \mathrm{Irr}$ and let $\tau=I^L_{\Delta'}(\pi)$. We say that $(\Delta, \Delta', \pi)$ is a {\it pre-RdLi-commutative triple} if $D^R_{\Delta}(\pi)\neq 0$ and the following composition 
\begin{align} \label{eqn pre commutative triple}
 D^R_{\Delta}(\tau)\boxtimes \mathrm{St}(\Delta) \hookrightarrow \tau_N \hookrightarrow (\mathrm{St}(\Delta')\times \pi)_N \twoheadrightarrow \mathrm{St}(\Delta') \dot{\times}^1 \pi_{N'} ,
\end{align}
where $N$ and $N'$ are corresponding unipotent radical, and the last term is the top layer in the geometric lemma (see, Section \ref{ss geometric lemma for pre commutativity}, for precise notations). Here the first map comes from the unique map in Definition \ref{def integral derivative}(2) and the second map is induced from the embedding in Definition \ref{def integral derivative}(1). The last map is the natural projection from the geometric lemma.

We say that $(\Delta, \Delta', \pi)$ is a {\it strongly RdLi-commutative triple} if the triple is pre-RdLi-commutative and moreover, the map in (\ref{eqn pre commutative triple}) factors through the embedding:
\[ (\mathrm{St}(\Delta')\times D^R_{\Delta_1}(\pi))\boxtimes \mathrm{St}(\Delta) \hookrightarrow \mathrm{St}(\Delta') \dot{\times}^1 \pi_{N'} ,
\]
which is induced from the unique map $D^R_{\Delta}(\pi)\boxtimes \mathrm{St}(\Delta)\hookrightarrow \pi_{N'}$.
\end{definition}




In \cite{Ch22+d}, we formulate and show other equivalent definitions for the strong commutation in terms of $\eta$-invariants (See  Definitions \ref{def combinatorial comm triple}, \ref{def dual com commutative} and Theorem \ref{thm combinatorial def}). Such triples can then be checked combinatorially in terms of Langlands parameters (as well as Zelevinsky parameters).

For a multisegment $\mathfrak m$, we label the segments in $\mathfrak m$ as: $\Delta_i\not> \Delta_j$  (resp. $\Delta_i \not<\Delta_j$) for any $i<j$. For such an ordering, we shall call it an {\it ascending order} (resp. {\it descending order}) (c.f. \cite[Theorem 6.1]{Ze80}). Define 
\[  D^R_{\mathfrak m}(\pi)=D^R_{\Delta_r}\circ \ldots \circ D^R_{\Delta_1}(\pi), \quad (\mathrm{resp.}\ D^L_{\mathfrak n}(\pi)=D^L_{\Delta_r}\circ \ldots \circ D^L_{\Delta_1}(\pi)) ,
\]
\[  I^L_{\mathfrak m}(\pi)=I^L_{\Delta_r}\circ \ldots \circ I^L_{\Delta_1}(\pi), \quad (\mathrm{resp. }\ I^R_{\mathfrak n}(\pi)=I^R_{\Delta_r}\circ \ldots \circ I^R_{\Delta_1}(\pi)) .
\]

We now define the notion of strongly RdLi-commutative triples for multisegments:

\begin{definition} \label{def commutative triple}
Let $\mathfrak m, \mathfrak n$ be multisegments. We write $\mathfrak m=\left\{ \Delta_1, \ldots, \Delta_r \right\}$ and $\mathfrak n=\left\{ \Delta_1', \ldots, \Delta_s'\right\}$ in an ascending order. Let $\mathfrak m_i=\left\{ \Delta_1, \ldots, \Delta_i \right\}$ and let $\mathfrak n_j=\left\{ \Delta_1', \ldots, \Delta_s' \right\}$. We say that $(\mathfrak m, \mathfrak n, \pi)$ is a {\it strongly RdLi-commutative triple} if  for any $r \geq i\geq 1$ and $s \geq j \geq 1$, 
\[  (\Delta_{i}, \Delta_{j}', I_{\mathfrak n_{j-1}}\circ D_{\mathfrak m_{i-1}}(\pi))
\]
is a strongly RdLi-commutative triple. (When $\mathfrak m$ or $\mathfrak n$ is empty, the triple $(\mathfrak m, \mathfrak n, \pi)$ is automatically relevant.)
\end{definition}

It follows from Propositions \ref{prop unlinked commute integral} and \ref{prop unlinked commute derivative} that the above definition is independent of a choice of orderings.

\begin{definition}(c.f. Definition \ref{def relevant pair content}) (Generalized GGP relevant pair) \label{def g relevant pair}
Let $\pi \in \mathrm{Irr}(G_n)$ and let $\pi' \in \mathrm{Irr}(G_m)$. We say that $(\pi, \pi')$ is (generalized) {\it relevant} if there exist multisegments $\mathfrak m$ and $\mathfrak n$ such that 
\begin{enumerate}
\item $D^R_{\mathfrak m}(\nu^{1/2}\cdot \pi) \cong D^L_{\mathfrak n}(\pi')$;
\item $(\mathfrak m, \mathfrak n, \nu^{1/2}\cdot\pi)$ is a strongly RdLi-commutative triple.
\end{enumerate}
\end{definition}

The original GGP relevant pair \cite{GGP20} (for Arthur type representations and all classical groups) is defined in terms of Langlands parameters, and reformulated by Zhiwei Yun more geometrically in terms of moment maps \cite[Section 4]{GGP12}. Our main result, Theorem \ref{thm main branching} below, shows in a rather indirect way that the original GGP relevance coincides with Definition  \ref{def g relevant pair} for Arthur type representations. 

The first main property is the symmetry of such notion:

\begin{theorem} \label{thm symmetry of left right branching} (=Theorem \ref{thm symmetric property of relevant})
Let $\pi \in \mathrm{Irr}(G_n)$ and let $\pi' \in \mathrm{Irr}(G_m)$. Then $(\pi, \pi')$ is relevant if and only if $(\pi', \pi)$ is relevant.
\end{theorem}

Theorem \ref{thm symmetry of left right branching} can be viewed as a compatibility with the left-right BZ filtrations (see discussions in Section \ref{ss technique bz filtrations}). 

To state the second property, we need more notations. For two multisegments $\mathfrak m$ and $\mathfrak m'$, we write $\mathfrak m'\leq_Z \mathfrak m$ if either $\mathfrak m'$ is obtained from $\mathfrak m$ by a sequence of intersection-union operations (see Section \ref{ss intersect union}) or $\mathfrak m'=\mathfrak m$. 

\begin{definition}
Let $\pi \in \mathrm{Irr}$. A multisegment $\mathfrak m$ is said to be {\it Rd-minimal} (resp. {\it Li-minimal}) to $\pi$ if $D^R_{\mathfrak m}(\pi)\neq 0$ and $D^R_{\mathfrak m}(\pi) \not\cong D^R_{\mathfrak n}(\pi)$ (resp. $I^L_{\mathfrak m}(\pi)\not\cong I^L_{\mathfrak n}(\pi)$) for any other multisegment $\mathfrak n$ with $\mathfrak n\lneq_Z \mathfrak m$. We have analogous notions for {\it Ld-minimal} and {\it Ri-minimal}.
\end{definition}

One has uniqueness for such minimal multisegments \cite{Ch22+}, see Theorem \ref{thm minimality of d and i}. We have the following uniqueness statement.


\begin{theorem} (=Theorem \ref{thm unique of relevant pairs}) \label{thm uniquess relevant intro}
Let $(\pi, \pi')$ be a relevant pair in Definition \ref{def g relevant pair}. There exist unique multisegments $\mathfrak m$ and $\mathfrak n$ such that $\mathfrak m$ is $Rd$-minimal to $\nu^{1/2}\cdot \pi$ and $\mathfrak n$ is $Ld$-minimal to $\pi'$, and the conditions in Definition \ref{def g relevant pair} are satisfied.
\end{theorem}

The notion of a strongly commutative triple has good properties with respect to the intersection-union process, see Propositions \ref{prop intersect union comm} and \ref{prop two strong commute dual 2}. This roughly gives the existence part in Theorem \ref{thm uniquess relevant intro}.

The uniqueness of Theorem \ref{thm uniquess relevant intro} is related to the layers of right BZ filtration that governs the branching law. More precisely, the uniqueness allows one to define a refined terminology: $i^*$-relevant, where $i^*$ is equal to $l_a(\mathfrak m)$.

\section{The Bernstein-Zelevinsky filtration} \label{s bz filtration}

The Bernstein-Zelevinsky filtration (a.k.a. Mackey theory approach) is our primary tool for branching laws. Such approach has recently been found in much success for dealing with branching laws for non-tempered cases \cite{GGP20, CS21, Ch21, Gu22, Ch22, Ch22+, Ch22+d} (also see \cite{Pr93, MW12, Ve13, BP16, SV17, Pr18}). R. Chen and C. Wang \cite{CW} deal with some non-tempered cases for unitary groups along such line. For other approaches such as using theta correspondence and relative trace formulae, one sees \cite{Gr21+}. 

\subsection{Bernstein-Zelevinsky functors} \label{ss bz functors}

Let $V_n$ be the unipotent radical of $M_n$. Let $\overline{\psi}$ be a non-degenerate character on $F$. Let $\psi: V_n \rightarrow \mathbb C$ given by $\psi(v)=\overline{\psi}(v_{n-1})$, where $v_{n-1}$ is the $(n-1,n)$-entry in $v$. For $\pi \in \mathrm{Alg}(G_n)$, let
\[  \pi_{V_n, \psi}= \delta^{-1/2} \pi/\langle n.v-\psi(n)v :n \in V_n, v \in \pi \rangle \]
where $\delta$ is a modular character of $M_n$.

Following \cite{BZ76, BZ77},
\[ \Phi^+: \mathrm{Alg}(M_{n-1})\rightarrow \mathrm{Alg}(M_n), \Phi^-: \mathrm{Alg}(M_n)\rightarrow \mathrm{Alg}(M_{n-1}),\] 
\[ \Psi^+: \mathrm{Alg}(G_{n-1})\rightarrow \mathrm{Alg}(M_n), \Psi^-: \mathrm{Alg}(M_n) \rightarrow \mathrm{Alg}(G_{n-1}) \]
 given by:
\[  \Phi^+(\pi)=\mathrm{Ind}_{M_{n-1}V_n}^{M_n} \pi \boxtimes \psi, \quad \Phi^-(\pi)=\pi_{V_n, \psi} ,
\]
and let 
\[  \Psi^+(\pi)=\mathrm{Ind}_{G_{n-1}V_n}^{M_n} \pi \boxtimes 1, \quad \Psi^-(\pi)=\pi_{V_n,1} .
\]
The $i$-th right Bernstein-Zelevinsky derivative of $\pi$ is defined as:
\[  \pi^{(i)}= \Psi^-\circ(\Phi^-)^{i-1}(\pi) .
\]
Let $\theta=\theta_n: G_n\rightarrow G_n$ given by $\theta(g^{-t})$, the inverse transpose (a.k.a Gelfand-Kazhdan involution). It induces a categorical auto-equivalence on $\mathrm{Alg}(G_n)$, still denoted by $\theta$. The left $i$-the BZ derivatives is defined as:
\[ {}^{(i)}\pi =\theta(\theta(\pi)^{(i)}). 
\]

Their {\it shifted derivatives} are defined as:
\[  \pi^{[i]}=\nu^{1/2}\cdot \pi^{(i)}, \quad {}^{[i]}\pi=\nu^{-1/2}\cdot {}^{(i)}\pi .
\]

Let $\pi \in \mathrm{Irr}$. The {\it level} of $\pi$, denoted $\mathrm{lev}(\pi)$, is the largest integer $i^*$ such that $\pi^{(i^*)}\neq 0$. Let $\mathfrak m\in \mathrm{Mult}$ such that $\pi \cong \langle \mathfrak m \rangle$. The following statements are equivalent, due to \cite{Ze80}:
\begin{enumerate}
\item $i^*$ is the level of $\pi$;
\item $i^*$ is the largest integer such that ${}^{(i^*)}\pi\neq 0$;
\item $i^*$ is equal to the number of segments in $\mathfrak m$. 
\end{enumerate}
We shall call $\pi^{(i^*)}$ to be the {\it highest right derivative}, denoted by $\pi^-$, and call ${}^{(i^*)}\pi$ to be the {\it highest left derivative}, denoted by ${}^-\pi$. Similarly, set $\pi^{[-]}=\pi^{[i^*]}$ and ${}^{[-]}\pi={}^{[i^*]}\pi$.

\subsection{Bernstein-Zelevinsky filtrations} \label{ss bz filtration describe}

For $\tau \in \mathrm{Alg}(G_k)$ and $i \geq 1$, define, as $M_{k+i}$-representations,
\[  \Gamma^i(\tau)=(\Phi^+)^{i-1}\circ \Psi^+(\tau) .
\]

For $\pi \in \mathrm{Alg}(G_{n+1})$, define
\[  \Lambda_i(\pi)= (\Phi^+)^i\circ (\Phi^-)^i(\pi) , \quad \Sigma_i(\pi)=\Gamma^i(\pi^{(i)}) .
\]
Bernstein-Zelevinsky \cite{BZ76, BZ77} shows that there are natural inclusions in $\mathrm{Alg}(M_{n+1})$:
\[  \Lambda_{n+1}(\pi):=0\hookrightarrow \Lambda_n(\pi)\hookrightarrow \ldots \Lambda_1(\pi) \hookrightarrow \Lambda_0(\pi)=\pi .
\]
Furthermore,
\[   \Lambda_{i-1}(\pi)/\Lambda_{i}(\pi) \cong \Sigma_i(\pi)  .
\]
We shall refer it to be the right BZ filtration of $\pi$. We shall frequently use the following for computations: for $\pi$ and $\pi'$ to be admissible, and for any $j$,
\[  \mathrm{Ext}^j_{G_n}(\Sigma_i(\pi), \pi')  \cong \mathrm{Ext}^j_{G_{n+i-i}}(\pi^{[i]}, {}^{(i-1)}\pi') ,
\]
see e.g. \cite[Lemma 2.4]{CS21}.

Note that $\theta$ also induces a categorical equivalence from $\mathrm{Alg}(M_{n+1})$ to $\mathrm{Alg}(M_{n+1}^{t})$, again denoted by $\theta$. For $\tau \in \mathrm{Alg}(G_k)$,
\[  {}^i\Gamma(\tau):=  \theta(\Gamma^i(\theta(\tau))) .
\]
For $\pi \in \mathrm{Alg}(G_{n+1})$, define
\[  {}_i\Lambda(\pi):=\theta(\Lambda_i(\theta(\pi))), \quad {}_i\Sigma(\pi):=\theta(\Sigma_i(\theta(\pi))) .
\]
We shall refer 
\[  {}_{n+1}\Lambda(\pi)\hookrightarrow {}_n\Lambda(\pi)=0\hookrightarrow \ldots \hookrightarrow {}_0\Lambda(\pi)=\pi 
\]
to be the left BZ filtration. For more discussions on left-right BZ filtrations, see \cite{CS21}.

When restricted to $G_n$ (via the embedding $g \mapsto \mathrm{diag}(g,1)$), the left and right BZ filtrations give two filtrations, as $G_n$-representations, on $\pi$.


\subsection{Variants of Bernstein-Zelevinsky filtrations} \label{ss bz filtrations variant}

For $\pi \in \mathrm{Alg}(G_n)$, we denote by $\pi \otimes \zeta_F \in \mathrm{Alg}(G_n)$ its Fourier-Jacobi model of $\pi$, and $\mathrm{RS}_k(\pi) \in \mathrm{Alg}(G_{n+k+1})$ its Rankin-Selberg model. We shall not recall (and need) its explicit definition (see \cite{Ch22, Ch21+}), and the main property we need is the following:

\begin{lemma} \label{lem multiplicity one of standard module}
Let $\lambda$ be a quotient of a standard module of $G_n$. Then,
\begin{enumerate}
\item for any quotient $\lambda'$ of a standard module of $G_n$,
\[ \mathrm{dim}~ \mathrm{Hom}_{G_n}(\lambda \otimes \zeta_F, \lambda'{}^{\vee}) \leq 1   ,
\]
\item for any quotient $\lambda'$ of a standard module of $G_{n+k}$,
\[  \mathrm{dim}~\mathrm{Hom}_{G_{n+k}}(\mathrm{RS}_{k-1}(\lambda), \lambda'{}^{\vee}) \leq 1  .
\]
\end{enumerate}
\end{lemma}

\begin{proof}
This follows from the multiplicity one theorem for standard modules in \cite{Ch21} and \cite[Section 5]{Ch22} (also see \cite{GGP12}).
\end{proof}

We recall that we have the following filtration:

\begin{proposition} \label{prop filtration on parabolic} \cite[Proposition 5.13]{Ch22}
Let $\pi_1 \in \mathrm{Alg}(G_{n_1})$ and let $\pi_2 \in \mathrm{Alg}(G_{n_2})$. Then $(\pi_1 \times \pi_2)|_{G_{n_1+n_2-1}}$ admits a filtration with successive subquotients as follows:
\[  \pi_1^{[0]} \times (\pi_2|_{G_{n_2-1}}) , \quad  \pi_1^{[1]} \times (\pi_2 \otimes \zeta^F) \]
and, for $k \geq 2$,
\[   \pi_1^{[k]} \times \mathrm{RS}_{k-2}(\pi_2) .\]
\end{proposition}

We remark that one can also consider the filtration as a $M_{n_1+n_2}$-representation as in \cite{Ch22}, and the filtration in Proposition \ref{prop filtration on parabolic} is then obtained under restriction to $G_{n_1+n_2-1}$. More precisely, the corresponding layers, as a $M_{n_1+n_2}$-filtration, takes the form:
\[  \pi_1 \times (\pi_2|_{M_{n_2}}) , \quad  \pi_1^{(1)} \times ( \pi_2\otimes \zeta^F),\quad  \pi_1^{(k)} \times  ((\Phi^+)^{k-1} (\nu^{-1/2}\cdot \mathrm{ind}_{G_{n_2}}^{M_{n_2+1}}\pi_2)).
\]
Here the product is the one for mirabolic subgroups, see \cite[Section 4.12]{BZ77} and \cite[Section 3.1]{Ch22}. Viewing as $M_{n_1+n_2}$-representations, one can then talk about the layer supporting properties in Definition \ref{def integer determining branching}. One can further take a $M_{n_2+1}$-filtration on $\nu^{-1/2}\cdot \mathrm{ind}_{G_{n_2}}^{M_{n_2+1}}\pi_2$ to give layers of the form $(\Phi^+)^{r-1}\circ \Psi^+(\pi_2^{(r)})$.

\section{Main results and techniques}
\subsection{Main results}
We now state two main results. The first one shows that the generalized GGP relevant pair governs the branching law.

\begin{theorem} \label{thm main branching} (= Part of Theorem \ref{thm sufficiency}+Theorem \ref{thm refine brancing})
Let $\pi \in \mathrm{Irr}(G_{n+1})$ and let $\pi' \in \mathrm{Irr}(G_n)$. Then $(\pi, \pi')$ is relevant if and only if $\mathrm{Hom}_{G_n}(\pi, \pi') \neq 0$. 
\end{theorem}

\begin{theorem} (=Theorem \ref{thm exhaustion thm}) \label{thm exhaustion intro}
Let $\pi \in \mathrm{Irr}(G_n)$. Let $\tau$ be a simple quotient of $\pi^{(i)}$ for some $i$. Then there exists a multisegment $\mathfrak m$ such that 
\[   D^R_{\mathfrak m}(\pi) \cong \tau .
\]
\end{theorem}

 We give some remarks.
\begin{itemize}
\item Indeed, we prove a refinement for Theorem \ref{thm main branching}, which also determines the BZ layer contributing the branching law from $i$-relevance.
\item We shall refer Theorem \ref{thm exhaustion intro} to the exhaustion theorem. The construction part is largely studied in \cite{Ch22+}.
\item The $n$ and $m$ are arbitrary in Definition \ref{def g relevant pair}. We shall explain briefly the corresponding problems without introducing new terminologies. When $n-m$ is positive and odd (resp. non-negative and even), the related restriction problem arises from so-called Bessel models (resp. Fourier-Jacobi models) \cite{GGP12}. When $m-n$ is negative, the related problem can be phrased as an induction (by taking the smooth duals and Frobenius reciprocity in the restriction problem of the corresponding positive case).

In \cite{GGP12}, it is shown that the quotient branching law for Bessel models and Fourier-Jacobi models, follows from the corank one case. Some variants of those models, see Chen-Sun \cite{CSu15} and Y. Liu \cite[Section 1.4]{Li14}, are also now established in \cite[Section 5]{Ch22} by relating various models via BZ theory.
\item The Langlands correspondence for GL has now been established by Laumon-Rapoport-Stuhler \cite{LRS93}, Harris-Taylor \cite{HT01},  Henniart \cite{He00} and Scholze \cite{Sc13} (also a recent geometric construction of Fargues-Scholze \cite{FS21}). The quotient branching law can be checked using the combinatorial formulation for strongly RdLi-commutative triple (Theorem \ref{thm combinatorial def}).

More precisely, the cuspidal support between $\pi \in \mathrm{Irr}(G_{n+1})$ and $\pi'\in \mathrm{Irr}(G_n)$ reduces to check finitely many $\mathfrak m$ and $\mathfrak n$ for the relevance. Thus, our Theorem \ref{thm main branching} provides a finite deterministic algorithm to the quotient branching law, and see \cite{CP25+} for some practical algorithms in computing derivatives and integrals. The proofs give hints on how to find a more effective algorithm for checking the relevance, and we plan to deal with that in the sequel.
\end{itemize}

\subsection{Examples}

We compare some known examples with Theorem \ref{thm main branching}. 

\begin{enumerate}
\item (Example from distinguished representations \cite{Pr93,Ve13}) Let $\pi'$ be the trivial representation of $G_n$. Let $\pi=\langle [-(n-2)/2, n/2], [(n+2)/2]\rangle$. Let $\Delta=[(n+1)/2, (n+3)/2]$ and let $\Delta'=[-(n-1)/2]$. Then, it follows from a simple cuspidal condition that $(\Delta_1, \Delta_2, \pi)$ is a strongly RdLi-commutative triple. This gives the relevance for $(\pi, \pi')$. On the other hand, it is shown in \cite[THEOREM 2]{Pr93} and \cite[Corollary 6.15]{Ve13} that $\pi$ is $G_n$-distinguished. 
\item (Example from relatively projective representations \cite{BZ77, CS21, Ch21}) Let $\pi$ be the Steinberg representation of $G_{n+1}$. Assume $n \geq 2$. Let $\pi'=\mathrm{St}([-(n-5)/2,(n+1)/2]+[-(n-3)/2])$. Let 
\[ \Delta_1=[-(n-1)/2,-(n-3)/2], \quad \Delta_2=[-(n-3)/2]. \]
We shall use the notations in Definition \ref{def comb commute}. In such case, we have: 
\[  \eta_{\Delta_1}(\pi)=(1,0), \quad \eta_{\Delta_1}(I_{\Delta_2}(\pi))=(1,1) .
\]
Hence, $\eta_{\Delta_1}(\pi)\neq \eta_{\Delta_1}(I_{\Delta_2}(\pi))$. Using the equivalent condition in Definition \ref{def combinatorial comm triple} and checking other possible strong commutations, one has that $(\pi, \pi')$ is not relevant. On the other hand, it is shown in \cite[Theorem 3.3]{CS21} that $\mathrm{Hom}_{G_n}(\pi, \pi')=0$. 
\item (Example from non-tempered GGP \cite{GGP20, Gu22, Ch22}) Let $\pi =\langle [0] \rangle \times \langle [0] \rangle \times \langle[-1,1] \rangle$ and let $\pi'=\langle [-1/2,1/2] \rangle \times \mathrm{St}([-1/2,1/2])$. Let $\mathfrak m=\left\{[1/2], [1/2],[3/2] \right\}$ and let $\mathfrak n=\left\{ [-1/2,1/2] \right\}$. One checks that $\mathfrak m$ and $\mathfrak n$ define the relevance for $\pi$ and $D^R_{\mathfrak m}(\nu^{1/2}\pi)\cong D^L_{\mathfrak n}(\pi')$. On the other hand, from \cite{GGP20} and \cite{Ch22}, we have $\mathrm{Hom}_{G_n}(\pi, \pi')\neq 0$.
\end{enumerate}

\subsection{Techniques in branching laws and derivatives} \label{ss technique bz filtrations}

We highlight that the fifth one for standard modules below is established in \cite{Ch21+}, and the seventh and last ones below are mainly developed in \cite{Ch22+}, and the sixth, eighth and ninth ones below seem to be new in studying branching law. We consider that the sixth one provides more fundamental structure while in view of Appendix B, the eighth one provides structural information refining the asymmetry property in \cite{Ch21}.

\begin{enumerate}
\item {\bf Bernstein-Zelevinsky theory.} 
The use of BZ filtration in branching laws goes back to the earlier work of D. Prasad \cite{Pr93} and Flicker \cite{Fl93} for studying distinguished representations.
\item {\bf Left-Right BZ filtrations.} 
It turns out that the left and right BZ filtrations in Section \ref{ss bz filtration describe} give complementary information in most of cases. Substantial uses include branching law for the Steinberg representation \cite{CS21} and for indecomposability \cite{Ch21}. 
\item {\bf Rankin-Selberg integrals for constructing branching laws.} One can realize irreducible generic representations $\pi$ and $\pi'$ of $G_{n+1}$ and $G_n$ via their Whittaker models. Then one can roughly define $G_n$-invariant linear functionals on $\pi\otimes \pi'{}^{\vee}$ via the Rankin-Selberg integrals on their Whittaker functions.  A detailed argument appears in \cite{Pr93}. We shall use some variations from Cogdell--Piatetski-Shapiro \cite{CPS17} to construct branching laws in Section \ref{s construct branching via rs}.
\item {\bf Gan-Gross-Prasad type reduction and duality.} 
The GGP type reduction relates restriction problems of different coranks. With similar ideas, one has: for any representation $\pi$ of $G_{n+1}$ and $\pi'$ of $G_n$
\begin{align} \label{eqn duality for restriction}
  \mathrm{Hom}_{G_n}(\pi, \pi'{}^{\vee}) \cong \mathrm{Hom}_{G_{n+1}}(\pi' \times \sigma, \pi^{\vee}) 
\end{align}
for some cuspidal representation $\sigma$ of $G_2$. Switching the dual combined with left-right filtration above is useful in the proofs, see for example the proof of Theorem \ref{thm sufficiency} and Proposition \ref{prop refinement conjecture}.
\item {\bf Multiplicity-one theorems for irreducible and standard representations.} 
The multiplicity-one theorem for irreducible and standard representations \cite{AGRS10, Ch21+} provides a uniqueness result and thus one can combine with constructing branching laws (e.g. Rankin-Selberg integral method above) to identify the required maps. For example, this is used in Lemma \ref{lem construct branching from deformation}. 

We also remark that the multiplicity one for irreducible ones for fields of positive characteristics is also established in the work of Aizenbud-Avni-Gourevitch \cite{AAG12} and Mezer \cite{Me21}.
\item {\bf Characterizing the layer supporting a branching law from BZ derivatives.} Another application of the multiplicity-one theorem for standard representations (combined with some homological properties of standard representations, see Lemma \ref{lem vanishing ext filtration}) is to give a characterization in Corollary \ref{cor integer for branching law equal layer}. Such characterization combined with the computation for smallest derivatives in the relevance pair (Section \ref{s smallest derivative strongly}) automatically gives the refined branching law, once we know the original branching law from relevance.  In the proof of the sufficiency statement, the refined branching law is crucial in singling out a quotient that contributes to the branching law in Lemma \ref{lem a branching factor through quotient}, and such quotient satisfies certain multiplicity one property allowing to use the strategy in Lemma \ref{lem general strategy}. In the proof of the necessity statement, the refined branching law is crucial in showing certain admissibility for some multisegments to get a relevance. Another use is to give a refinement on the standard trick in Lemma \ref{lem construct branching from deformation} while it could possibly be avoided.
\item {\bf Minimal sequences in constructing simple quotients of BZ derivatives.} As suggested by the definition of generalized relevant pairs and Theorem \ref{thm uniquess relevant intro}, minimal sequences of derivatives are crucial in the study. Remarkable properties of the minimal sequences   include \cite{Ch22+}: Let $\mathfrak m$ be Rd-minimal to $\pi \in \mathrm{Irr}$. Then
\begin{itemize}
  \item (Subsequent property) for any submultisegment $\mathfrak n$ of $\mathfrak m$, $\mathfrak n$ is also Rd-minimal to $\pi$;
	\item (Commutativity) for any  submultisegment $\mathfrak n$ of $\mathfrak m$, $\mathfrak m-\mathfrak n$ is also Rd-minimal to $D^R_{\mathfrak n}(\pi)$.
\end{itemize}
It turns out that both properties behave well in the context of strong commutation, see Corollaries \ref{cor seq strong commut under minimal and sub} and \ref{cor seq strong commut under minimal}.
\item {\bf Deforming branching laws.}  Let $\pi \in \mathrm{Irr}(G_{n+1})$ and let $\tau$ be a simple quotient of $\pi^{(i)}$ for some $i$. Then $\tau \times \sigma$ is a quotient of $\pi$ for some good choice of a cuspidal representation $\sigma \in \mathrm{Irr}(G_{i-1})$. When $\pi$ is thickened (Definition \ref{def thickened multisegments}), one can carry out certain highest derivatives to construct branching laws, see more discussions in the beginning of Section \ref{s taking highest derivatives in BL}.
\item {\bf Deforming simple quotients of BZ derivatives.}  For an irreducible representation $\pi$, the Zelevinsky theory \cite{Ze80} implies that there exists an irreducible thickened representation such that ${}^-\pi' \cong \pi$ and $\mathrm{lev}(\pi')=\mathrm{lev}(\pi)$. It is shown in Theorem \ref{thm deform segments} that there is a natural one-to-one correspondence between the set of simple quotients of $\pi'{}^{(i)}$ and the set of simple quotients of $\pi^{(i)}$ via taking the highest left derivatives. 

The two ideas (8) and (9) of deformations are useful since thickened cases can be reduced to by simple combinatorics. They are also closely related.
\item {\bf Double derivatives and integrals.} Let $\pi \in \mathrm{Irr}$ and let $\mathfrak m \in \mathrm{Mult}$ with $D^R_{\mathfrak m}(\pi)\neq 0$. Then there exists a multisegment $\mathfrak m'$ such that $D^R_{\mathfrak m'}\circ D^R_{\mathfrak m}(\pi) \cong \pi^-$ \cite{Ch22+}. The integral analogue is established in Theorem \ref{thm double integral}. It is an important tool in the proof of Proposition \ref{prop exhaustion thm from branching}.
\end{enumerate}

\subsection{Standard trick} \label{ss standard trick intro}

We are going to outline the proofs which are inductive in nature. We first explain some more notations and a standard trick below in doing induction. We shall call it a standard trick since idea of this sort has been used in e.g. \cite{GGP12, Ch22}. 

Let $\Delta=[a,b]_{\rho}$ be a segment. A segment $\Delta'$ is said to be {\it R-$\Delta$-saturated} if $\Delta'$ takes the form $[a',b]_{\rho}$ for some $a\leq a'\leq b$ and a multisegment $\mathfrak p$ is said to be {\it R-$\Delta$-saturated} if any segment in $\mathfrak p$ is $R$-$\Delta$-saturated. A representation $\pi$ is said to be {\it R-$\Delta$-reduced} if $D^R_{\Delta'}(\pi)=0$ (i.e. $\varepsilon_{\Delta'}(\pi)=0$) for any R-$\Delta$-saturated segment $\Delta'$. These notions indicate a reduction technique. One also has analogous left version of those notions, and we only remark that we use the segments of the form $[a,b']_{\rho}$ (for $a\leq b'\leq b$) in the left version.

For a cuspidal representation $\rho$, we say that $\pi$ is {\it strongly R-$\rho$-reduced} (resp. {\it strongly L-$\rho$-reduced}) if for any $\Delta$ with $b(\Delta)\cong \rho$ (resp. $a(\Delta)\cong \rho$), $D^R_{\Delta}(\pi)=0$ (resp. $D^L_{\Delta}(\pi)=0$). A multisegment $\mathfrak m$ is said to be {\it strongly R-$\rho$-saturated} (resp. {\it strongly L-$\rho$-saturated}) if any segment $\Delta$ in $\mathfrak m$ satisfies $b(\Delta)\cong \rho$ (resp. $a(\Delta)\cong \rho$).

For $\pi \in \mathrm{Irr}$, define 
\[  \mathrm{csupp}_{\mathbb Z}(\pi)=\left\{\nu^c\cdot \rho : \rho \in \mathrm{csupp}(\pi), \quad c \in \mathbb Z \right\}. 
\]
An irreducible cuspidal representation $\sigma$ is said to be {\it good} to an irreducible representation $\pi$ if $\sigma$ is not in $\mathrm{csupp}_{\mathbb Z}(\pi)$. In particular, we have $\sigma \times \pi$ is irreducible for such $\sigma$ \cite{Ze80}.

\begin{lemma} \label{lem standard trick first form}
Let $\Delta=[a,b]_{\rho}$ be a segment. Let $\pi \in \mathrm{Irr}(G_{n+1-k})$ and let $\pi' \in \mathrm{Irr}(G_n)$. Let $\mathfrak p$ be a R-$\Delta$-saturated multisegment. Let $k=l_a(\mathfrak p)$. Suppose $\nu^b\rho$ is good to $\nu^{-1/2}\pi'$. Let $\sigma \in \mathrm{Irr}^c(G_{k})$ good to $\pi$ and $\nu^{-1/2}\pi'$. We also assume that $\nu^b\rho$ is $\leq$-maximal in $\mathrm{csupp}(\mathfrak p)+\mathrm{csupp}(\pi)$. Then 
\[  \mathrm{Hom}_{G_n}(\mathrm{St}(\mathfrak p)\times \pi, \pi') \cong \mathbb C
\] 
if and only if  
\[  \mathrm{Hom}_{G_n}(\sigma \times \pi, \pi') \cong \mathbb C .
\]
\end{lemma}

We shall prove an alternative and stronger form in Lemma \ref{lem construct one more non-zero element}. We remark that the assumption $\nu^b\rho$ is $\leq$-maximal in $\mathrm{csupp}(\mathfrak p)+\mathrm{csupp}(\pi)$ is not essential in the above lemma, but it is more convenient to consider the stronger form since we then have a standard module projecting to $\mathrm{St}(\mathfrak p)\times \pi$.

\section{Outline of proofs}
\subsection{Outline of the proof of sufficiency part} \label{ss outline sufficient}

We start with $\pi \in \mathrm{Irr}(G_{n+1})$ and $\pi' \in \mathrm{Irr}(G_n)$ such that $(\pi, \pi')$ is relevant. We mainly explain steps on showing the (non-refined) branching law, and as mentioned before, the refined one follows from Proposition \ref{prop smallest integer in strong triple} and Corollary \ref{cor integer for branching law equal layer}.  

 Let $\rho'$ be a $\leq$-minimal element in $\mathrm{csupp}(\pi')$ (see Section \ref{ss basic notations} for the ordering $\leq$). We shall illustrate the idea by only considering the case that $\rho' \notin \mathrm{cuspp}(\nu^{1/2}\cdot \pi)$. Then we find a multisegment $\mathfrak p'$ such that $D^L_{\mathfrak p'}(\pi')$ is strongly L-$\rho'$-reduced. 

Now the compatibility of commutativity of minimal sequences and strongly commutative sequences gives that $(\pi, D^L_{\mathfrak p'}(\pi'))$ is still relevant (see Corollary \ref{cor seq strong commut under minimal} and Lemma \ref{lem delta reduced subset}), and so for a suitable choice of cuspidal representation $\sigma$ of $G_{l_a(\mathfrak p')}$, one can use induction to obtain
\[  \mathrm{Hom}_{G_n}(\pi, \sigma \times D^L_{\mathfrak p'}(\pi')) \cong \mathbb C .
\]
Now, let $\tau'=\mathrm{St}(\mathfrak p') \times D^L_{\mathfrak p'}(\pi')$ (so that $\pi'$ is a submodule of $\tau'$) and the standard trick gives that:
\begin{align} \label{eqn non-zero construct inductive intro}
  \mathrm{Hom}_{G_n}(\pi, \tau') \cong \mathbb C .
\end{align}
(Note that the strong form of the standard trick in Lemma \ref{lem construct one more non-zero element} also determines the layer in the BZ filtration contributing the branching law, say the one from $i^*$-th one. Thus, this gives a non-zero map in 
\[  \mathrm{Hom}_{G_n}(\pi^{[i^*]}, {}^{(i^*-1)}\tau') .
\]
Then uniqueness of the Hom (see Corollary \ref{cor integer determining branching law and layers}) forces such map factors through the submodule map from ${}^{(i^*-1)}\pi'$ to ${}^{(i^*-1)}\tau'$. This gives a strong evidence for $\mathrm{Hom}_{G_n}(\pi, \pi')\neq 0$. One may also compare with some arguments of this type in the necessity part below.)

To illustrate how to conclude the proof, we first consider a simpler situation. Let $\rho$ be a $\leq$-maximal element in $\mathrm{csupp}(\pi)$. Suppose $\rho$ is good to $\nu^{-1/2}\cdot \pi'$. (This situation may not always happen and we shall explain how to resolve this next.) Under this assumption, there exists a multisegment $\mathfrak p$ such that $D^R_{\mathfrak p}(\pi)$ is strongly R-$\rho$-reduced (for the existence of such multisegment, see \cite[Section 7]{LM22} and \cite[Section 2]{Ch22+b}).

Thus, let $\tau=\mathrm{St}(\mathfrak p)\times D^R_{\mathfrak p}(\pi)$ (so that $\pi$ is a quotient of $\tau$) and similar reasoning as before with the standard trick gives that:
\[  \mathrm{Hom}_{G_n}( \tau, \pi') \cong \mathbb C 
\]
A variation of the standard trick also gives that
\[  \mathrm{Hom}_{G_n}(\tau, \tau') \cong \mathbb C .\]
Now $\mathrm{Hom}_{G_n}(\pi, \pi') \cong \mathbb C$ follows from some constraints in module theory, see Lemma \ref{lem general strategy}.

Now the technical issue is to deal with when such $\rho$ does not exist. To do so, we introduce more notations. Let $\mathfrak m, \mathfrak n \in \mathrm{Mult}$ such that:
\begin{itemize}
\item $D^R_{\mathfrak m}(\nu^{1/2}\cdot \pi) \cong D^L_{\mathfrak n}(\pi')$;
\item $(\mathfrak m, \mathfrak n, \nu^{1/2}\cdot \pi)$ is a minimal strongly RdLi-commutative triple (in the sense of Theorem \ref{thm uniquess relevant intro}).
\end{itemize}
Let $\rho$ be a $\leq$-maximal element in $\mathrm{csupp}(\mathfrak m)$. Then find a longest segment $\Delta$ in $\mathfrak m_{b=\rho}$ (see Section \ref{ss notations a b points} for unexplained notation). Let $\mathfrak p \in \mathrm{Mult}$ such that $D_{\mathfrak p}^R(\pi)$ is R-$\Delta$-reduced. 

A main problem is that it is unclear that
\[  \mathrm{Hom}_{G_n}(\mathrm{St}(\mathfrak p)\times D_{\mathfrak p}^R(\pi), \pi') 
\]
has dimension $1$. We circumvent this issue by considering certain quotients of $\mathrm{St}(\mathfrak p)\times D_{\mathfrak p}^R(\pi)$ and $\pi$. 

In order to do so, let $\omega =D_{\mathfrak p}^R(\pi)$ and we consider a filtration on $\mathrm{St}(\mathfrak p)\times \omega$ via a variant of Bernstein-Zelevinsky filtration in Proposition \ref{prop filtration on parabolic} of the form:
\begin{align} \label{eqn bz layers 1}
   \mathrm{St}(\mathfrak p)^{[0]} \times (\omega|_{G_{n-l_a(\mathfrak p)}}) , \quad    \mathrm{St}(\mathfrak p)^{[1]} \times (\omega \otimes \zeta_F) , \\
 \label{eqn bz layers 3}  \mathrm{St}(\mathfrak p)^{[k]} \times \mathrm{RS}_{k-2}(\omega), \quad \mbox{ for $k \geq 2$ }. 
	\end{align}
 Let $k^*=l_a(\mathfrak p)-l_a(\mathfrak{mx}(\pi', \nu^{1/2} \Delta))$ (see Section \ref{ss multisegment for eta} for unexplained notions). We let $\tau$ be the quotient of $(\mathrm{St}(\mathfrak p)\times \omega)|_{G_n}$ modulo out the submodule containing those factors
\[  \kappa:=\mathrm{St}(\mathfrak p)^{[k]} \times \mathrm{RS}_{k-2}(\omega), \quad k> k^*\geq 1  .
\]
and let
\[  \widetilde{\tau}=\tau/ \kappa .
\]
Let 
\[   \widetilde{\pi}= \mathrm{pr}(\tau)/\mathrm{pr}(\kappa) ,
\]
where $\mathrm{pr}$ is the quotient map from $\tau$ to $\pi$.  Let $\tau'=\mathrm{St}(\mathfrak p) \times \omega$.
It turns out that 
\begin{align} \label{eqn three part to multiplicity one}
   \mathrm{Hom}_{G_n}(\widetilde{\tau}, \pi') \cong \mathbb C ,\quad  \mathrm{Hom}_{G_n}(\widetilde{\pi}, \tau') \cong \mathbb C , \quad \mathrm{Hom}_{G_n}(\widetilde{\tau}, \tau')\cong \mathbb C. 
\end{align}
Then one applies Lemma \ref{lem general strategy} to obtain that $\mathrm{Hom}_{G_n}(\widetilde{\pi}, \pi') \cong \mathbb C$, and so $\mathrm{Hom}_{G_n}(\pi, \pi')\cong \mathbb C$. 

We finally comment on (\ref{eqn three part to multiplicity one}). The uniqueness part for $\mathrm{Hom}_{G_n}(\widetilde{\tau}, \pi')$ and $\mathrm{Hom}_{G_n}(\widetilde{\tau}, \tau')$ follow from analysing the layers in (\ref{eqn bz layers 1}) and (\ref{eqn bz layers 3}) and seeing that the only possible layer contributing the non-zero Hom is 
\[ \mathrm{St}(\mathfrak p)^{[k^*]} \times \mathrm{RS}_{k^*-2}(\omega)
\]
Then one applies Frobenius reciprocity and induction. (The uniqueness part for $\mathrm{Hom}_{G_n}(\widetilde{\pi}, \tau')$ follows from (\ref{eqn non-zero construct inductive intro}).)

For the existence part of $\mathrm{Hom}_{G_n}(\widetilde{\pi}, \tau')$, one first considers the non-zero map $f$ (by (\ref{eqn non-zero construct inductive intro})) in 
\[  \quad \mathrm{Hom}_{G_n}(\pi, \tau')
\]
and then shows that $f|_{\mathrm{pr}(\kappa)}=0$, which one analyses the possible layer contributing the branching law in $\mathrm{pr}(\kappa)$ and $\kappa$ (using the technique of smallest derivatives). The details of computation is done in Section \ref{s smallest derivative strongly} (particularly Corollary \ref{cor strict inequality for integer db}). For the existence part of $\mathrm{Hom}_{G_n}(\widetilde{\tau}, \pi')$, the reasoning is similar, but one uses a Rankin-Selberg construction (with control on the layer supporting the refined branching law, see Lemma \ref{lem construct branching from deformation}) to replace (\ref{eqn non-zero construct inductive intro}). (The existence for $\mathrm{Hom}_{G_n}(\widetilde{\tau}, \tau')$ follows from any of the above two cases.)

\subsection{Outline of the proof of necessity part} \label{ss outline necc}

We start with  $\pi \in \mathrm{Irr}(G_{n+1})$ and $\pi' \in \mathrm{Irr}(G_n)$ with $\mathrm{Hom}_{G_n}(\pi, \pi')\neq 0$. Again, we mainly outline how to show the relevance and the refined relevance will follow from Proposition \ref{prop smallest integer in strong triple} and Corollary \ref{cor integer for branching law equal layer}. 

Let $\rho$ be a $\leq$-maximal element in $\mathrm{csupp}(\pi)$. We shall only deal with the case that $\rho \notin \mathrm{csupp}(\nu^{-1/2}\pi')$. The remaining case can be dealt with using the technique of left-right BZ filtrations and the symmetry of relevant pairs in Theorem \ref{thm symmetry of left right branching}. This technique has been used in \cite{CS21, Ch21, Ch22, Ch21+} before and so we also refer the reader to those articles.

One can again find a strongly R-$\rho$-reduced multisegment $\mathfrak p$ such that 
\[  \mathrm{St}(\mathfrak p)\times D^R_{\mathfrak p}(\pi) \twoheadrightarrow \pi  .
\]
This gives 
\[  \mathrm{Hom}_{G_n}(\mathrm{St}(\mathfrak p)\times D^R_{\mathfrak p}(\pi), \pi') \neq 0 .
\]
Thus, the standard trick in Lemma \ref{lem standard trick first form} gives that 
\[  \mathrm{Hom}_{G_n}(\sigma \times D^R_{\mathfrak p}(\pi), \pi') \neq 0
\]
for some suitable choice of a cuspidal representation $\sigma$. This inductively gives the relevance for $(D_{\mathfrak p}^R(\pi), \pi')$. In other words, there exist multisegments $\mathfrak m$ and $\mathfrak n$ such that $(\mathfrak m, \mathfrak n, \nu^{1/2}\cdot D_{\mathfrak p}^R(\pi))$ is a minimal strongly RdLi-commutative triple and 
\[  D_{\mathfrak m}^R(\nu^{1/2}\cdot D_{\mathfrak p}^R(\pi)) \cong D_{\mathfrak n}^L(\pi') .
\]

Let $\mathfrak p'=\nu^{1/2}\mathfrak p$. Now showing the relevance of $(\pi, \pi')$ can be accomplished by the following steps:
\begin{enumerate}
\item[(1)] First, show that $\mathfrak m+\mathfrak p'$ is admissible to $\nu^{1/2}\cdot\pi$ and
\begin{align} \label{eqn isom strong comm nec part}
  D^R_{\mathfrak m+\mathfrak p'}(\nu^{1/2}\cdot \pi)\cong D^L_{\mathfrak n}(\pi') .
\end{align}
 The main idea is that the multiplicity-one theorem for standard modules (see Theorem \ref{cor quotient of standard mult one} for a precise form) forces that
\[ \mathrm{dim}~\mathrm{Hom}_{G_n}((\mathrm{St}(\mathfrak p)\times D^R_{\mathfrak p}( \pi))^{[i]},  {}^{(i-1)}\pi') \leq 1
\]
and
\[  \mathrm{dim}~ \mathrm{Hom}_{G_n}(\pi^{[i]}, {}^{(i-1)}\pi') \leq 1 ,
\]
where $i=l_a(\mathfrak m+\mathfrak p')$. Indeed, both inequalities are equalities. The first one follows from the induction and a stronger form of the standard trick (see Lemma \ref{lem construct one more non-zero element}). The second one follows from some control in multiplicity one theorems (see Corollary \ref{cor integer determining branching law and layers}). 

The multiplicity-one results above then give the simple quotient $D^R_{\mathfrak m}(\nu^{1/2} \cdot D^R_{\mathfrak p}( \pi))$ in
\[  (\mathrm{St}(\mathfrak p)\times D^R_{\mathfrak p}( \pi))^{[i]}
\] 
must come from $\pi^{[i]}$. (We remark that one has to use $i$-relevance to show the appearance of the simple quotient.) The exhaustion theorem then forces the admissibility (see Lemma \ref{lem admissible exhaust form}). It turns out that $\mathfrak m+\mathfrak p'$ is also Rd-minimal to $\nu^{1/2}\cdot \pi$ from the theory of minimal sequences and hence (\ref{eqn isom strong comm nec part}) follows from the commutativity of a minimal sequence. 
\item[(2)] It remains to show that $(\mathfrak m+\mathfrak p', \mathfrak n, \nu^{1/2}\cdot\pi)$ is a strongly RdLi-commutative triple. This part again relies on the commutativity of minimal sequences of derivatives and the compatibility with strong commutation in Corollary \ref{cor commute minimal strong rdli comm}.
\end{enumerate}

\subsection{Outline of the proof of exhaustion part} \label{ss outline exh} 

We now discuss the exhaustion theorem of simple quotients of BZ derivatives (Theorem \ref{thm exhaustion intro}). Note that our proof of Theorem \ref{thm main branching} also relies on Theorem \ref{thm exhaustion intro}. We shall also use Theorem \ref{thm main branching} in our proof of Theorem \ref{thm exhaustion intro}, but only in a smaller rank case. We shall refer the reader to Section \ref{ss proof of sufficiency exh} for a more precise argument for the interactions of these two theorems.

Let $\tau$ be a simple quotient of $\pi^{[i]}$ for some $i$. Let $\sigma' \in \mathrm{Irr}^c(G_{n-i-1})$ be good to $\nu^{1/2}\cdot \pi$. Then, a standard computation using right BZ filtration gives that 
\[  \mathrm{Hom}_{G_n}(\pi, \sigma' \times \tau )\neq 0 . \]
One can now switch to the left BZ filtration, and an intriguing point here is that the layer in the BZ filtration contributing to the branching law comes from the bottom layer one, that is Theorem \ref{thm simple quotient branching law}. (We remark that in order to prove Theorem \ref{thm simple quotient branching law}, one first has to consider a thickened case in Definition \ref{def thickened multisegments}, which can be done with some combinatorial arguments. Then one deduces the non-thickened case by the technique of deformation above.)

Let $\widetilde{\pi}=\nu^{-1/2}\cdot \pi$. The upshot of using the left BZ filtration is that one is then in the position for applying the standard trick. Now let $\mathfrak h$ be the highest left derivative multisegment of $\pi$, which we mean $D^L_{\mathfrak h}(\widetilde{\pi})={}^-\widetilde{\pi}$ (also see Definition \ref{def highest der multi}). Then Theorem \ref{thm main branching} (which becomes available in the induction process after the standard trick) with Theorem \ref{thm symmetry of left right branching}  and some commutativity for minimal sequences (with using a duality of Proposition \ref{prop dual multi strong comm} switching to a LdRi-version) gives that 
\[    D_{\mathfrak h}^L\circ I^R_{\mathfrak n}(\widetilde{\pi}) \cong I^R_{\mathfrak n}\circ D_{\mathfrak h}^L(\widetilde{\pi}) \cong  \tau .
\]
The remaining step is to apply double integrals (Theorem \ref{thm double integral}), but for that we first have to use deformation of derivatives to show that $\mathrm{lev}(I^R_{\mathfrak n}(\pi))=\mathrm{lev}(\pi)$, that is Corollary \ref{cor unique highest derivative to give submodule}.

Then, using double integrals, one finds a multisegment $\mathfrak n'$  such that 
\[ D_{\mathfrak h}^L\circ I^R_{\mathfrak n'}\circ I^R_{\mathfrak n}(\widetilde{\pi}) \cong {}^-(I^R_{\mathfrak n'}\circ I^R_{\mathfrak n}(\widetilde{\pi}))\cong \nu\cdot \widetilde{\pi}\]
 (the first isomorphism uses the $\mathrm{lev}(I^R_{\mathfrak n'}\circ I^R_{\mathfrak n}(\widetilde{\pi}))=\mathrm{lev}(I^R_{\mathfrak n}(\widetilde{\pi}))=\mathrm{lev}(\widetilde{\pi})$). With the commutativity from level-preserving integrals (Proposition \ref{level preserving strong comm}), we have
\[ I^R_{\mathfrak n'}\circ D^L_{\mathfrak h}\circ I^R_{\mathfrak n}(\widetilde{\pi})\cong D_{\mathfrak h}^L\circ I^R_{\mathfrak n'}\circ I^R_{\mathfrak n}(\widetilde{\pi}) .\] 
The RHS is isomorphic to $ \nu\cdot \widetilde{\pi}$ as discussed above. Applying $D_{\mathfrak n'}^R$ to give
\[  D_{\mathfrak n'}^R(\nu^{1/2}\cdot\pi) \cong  \tau .
\]

\section{Summary of some key results and their relations}
We shall abbreviate BL for branching law and also abbreviate the following results:
\begin{itemize}
\item (StCo): Properties for (sequences of) strongly commutative triples (Sections \ref{s unlinked commute triples}-\ref{s strong commutation for sequences})
\item (BZd): Characterize BZ layers determining branching law in terms of derivatives (Corollary \ref{cor integer for branching law equal layer})
\item (UqR): Uniqueness of relevant pairs (Theorem \ref{thm unique of relevant pairs})
\item (SmD): The smallest derivative for relevant pairs (Proposition \ref{prop smallest integer in strong triple})
\item (RS): Rankin-Selberg construction (Lemma \ref{lem construct branching from deformation})
\item (DdDi): Double derivatives and double integrals (Theorems \ref{thm double derivative} and \ref{thm double integral})
\item (SyR): Symmetry of relevant pairs (Theorem \ref{thm symmetric property of relevant}) + duality (Proposition \ref{prop dual multi strong comm})
\item (Su): Sufficiency of relevant pairs (Theorem \ref{thm sufficiency})
\end{itemize}
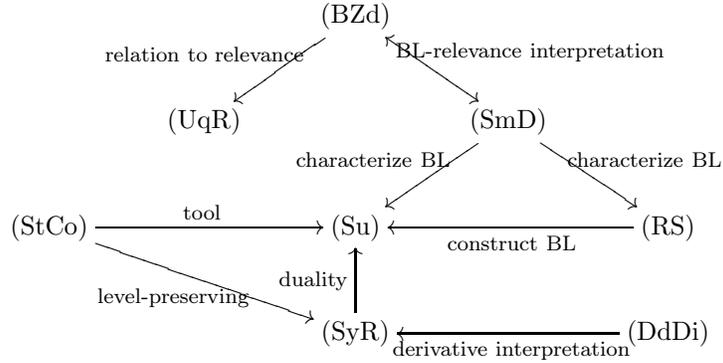
\[
\xymatrix{  &   &\mathrm{(BZd)} \ar[ld]_{\mbox{\footnotesize{relation to relevance}}} \ar@{<->}[rd]^{\mbox{\footnotesize{BL-relevance interpretation}}}  &  & \\
 &  \mathrm{(UqR)}  &               &  \mathrm{(SmD)} \ar[ld]_{\mbox{\footnotesize{characterize BL}}} \ar[rd]^{\mbox{\footnotesize{characterize BL}}} & \\
\mathrm{(StCo)} \ar[rrd]_{\mbox{\footnotesize{level-preserving}}} \ar[rr]^{\mbox{\footnotesize{tool}}}	 &               & \mathrm{(Su)}  &     &  \mathrm{(RS)} \ar[ll]^{\mbox{\footnotesize{construct BL}}} \\
&  & \mathrm{(SyR)} \ar[u]^{\mbox{\footnotesize{duality}}} &  & \ar[ll]^{\mbox{\footnotesize{derivative interpretation}}} \mathrm{(DdDi)} 
}
\]

\begin{itemize}
\item (BDth): Left-right branching law from BZ derivatives in thickened cases (Section \ref{ss bz for derivatives thickened})
\item (BD): Left-right branching law from some BZ derivatives (Theorem \ref{thm simple quotient branching law})
\item (DDer): Deform simple quotients of BZ derivatives by taking highest derivatives (Theorem \ref{thm deform segments})
\item (ExT): Exhaustion theorem for BZ derivatives (Theorem \ref{thm exhaustion thm})
\item (Ne): Necessity of relevant pairs (Theorem \ref{thm refine brancing})
\end{itemize}
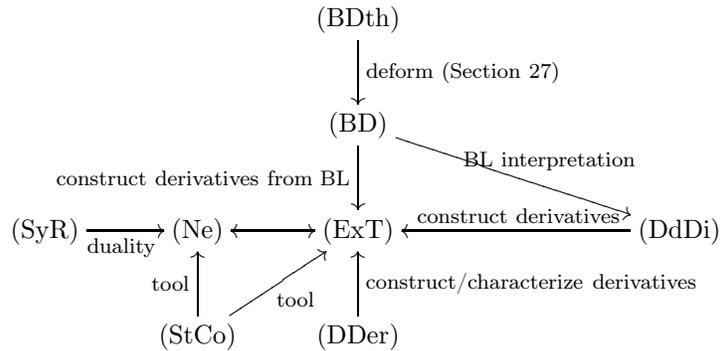
\[  \xymatrix{ &  & \mathrm{(BDth)} \ar[d]^{\mbox{\footnotesize{deform (Section \ref{s taking highest derivatives in BL})}}} &  & & \\
               &  & \mathrm{(BD)} \ar[drrr]^{\mbox{\footnotesize{BL interpretation}}}  \ar[d]_{\mbox{\footnotesize{construct derivatives from BL}}} & &&   \\
\mathrm{(SyR)} \ar[r]_{\mbox{\footnotesize{duality}}} &	\mathrm{(Ne)}\ar@{<->}[r]	 & \mathrm{(ExT)} & & & \mathrm{(DdDi)} \ar[lll]_{\mbox{\footnotesize{construct derivatives}}} \\
	                           &  (\mathrm{StCo}) \ar[u]^{\mbox{\footnotesize{tool}}} \ar[ur]_{\mbox{\footnotesize{tool}}} &      (\mathrm{DDer}) \ar[u]_{\mbox{\footnotesize{construct/characterize derivatives}}}& & &   }
\]

In Appendix B, we also have $(\mathrm{BD})+(\mathrm{BZd}) \Rightarrow$ the asymmetry property of left-right BZ-derivatives in \cite{Ch21}.

We also summarize key techniques in constructing (sequences of) strongly commutative triples:
\begin{itemize}
\item Unlinked pairs (Section \ref{s unlinked commute triples})
\item Intersection-union process (Sections \ref{s linked strong commutation} and \ref{s comm minimal pairs})
\item Commutativity from minimality (Sections \ref{s linked strong commutation} and \ref{s comm minimal pairs})
\item Level-preserving integrals (Proposition \ref{level preserving strong comm})
\item Completing to $\Delta$-reduced triples (Corollary \ref{cor reduced one strong commutation})
\end{itemize}
For example, level-preserving integrals play an important role in proving the symmetry property while the commutativity from minimality is usually combined with induction from the standard trick.

We finally remark that the above relations mainly follow from our logical development. One can also reverse some arrows once the main results are established.

\section{Acknowledgments} 
This article is benefited from communications with Wee Teck Gan, Erez Lapid, Dipendra Prasad and Gordan Savin. The author would like to thank them all. This project is supported in part by the Research Grants Council of the Hong Kong Special Administrative Region, China (Project No: 17305223, 17308324) and the National Natural Science Foundation of China (Project No. 12322120).

\part{Equivalent definitions of strongly commutative triples} \label{part defintions comm triples}

We review some key results in \cite{Ch22+} and \cite{Ch22+d} in this part.

\section{Notations on $\eta$-invariants}

\subsection{Notations} \label{ss notion simplify}

Most of time, we shall usually give/prove statements involving right derivatives and/or left integrals. The analogous version switching between left and right can be formulated/proved similarly and will be left to the reader.

For lightening notation, set $D_{\Delta}=D^R_{\Delta}$ and $I_{\Delta}=I^L_{\Delta}$ for a segment $\Delta$.

\subsection{$\eta$-invariants}

\begin{definition} \label{def comb commute}
Let $\Delta=[a,b]_{\rho}$ be a segment. Let $\pi \in \mathrm{Irr}$.
\begin{itemize}
\item  Define $\varepsilon_{\Delta}(\pi)=\varepsilon_{\Delta}^R(\pi)$ (resp. $\varepsilon_{\Delta}^L(\pi)$) to be the largest non-negative integer $k$ such that $D^k_{\Delta}(\pi)\neq 0$ (resp. $(D^L_{\Delta})^k(\pi)\neq 0$). (The power $k$ means doing compositions $k$-times.)
\item  Define 
\[ \eta_{\Delta}(\pi)=\eta^R_{\Delta}(\pi)=(\varepsilon_{[a,b]_{\rho}}(\pi), \varepsilon_{[a+1,b]_{\rho}}(\pi), \ldots, \varepsilon_{[b,b]_{\rho}}(\pi));\]
\[  (resp. \quad \eta_{\Delta}^L(\pi)=(\varepsilon^L_{[a,b]_{\rho}}(\pi), \varepsilon^L_{[a,b-1]_{\rho}}(\pi), \ldots, \varepsilon^L_{[a,a]_{\rho}}(\pi)) ).
\] 
\item We may simply write $\eta_{\Delta}(\pi)=0$ if $\varepsilon_{[c,b]_{\rho}}(\pi)=0$ for all $c=a, \ldots, b$. Similarly, we may write $\eta_{\Delta}(\pi)\neq 0$ if $\varepsilon_{[c,b]_{\rho}}(\pi)\neq 0$ for some $c=a, \ldots, b$. 
\item We write $\eta_{\Delta}(\pi) ? \ \eta_{\Delta}(\pi')$ for $?\in \left\{ =,  \leq , <, \geq, > \right\}$ if $\varepsilon_{[c,b]_{\rho}}(\pi) ? \varepsilon_{[c,b]_{\rho}}(\pi')$ for all $c$ satisfying $a\leq c\leq b$. 
\item We similarly define the left version terminologies for $\varepsilon^L_{\Delta}$ and $\eta^L_{\Delta}$ if one uses the left derivatives $D^L_{\Delta}$. 
\end{itemize}
\end{definition}

\subsection{Multisegment counterpart of the $\eta$-invariant} \label{ss multisegment for eta}

\begin{definition} \label{def max multisegment}
For $\pi \in \mathrm{Irr}$ and a segment $\Delta=[a,b]_{\rho}$, define 
\[  \mathfrak{mx}(\pi, \Delta):=\mathfrak{mx}^R(\pi, \Delta):=\sum_{k=0}^{b-a}\varepsilon_{[a+k,b]_{\rho}}(\pi)\cdot [a+k,b]_{\rho} ,
\]
\[  \mathfrak{mx}^L(\pi, \Delta):=\sum_{k=0}^{b-a}\varepsilon^L_{[a,b-k]_{\rho}}(\pi)\cdot [a,b-k]_{\rho} .
\]
Here the numbers $\varepsilon_{[a+k,b]_{\rho}}(\pi)$ and $\varepsilon^L_{[a,b-k]_{\rho}}(\pi)$ are the multiplicities of those segments in the corresponding multisegments. 

For $\pi \in \mathrm{Irr}$ and $\rho \in \mathrm{Irr}^c$, similarly define:
\[   \mathfrak{mxpt}(\pi, \rho) :=\mathfrak{mxpt}^R(\pi, \rho):=\sum_{\Delta: b(\Delta)\cong \rho}\varepsilon_{\Delta}(\pi)\cdot \Delta ,\]
\[  \mathfrak{mxpt}^L(\pi, \rho) =\sum_{\Delta: a(\Delta)\cong \rho} \varepsilon^L_{\Delta}(\pi) \cdot \Delta
\]
\end{definition}

We have that $D_{\mathfrak{mx}(\pi, \Delta)}(\pi)$ is R-$\Delta$-reduced (see Section \ref{ss standard trick intro}), and similarly $D_{\mathfrak{mx}^L(\pi, \Delta)}^L(\pi)$ is L-$\Delta$-reduced (i.e. $\eta_{\Delta}^L(\pi)=0$). For the details, one sees \cite[Corollary 7.23]{Ch22+} (also see \cite{Ch22+b, LM22}). Moreover, if $\mathfrak m$ is a $\Delta$-saturated multisegment such that $D_{\mathfrak m}(\pi)\neq 0$, then $\mathfrak m$ is a submultisegment of $\mathfrak{mx}(\pi, \Delta)$.


\subsection{More notations on multisegments} \label{ss notations a b points}

For $\mathfrak m \in \mathrm{Mult}$ and $\rho \in \mathrm{Irr}^c$, let 
\[  \mathfrak m_{a=\rho}=\left\{ \Delta \in \mathfrak m: a(\Delta)\cong \rho \right\}, \quad  \mathfrak m_{b=\rho}=\left\{ \Delta \in \mathfrak m: b(\Delta)\cong \rho \right\}.
\]

 We remark that $\varepsilon_{\Delta}(\pi)$ is equal to the cardinality of the following set (c.f. \cite[Lemma 5.1]{Ch22+}):
 \[ \left\{ \widetilde{\Delta} \in \mathfrak{hd}(\pi)_{a=a(\Delta)}: \Delta \subset \widetilde{\Delta}         \right\} .
\] 

\subsection{Highest derivative multisegments $\mathfrak{hd}$} \label{ss highest derivative multisegment}

Following results \cite[Theorem 7.3, Section 8.1]{Ch22+} and \cite[Theorem 7.1]{Ch22+bb}, we shall use the following quick definition of the highest derivative multisegment.

\begin{definition} \cite{Ch22+, Ch22+bb} \label{def highest der multi}
Let $\pi \in \mathrm{Irr}$. Define the {\it highest right (resp. left) derivative multisegment}, denoted by $\mathfrak{hd}(\pi):=\mathfrak{hd}^R(\pi)$ (resp. $\mathfrak{hd}^L(\pi)$), of $\pi$ to be the unique minimal multisegment such that 
\[   D_{\mathfrak{hd}(\pi)}(\pi) \cong \pi^- , \quad (\mbox{resp. $D_{\mathfrak{hd}^L(\pi)}^L(\pi)\cong {}^-\pi$}) .
\]
\end{definition}

\begin{definition} \label{def removal process} \cite[Section 7]{Ch22+}
Let $\mathfrak h \in \mathrm{Mult}$. Let $\Delta=[a,b]_{\rho}$. The {\it removal process} for $(\Delta, \mathfrak h)$ is an algorithm to carry out the following steps:
\begin{enumerate}
\item Find the shortest segment $[a,b']_{\rho}$ in $\mathfrak h$ with $b' \geq b$. Name the segment to be $\Delta_1=[a_1,b_1]_{\rho}$.
\item We inductively define the shortest segment $\Delta_i=[a_i,b_i]_{\rho}$ (for $i \geq 2$) in $\mathfrak h$ satisfying $a_{i-1}<a_i$ and $b\leq b_i<b_{i-1}$. The process terminates when no more such a segment can be found and we denote $\Delta_2, \ldots, \Delta_r$ to be all those segments.
\item Define new segments as follows:
\begin{itemize}
\item $\Delta^{tr}_i=[a_{i+1}, b_i]_{\rho}$ for $i <r$
\item $\Delta^{tr}_r=[b_{r-1}+1,b_r]_{\rho}$ (possibly empty).
\end{itemize}
\item Define 
\[ \mathfrak r(\Delta, \mathfrak h)=\mathfrak h-\sum_{i=1}^r\Delta_i +\sum_{i=1}^r\Delta_i^{tr} .\]
\end{enumerate}
We shall call $\Delta_1, \ldots, \Delta_r$ to be the {\it removal sequence} for $(\Delta, \mathfrak h)$.
\end{definition}

For a multisegment $\mathfrak m=\left\{ \Delta_1, \ldots, \Delta_r \right\}$ written in an ascending order, define 
\[  \mathfrak r(\mathfrak m, \mathfrak h)=\mathfrak r(\Delta_r, \ldots , \mathfrak r(\Delta_1, \mathfrak h)\ldots ).
\]
It is shown in \cite[Section 7]{Ch22+} that $\mathfrak r(\mathfrak m, \mathfrak h)$ is independent of a choice of such an ordering.

For a segment $\Delta$ and a multisegment $\mathfrak h$, define $\varepsilon_{\Delta}(\mathfrak h)$ to be the number of segments $\widetilde{\Delta}$ in $\mathfrak h_{a=a(\Delta)}$ with $\widetilde{\Delta}\subset \Delta$. 

As shown in \cite{Ch22+}, the removal process is a simpler process in keep-tracking the effect of derivatives. For a multisegment $\mathfrak h$, define
\[  \varepsilon_{[a,b]_{\rho}}(\mathfrak h)=|\left\{ [a,c]_{\rho} \in \mathfrak h: b \leq c \right\} |.
\]
We have:

\begin{theorem} \label{thm change in highest derivative after d} \cite[Theorem 9.3]{Ch22+}
Let $\pi \in \mathrm{Irr}$ and let $\mathfrak h=\mathfrak{hd}(\pi)$. Let $\Delta$ be a segment such that $D_{\Delta}(\pi) \neq 0$. Then for any segment $\Delta' \not< \Delta$, $\varepsilon_{\Delta'}(D_{\Delta}(\pi))$ is equal to $\varepsilon_{\Delta'}(\mathfrak r(\Delta, \mathfrak h))$.
\end{theorem}

\subsection{Change of $\eta_{\Delta}$ by a left integral}

The following lemma will be useful later.

\begin{lemma} \label{lem right multi submulti}
Let $\pi \in \mathrm{Irr}$. For any segments $\Delta$ and $\Delta'$, $\mathfrak{mx}(\pi, \Delta)$ is a submultisegment of $\mathfrak{mx}(I_{\Delta'}(\pi), \Delta)$, equivalently
\[   \eta_{\Delta}(\pi) \leq \eta_{\Delta}(I_{\Delta'}(\pi)).
\]
\end{lemma}

\begin{proof}
Let $\mathfrak m=\mathfrak{mx}(\pi,\Delta)$ and let $\sigma=\mathrm{St}(\Delta')$. We have the embedding:
\[   \pi \hookrightarrow D_{\mathfrak m}(\pi) \times \mathrm{St}(\mathfrak m) .
\]
Then 
\[  I_{\sigma}(\pi) \hookrightarrow \sigma \times \pi \hookrightarrow \sigma \times D_{\mathfrak m}(\pi) \times \mathrm{St}(\mathfrak m) .
\]
Thus $D_{\mathfrak m}(I_{\sigma}(\pi))\neq 0$. Now, one applies Theorem \ref{thm change in highest derivative after d} and \cite[Proposition 5.5]{Ch22+} to see that $\mathfrak m \subset \mathfrak{mx}(I_{\sigma}(\pi), \Delta)$.
\end{proof}

\subsection{Some properties on $|\eta|_{\Delta}$} 



\begin{definition} (c.f. \cite[Section 7]{LM22})
For $\pi \in \mathrm{Irr}$ and a segment $\Delta=[a,b]_{\rho}$, define
\[   |\eta|_{\Delta}(\pi)=\varepsilon_{[a,b]_{\rho}}(\pi)+\varepsilon_{[a+1,b]_{\rho}}(\pi)+\ldots+\varepsilon_{[b,b]_{\rho}}(\pi) . \]
\end{definition}

We shall use the following properties in Section \ref{thm unique of relevant pairs}.

\begin{proposition} \label{prop eta function derivative}
Let $\pi \in \mathrm{Irr}$ and let $\Delta=[a,b]_{\rho}$ be a segment. Let $\Delta'=[c,d]_{\rho} $ be a segment with $D_{\Delta'}(\pi)\neq 0$. Then the following statements hold:
\begin{enumerate}
\item For $c<a$ and $d=b$, $\eta_{\Delta}(\pi)=\eta_{\Delta}(D_{\Delta'}(\pi))$.
\item For $ b\geq c\geq a$ and $d=b$, $|\eta|_{\Delta}(\pi)=|\eta|_{\Delta}(D_{\Delta'}(\pi))-1$.
\item For $b>d \geq c\geq a$, $|\eta|_{\Delta}(\pi)=|\eta|_{\Delta}(D_{\Delta'}(\pi))$.
\item For $b>d$ and $a >c$, $\eta_{\Delta}(D_{\Delta'}(\pi))\geq \eta_{\Delta}(\pi)$.
\end{enumerate}
\end{proposition}

\begin{proof}
We shall give a proof using removal sequences. A possible alternative is to do analysis in the geometric lemma.


For (2), let $\Delta_1, \ldots, \Delta_r$ be the removal sequence for $(\Delta', \pi)$, and let $\Delta_1^{tr}, \ldots, \Delta_r^{tr}$ be the truncated segments. Note that, by Definition \ref{def removal process}, those $\Delta_1, \ldots, \Delta_r$ contribute to $|\eta|_{\Delta}(\pi)$. Only $\Delta_1^{tr}, \ldots, \Delta_{r-1}^{tr}$, but not $\Delta_r^{tr}$ contribute to $|\eta|_{\Delta}(\pi)$. Now Theorem \ref{thm change in highest derivative after d} concludes that $|\eta|_{\Delta}(\pi)$ and $|\eta|_{\Delta}(D_{\Delta'}(\pi))$ differs by $1$. 

For (1), it is similar to the one for (2). Use the terminologies above. Suppose $i$ is the smallest integer such that $\Delta_i$ contributes $\eta_{\Delta}(\pi)$. We have that $i>1$ from the condition of (1). We only have $\Delta_i, \ldots, \Delta_r$ (among those $\Delta_k$'s) contributing $\eta_{\Delta}(\pi)$ and only have $\Delta_{i-1}^{tr}, \ldots, \Delta_{r-1}^{tr}$ (among those $\Delta_k^{tr}$'s)  contributing $\eta_{\Delta}(D_{\Delta'}(\pi))$. From the removal process in Definition \ref{def removal process}(3), one also has that $\Delta_{k-1}^{tr}$ and $\Delta_k$ contribute to the same $\varepsilon_{\widetilde{\Delta}}$ for some $\Delta$-saturated segment $\widetilde{\Delta}$. Now Theorem \ref{thm change in highest derivative after d} concludes (1). 

(3) is similar to (2) and (1). Using notations in (2), suppose  $\Delta_1, \ldots, \Delta_j$ (possibly $i=0$) are those all, among those $\Delta_k$'s, contributing to $|\eta|_{\Delta}(\pi)$. Then, $\Delta_1^{tr}, \ldots, \Delta_j^{tr}$ are all, among those $\Delta_k^{tr}$'s, contributing to $|\eta|_{\Delta}(D_{\Delta'}(\pi))$. Now (3) again follows from Theorem \ref{thm change in highest derivative after d}.

(4) is also similar. Use the notations in (2). By (3), we may assume that $a > c$. Suppose $\Delta_i, \ldots, \Delta_j$ are those all contributing to $\eta_{\Delta}(\pi)$. (If those segments $\Delta_i, \ldots, \Delta_j$ do not exist, then the case is slightly simpler as the number may increase by the contribution from a truncated segment.) Then $\Delta_{i-1}^{tr}, \ldots, \Delta_j^{tr}$ are those all contributing to $\eta_{\Delta}(\pi)$. Again, we have that among those, $\Delta_{k-1}^{tr}$ and $\Delta_k$ give contribution to the same $\varepsilon_{\widetilde{\Delta}}$ for some $\Delta$-saturated segment $\widetilde{\Delta}$; for $\Delta_j^{tr}$, it increases $\varepsilon_{\Delta_j^{tr}}$ in $\eta_{\Delta}$ by $1$. One then uses Theorem \ref{thm change in highest derivative after d} again.
\end{proof}

We give a refinement of Proposition \ref{prop eta function derivative}(3).

\begin{proposition} \label{prop removal sequence for intermediate case}
We use the notations in the previous lemma. Suppose $b>d\geq c\geq a$. Let $\widetilde{\Delta}=[\widetilde{a}, \widetilde{b}]_{\rho}$ be the first segment in the removal sequence for $(\Delta', \mathfrak{hd}(\pi))$. 
\begin{enumerate}
\item If $\widetilde{b}< b$, then $\eta_{\Delta}(D_{\Delta'}(\pi))=\eta_{\Delta}(\pi)$.
\item If $\widetilde{b}\geq b$, then there exists a $\Delta$-saturated segment $\overline{\Delta}=[\overline{a}, b]_{\rho}$ with $\overline{a}>c$ such that $\eta_{\overline{\Delta}}(D_{\Delta'}(\pi))\neq \eta_{\overline{\Delta}}(\pi)$ and $\varepsilon_{\overline{\Delta}}(D_{\Delta'}(\pi))>0$. 
\end{enumerate}
\end{proposition}

\begin{proof}
Let $\Delta_1, \ldots, \Delta_r$ be the removal sequence for $(\Delta', \mathfrak{hd}(\pi))$. Let $\Delta_1^{tr}, \ldots, \Delta_r^{tr}$ be the truncated ones. For (1), by Definition \ref{def removal process} (2) and (3), none of those segments contributes to $\eta_{\Delta}(D_{\Delta'}(\pi))$ and $\eta_{\Delta}(\pi)$. Thus, we have the equality.

We now consider (2). Let $k$ be the largest integer such that $b(\Delta_k)\geq b(\Delta)$. Such integer exists by the assumption. If $k\neq r$, we pick $\overline{\Delta}=[a(\Delta_{k+1}), b(\Delta)]$; and if $k=r$, we pick $\overline{\Delta}=[\nu b(\Delta'),b(\Delta)] (\neq \emptyset)$. Then $\Delta_k^{tr}$ will add extra one to the coordinates $\varepsilon_{\overline{\Delta}}(\pi)$ to get $\varepsilon_{\overline{\Delta}}(D_{\Delta'}(\pi))$ while other coordinates are unchanged. This gives the desired statement. 
\end{proof}

\subsection{A criterion on minimality}

We first have the following minimality result. The following is shown in \cite[Proposition 4.10]{Ch22+}, while we demonstrate a different argument (which however depends on results \cite{Ch22+} as well).

\begin{lemma} \label{lem minimal eta criterai} \cite[Proposition 4.10]{Ch22+bb}
Let $\mathfrak m \in \mathrm{Mult}$ be minimal to $\pi$. Let $\Delta$ be a segment such that 
\[   a(\widetilde{\Delta}) < a(\Delta), \quad b(\widetilde{\Delta}) < b(\Delta) 
\]
for any segment $\widetilde{\Delta} \in \mathfrak m$. We also assume $D_{\Delta}\circ D_{\mathfrak m}(\pi)\neq 0$. Then $\mathfrak m+\Delta$ is minimal to $\pi$ if and only if 
\[  \eta_{\Delta}(D_{\mathfrak m}(\pi))=\eta_{\Delta}(\pi) .
\]
\end{lemma}

\begin{proof}
Note that when $\mathfrak m$ has only one segment. This is proved in \cite[Lemma 14.3 and Proposition 14.4]{Ch22+}.

Let $\mathfrak m=\left\{ \Delta_1, \ldots, \Delta_r \right\}$ in an ascending order. Suppose $\mathfrak m+\Delta$ is minimal to $\pi$. Then, by the subsequent property \cite[Theorem 1.3]{Ch22+bb}, $\Delta_1+\Delta$ is still minimal to $\pi$. Thus, by the two segment case \cite[Proposition 9.5]{Ch22+aa}, $\eta_{\Delta}(D_{\Delta_1}(\pi))=\eta_{\Delta}(\pi)$. We have that $\left\{ \Delta_2, \ldots, \Delta_r\right\}+\Delta$ is minimal to $D_{\Delta_1}(\pi)$ by \cite[Theorem 1.4]{Ch22+bb} (also see Lemma \ref{lem commute and minimal}  below) and so by induction, $\eta_{\Delta}(D_{\mathfrak m}(\pi))=\eta_{\Delta}(D_{\Delta_1}(\pi))$ and so we have the only if direction.

For the if direction, suppose $\mathfrak m+\Delta$ is not minimal. Then, by \cite[Corollary 1.6]{Ch22+}, we can find a consecutive pair such that the intersection-union still gives the same derivatives. Then, we must have one segment to be $\Delta$. Then, by using consecutive pairs, we may relabel the segments such that another segment is $\Delta_r$. Then \cite[Proposition 14.4]{Ch22+} and Proposition \ref{prop eta function derivative}(4) imply that 
\[   \eta_{\Delta}(D_{\Delta_r}\circ D_{\mathfrak m-\Delta_r}(\pi))>\eta_{\Delta}(D_{\mathfrak m-\Delta_r}(\pi)) .
\]
On the other hand, by Proposition \ref{prop eta function derivative}(4):
\[  \eta_{\Delta}(D_{\Delta_{r-1}}\circ \ldots \circ D_{\Delta_1}(\pi))\geq \ldots \geq \eta_{\Delta}(D_{\Delta_1}(\pi)) .
\]
 Thus, combining, we obtain that $\eta_{\Delta}(D_{\mathfrak m}(\pi))>\eta_{\Delta}(\pi)$.
\end{proof}




\section{Strong commutativity}

Since most of results in this section have been shown in \cite{Ch22+d}, some parts will be brief. For generic representations $\sigma, \sigma'$, we can extend the definition of pre-commutativity and strong commutativity for a triple $(\sigma, \sigma', \pi)$ in an obvious manner (see \cite{Ch22+d} for more discussions). In particular, when $\sigma=\mathrm{St}(\Delta)$ and $\sigma'=\mathrm{St}(\Delta')$, their terminologies for $(\sigma, \sigma', \pi)$ coincide with $(\Delta, \Delta', \pi)$.

\subsection{Geometric lemma} \label{ss geometric lemma for pre commutativity}

For a $G_k \times G_l$-representation $\omega$ and a $G_n$-representation $\pi$, inflate the $G_n\times G_k \times G_l$-representation $\pi \boxtimes \omega$ to a $P_{n+k}\times G_l$-representation. Denote the (normalized) parabolically induced module
\[  \mathrm{Ind}_{P_{n,k}\times G_l}^{G_{n+k}\times G_l} (\pi \boxtimes \omega)
\]
by $\pi \dot{\times}^1 \omega$.

Recall that the geometric lemma is shown in \cite{BZ77}, which is important in formulating the pre-commutativity in Definition \ref{def strong commu intro} and \cite{Ch22+d}. For $\pi \in \mathrm{Alg}(G_k)$ and $\pi' \in \mathrm{Alg}(G_l)$ and for $r\leq l$, $P_{k+l}P_{k+l-r,r}$ is the closed set in $G_{k+1}$. This gives rise the top layer $\pi \dot{\times}^1 (\pi'_{N_r})$ in the geometric lemma on $(\pi \times \pi')_{N_r}$. Hence, it gives a natural surjection from $(\pi \times \pi')_{N_r}$ to $\pi \dot{\times}^1 (\pi'_{N_r})$ .

\subsection{Strong commutativity $\Rightarrow$ Commutativity}

It is not hard to prove the following proposition from definitions:

\begin{proposition} \cite[Proposition 6.2]{Ch22+d} \label{prop strong commute imply commute}
Let $\sigma_1, \sigma_2$ be generic representations in $\mathrm{Irr}$ and let $\pi \in \mathrm{Irr}$. Let $(\sigma_1, \sigma_2, \pi)$ be a strongly RdLi-commutative triple. Then $I_{\sigma_2}\circ D_{\sigma_1}(\pi)\cong D_{\sigma_1}\circ I_{\sigma_2}(\pi)$. 
\end{proposition}

\subsection{Examples of pre-commutativity}
We provide simple examples of pre-commutativity, which can be deduced from a simple application of the geometric lemma:

\begin{example} \label{example pre commutative}
\begin{enumerate}
\item Suppose $\Delta_1 \cap \Delta_2 =\emptyset$. Then $(\Delta_1, \Delta_2, \pi)$ is a pre-RdLi-commutative triple (for any $\pi \in \mathrm{Irr}$).
\item Let $\Delta_1=[a_1,b_1]_{\rho}$ and let $\Delta_2=[a_2,b_2]_{\rho}$ be segments. Suppose $a_2 <a_1$ or $b_2<b_1$. Then $(\Delta_1, \Delta_2, \pi)$ is a pre-RdLi-commutative triple (for any $\pi \in \mathrm{Irr}$). 
\end{enumerate}
\end{example}

\subsection{Pre-commutativity $\Rightarrow$ Strong commutativity}
\begin{proposition} \label{prop pre implies max pre} \cite[Lemma 8.4]{Ch22+d}
Let $(\mathrm{St}(\Delta), \mathrm{St}(\Delta'), \pi)$ be a pre-RdLi-commutative triple. Let $\mathfrak p=\mathfrak{mx}(\Delta, \pi)$. Then $(\mathrm{St}(\mathfrak p), \mathrm{St}(\Delta'), \pi)$ is also a pre-RdLi-commutative triple. 
\end{proposition}

\begin{proof}
We only sketch the proof. Now one labels the segments in $\mathfrak p$ such that 
\[  a(\Delta_r) \leq \ldots \leq a(\Delta_1) .
\]
Let $\mathfrak p_k=\left\{ \Delta_1, \ldots, \Delta_k \right\}$, for $k=1, \ldots, r$. One proceeds inductively on $k$. When $k=1$, it is automatic from the given hypothesis. We shall assume that we are not in a case of Example \ref{example pre commutative}.

Suppose we have that $(\mathrm{St}(\mathfrak p_k), \mathrm{St}(\Delta'), \pi)$ is a strongly RdLi-commutative triple. We now proceed in two cases:

\begin{enumerate}
\item[(1)] $a(\Delta_{k+1})> a(\Delta')$. This follows from a simple application of the geometric lemma (or uses Example \ref{example pre commutative}(2)) and the inductive hypothesis. 
\item[(2)] $a(\Delta_{k+1}) \leq a(\Delta')$. Suppose $(\mathrm{St}(\mathfrak p), \mathrm{St}(\Delta_2), \pi)$ is not strongly RdLi-commutative triple. Let $l=l_a(\mathfrak p_k)$. Write $\Delta=[a,b]_{\rho}$ and $\Delta'=[a',b']_{\rho}$. Let 
\[  \overline{\Delta}'=[b+1 ,b']_{\rho}, \quad \underline{\Delta}'=[a',b]_{\rho} .
\]
Let $p=l_a(\mathfrak p_k)$, $n=l_a(\Delta_2)+n(\pi)$, $r=l_a(\overline{\Delta}')$ and $s=l_a(\underline{\Delta}')$. 
Note that only one possible geometric layer can work: 
\[   D_{\mathfrak p_{k+1}}(\pi)\boxtimes \mathrm{St}(\mathfrak p_{k+1}) \hookrightarrow  \mathrm{Ind}_{P_{r,n-p-r} \times P_{s,p-s}}^{G_{n-p}\times G_p} (\mathrm{St}(\overline{\Delta}')\boxtimes \mathrm{St}(\underline{\Delta}') \boxtimes \pi_{N_l})^{\phi},
\]
where $\phi$ is a twist bringing to a $G_r \times G_{n-p-r}\times G_s\times G_{p-s}$-representation. Let $\lambda$ be the rightmost representation. Now the pre-commutativity for $(\mathrm{St}(\mathfrak p_k), \mathrm{St}(\Delta_2), \pi)$ forces that the submodule 
\[ (*)\quad D_{\mathfrak p_{k+1}}(\pi)\boxtimes \mathrm{St}(\underline{\Delta}')\boxtimes D^L_{\underline{\Delta}'}(\mathrm{St}(\mathfrak p_{k+1})) \hookrightarrow \lambda_{N_{r,q-r}},
\]
where $q=l_a(\mathfrak p_{k+1})$, lies in the top layer of the geometric lemma on $\lambda_{N_{r,q-r}}$. Indeed, one has the embedding:
\[   \mathrm{St}(\mathfrak p_{k+1}) \hookrightarrow \mathrm{St}(\underline{\Delta}') \dot{\times}^2 \pi_{N_{p+s}}
\]
to be in the top layer of the geometric lemma from induction. This contradicts to a structure arising from the socle-irreducible property of a big derivative in \cite[Corollary 11.8]{Ch22+b}. See \cite[Lemma 8.4]{Ch22+d} for more details.
\end{enumerate}

\end{proof}

\begin{theorem} \label{thm pre implies strong} \cite[Theorem 9.10]{Ch22+d}
Let $(\Delta_1, \Delta_2, \pi)$ be a pre-RdLi-commutative triple. Then $(\Delta_1, \Delta_2, \pi)$ is a strongly RdLi-commutative triple. 
\end{theorem}

\begin{proof}
We sketch main ideas of a proof that relies on more explicit structure. Let $l=l_a(\Delta_2)$. As in \cite{Ch22+b}, we define the big derivative as a $G_l$-representation:
\[   \mathbb D_{\Delta_2}(\pi) :=\mathrm{Hom}_{G_{l}}( \mathrm{St}(\Delta_2), \pi_{N_{l}}) ,\]
where the $G_{l}$-action is defined via embedding to the second factor in $G_{n(\pi)-l}\times G_l$. The big derivative is equipped with $G_{n(\pi)-l}$-action via embedding to the factor in $G_{n(\pi)-l}\times G_l$.

 By the condition of the pre-RdLi-commutativity, we have that:
\[   D_{\Delta_1}\circ I_{\Delta_2}(\pi)  \hookrightarrow \mathrm{St}(\Delta_2)\times \mathbb D_{\Delta_1}(\pi)
\]
To show strong commutativity, it is equivalent to show that the map factors through the embedding $\mathrm{St}(\Delta_2)\times D_{\Delta_1}(\pi) \hookrightarrow \mathrm{St}(\Delta_2)\times \mathbb D_{\Delta_1}(\pi)$. 

It is shown in \cite[Theorem 12.7]{Ch22+b} that $\mathbb D_{\Delta_1}(\pi)$ satisfies the socle-irreducible property. Thus, it suffices to show that for any composition factor $\tau$ in $\mathbb D_{\Delta_1}(\pi)/D_{\Delta_1}(\pi)$, there is no embedding from $D_{\Delta_1}\circ I_{\Delta_2}(\pi)$ to $\mathrm{St}(\Delta_2)\times \tau$. 

We shall not go through all the details. The main idea is to consider the invariants $\eta_{\Delta_1}$ and $\eta_{\Delta_2}^L$. By the arguments in the proof of Theorem \ref{thm combinatorial def} below (which we only need pre-commutativity rather than strong commutativity, and use Proposition \ref{prop pre implies max pre}), we have that 
\[  \eta_{\Delta_1}(\pi) =\eta_{\Delta_1}(I_{\Delta_2}(\pi)), \quad \eta_{\Delta_2}^L(D_{\Delta_1}\circ I_{\Delta_2}(\pi))=\eta^L_{\Delta_2}(I_{\Delta_2}(\pi)) .
\]
On the other hand, some computations on Jacquet functors give that 
\[  |\eta|_{\Delta_1}(\tau) \leq |\eta|_{\Delta_1}(D_{\Delta_1}(\pi))=|\eta|_{\Delta_1}(\pi)-1, \quad |\eta|_{\Delta_2}^L(\tau)\leq |\eta|_{\Delta_2}^L(\pi)
\]
and $\eta_{\Delta_1}(\tau)\neq \eta_{\Delta_1}(D_{\Delta_1}(\pi))$. Then one uses those data to analyse and show that the embeddings $D_{\Delta_1}\circ I_{\Delta_2}(\pi) \hookrightarrow \mathrm{St}(\Delta_2)\times \tau$ cannot happen.
\end{proof}

\section{Combinatorial commutativity} \label{s combinatorial commutation}


\begin{definition} \label{def combinatorial comm triple}
Let $\Delta_1, \Delta_2$ be segments. Let $\pi \in \mathrm{Irr}$. We say that $(\Delta_1, \Delta_2, \pi)$ is a {\it combinatorially RdLi-commutative triple} if $D_{\Delta_1}(\pi)\neq 0$ and
\[ \eta_{\Delta_1}(I_{\Delta_2}(\pi)) = \eta_{\Delta_1}(\pi) . 
\]
\end{definition}

There is a dual definition:

\begin{definition} \label{def dual com commutative}
Let $\Delta_1, \Delta_2$ be segments. Let $\pi \in \mathrm{Irr}$. We say that $(\Delta_1, \Delta_2, \pi)$ is a {\it dual combinatorially RdLi-commutative triple} if 
\[ \eta_{\Delta_2}^L(D_{\Delta_1}\circ I_{\Delta_2}(\pi))=\eta_{\Delta_2}^L(I_{\Delta_2}(\pi)). \]
\end{definition}





\begin{theorem} \label{thm combinatorial def} \cite[Theorem 10.3]{Ch22+d}
Let $\Delta_1, \Delta_2$ be segments. Let $\pi \in \mathrm{Irr}$. Then the following statements are equivalent:
\begin{enumerate}
\item $(\mathrm{St}(\Delta_1), \mathrm{St}(\Delta_2), \pi)$ is a strongly RdLi-commutative triple;
\item $(\mathrm{St}(\Delta_2), \mathrm{St}(\Delta_1), D_{\Delta_1}\circ I_{\Delta_2}(\pi))$ is a strongly LdRi-commutative triple;
\item $(\Delta_1, \Delta_2, \pi)$ is a combinatorially RdLi-commutative triple;
\item $(\Delta_1, \Delta_2, \pi)$ is a dual combinatorially RdLi-commutative triple.
\end{enumerate}
\end{theorem}

\begin{proof}
We sketch a proof.  Let $\mathfrak p=\mathfrak{mx}(\Delta_1, \pi)$. For (1) $\Rightarrow$ (3), note that Proposition \ref{prop pre implies max pre}  holds for any $\Delta$-saturated $\mathfrak p'$ such that $D_{\mathfrak p'}(I_{\Delta_2}(\pi)) \neq 0$. One can now deduce (3) with Lemma \ref{lem right multi submulti}. An alternative proof is given in \cite[Theorem 8.4]{Ch22+d}.

 For (3) $\Rightarrow$ (1). It is simpler to prove that $(\mathrm{St}(\mathfrak p), \mathrm{St}(\Delta_2), \pi)$ is a strongly RdLi-commutative triple \cite{Ch22+d} by using the property that $D_{\mathfrak p}\circ I_{\Delta_2}(\pi)\boxtimes \mathrm{St}(\mathfrak p)$ is a direct summand in $I_{\Delta_2}(\pi)_{N_l}$, where $l=l_a(\mathfrak p)$. (For instance, one may use \cite[Proposition 4.4]{Ch22+d} and the decomposition $\pi \hookrightarrow D_{\mathfrak p}(\pi)\times \mathrm{St}(\mathfrak p)$.) Then one sees that $(\Delta_1, \Delta_2, \pi)$ is pre-RdLi-commutative and so (1) follows from Proposition \ref{prop pre implies max pre}.  

We now prove (1) $\Rightarrow$ (2). Let $\sigma_1=\mathrm{St}(\mathfrak p)$ and let $\sigma_2=\mathrm{St}(\Delta_2)$. Let $\tau=I_{\Delta_2}\circ D_{\mathfrak p}(\pi) \cong D_{\mathfrak p}\circ I_{\Delta_2}(\pi)$ (the commutativity follows from Proposition \ref{prop pre implies max pre} and the strong commutativity of $(\mathrm{St}(\mathfrak p), \mathrm{St}(\Delta_2), \pi)$). Let $m=n(\pi)$, $l_1=l_a(\sigma_1)$, $l_2=l_a(\sigma_2)$. Let $N=N_{l_2,m,l_1}$. We consider the following commutative diagram:
\[ \xymatrix{                                  &   \tau_{N_{l_2,m}} \boxtimes \sigma_1 \ar[dr]   &       &         \\
\sigma_2 \boxtimes D_{\Delta_2}^L(\tau) \boxtimes \sigma_1 \ar[ur] \ar[dr] &                            &   I_{\Delta_2}(\pi)_N   \ar[r]  & (\tau \times \sigma_1)_N              \\
                                           &   \sigma_2 \boxtimes \pi_{N_{m,l_1}} \ar[ur]  &        &
																			}
\]
(The commutativity of the diagram follows from some socle irreducible property of parabolic inductions, see \cite[Section 2]{Ch22+d}.) Indeed, $\tau \boxtimes \sigma_1$ is a direct summand in $I_{\Delta_2}(\pi)_{N_{l_1}}$ (see \cite[Proposition 7.1]{Ch22+d}) and so gives a non-zero composition as follows:
\[  \tau \boxtimes \sigma_1  \hookrightarrow I_{\Delta_2}(\pi)_{N_{l_1}} \rightarrow (\tau \times \sigma_1)_{N_{l_1}} \twoheadrightarrow \tau \boxtimes \sigma_1 ,
\]
where the last projection comes from the quotient in the geometric lemma. (For more details, see \cite{Ch22+d}.) This gives the layer that the copy $\sigma_2\boxtimes D_{\sigma_2}(\tau)\boxtimes \sigma_1$ lies in $(\tau \times \sigma_1)_N$ from geometric lemma. Now from the commutativity diagram, one must have that $(\sigma_2, \sigma_1, \tau)$ is a pre-LdRi-commutative triple. This implies that $(\Delta_2, \Delta_1, \tau)$ is also pre-LdRi-commutative triple (see e.g. \cite[Proposition 7.2]{Ch22+d}). Then $(\Delta_2, \Delta_1, \tau)$ is also strongly LdRi-commutative triple by Theorem \ref{thm pre implies strong}  as desired.


Proof for (2)$\Rightarrow$(1) is similar to that of (1)$\Rightarrow$(2). A proof for (2)$\Leftrightarrow$(4) is similar to that of (1)$\Leftrightarrow$(3).
\end{proof}

\part{Generalized GGP relevance from sequences of strongly commutative triples}

In this section, we define and study the strong commutativity for multisegments, extending the segment case in Part \ref{part defintions comm triples}. As shown in \cite{Ch22+aa} (see Theorem \ref{thm minimality of d and i}), there is a good theory of minimal multisegments for derivatives and integrals. Sections \ref{s unlinked commute triples} to \ref{s minimal strong commutative triples for seq} show the compatibility of such theory with the strong commutation for multisegments, which will be useful in the later proofs. Another two main results in this part are the uniqueness property (Theorem \ref{thm unique of relevant pairs}) and the symmetry property (Theorem \ref{thm symmetric property of relevant}) for relevance.

\section{Unlinked segments for strongly commutative triples} \label{s unlinked commute triples}

\subsection{Unlinked segments} \label{ss unlinked seg}

We first have the following result, see e.g. \cite[Lemma 4.10]{Ch22+}:
\begin{lemma} \label{lem unlinked comm der and integral basic} 
Let $\Delta_1, \Delta_2$ be unlinked segments. Let $\pi \in \mathrm{Irr}$. Then $D_{\Delta_1}\circ D_{\Delta_2}(\pi)\cong D_{\Delta_2}\circ D_{\Delta_1}(\pi)$ and $I_{\Delta_1}\circ I_{\Delta_2}(\pi) \cong I_{\Delta_2}\circ I_{\Delta_1}(\pi)$. 
\end{lemma}

\begin{proposition} \label{prop unlinked commute integral}
Let $\Delta_1', \Delta_2'$ be unlinked segments. Let $\sigma=\mathrm{St}(\left\{ \Delta_1', \Delta_2' \right\})$. Let $\Delta$ be another segment. Let $\pi \in \mathrm{Irr}$. The following conditions are equivalent:
\begin{enumerate}
 \item $(\Delta, \Delta_1', \pi)$ and $(\Delta, \Delta_2', I_{\Delta_1'}(\pi))$ are strongly RdLi-commutative triples;
 \item $(\Delta, \Delta_2', \pi)$ and $(\Delta, \Delta_1', I_{\Delta_2'}(\pi))$ are strongly RdLi-commutative triples. 
\end{enumerate}
\end{proposition}

\begin{proof}
Suppose (1) holds. By Theorem \ref{thm combinatorial def},
\[   \eta_{\Delta}(\pi) =\eta_{\Delta}(I_{\Delta_1'}(\pi))=\eta_{\Delta}(I_{\Delta_2'}\circ I_{\Delta_1'}(\pi))
\]
and, by Lemma \ref{lem right multi submulti},
\[  \eta_{\Delta}(\pi) \leq \eta_{\Delta}(I_{\Delta_2'}(\pi))\leq \eta_{\Delta}(I_{\Delta_1'}\circ I_{\Delta_2'}(\pi))=\eta_{\Delta}(I_{\Delta_1'}\circ I_{\Delta_2'}(\pi)).
\]
Thus, those inequalities are equalities. This implies (2). The other direction can be proved similarly.  
\end{proof}

\begin{proposition} \label{prop unlinked commute derivative}
Let $\Delta_1, \Delta_2$ be unlinked segments. Let $\Delta'$ be another segment. Let $\pi\in \mathrm{Irr}$. The following conditions are equivalent:
\begin{enumerate}
\item $(\Delta_1, \Delta', \pi)$ and $(\Delta_2, \Delta', D_{\Delta_1}(\pi))$ are strongly RdLi-commutative triples;
\item $(\Delta_2, \Delta', \pi)$ and $(\Delta_1, \Delta', D_{\Delta_2}(\pi))$ are strongly RdLi-commutative triples.
\end{enumerate}
\end{proposition}

\begin{proof}
Let $\tau =I_{\Delta'}\circ D_{\Delta_2}\circ D_{\Delta_1}(\pi)$. Note that if $(\Delta', \Delta_2, \tau)$ is strongly LdRi-commutative, then $D_{\Delta'}\circ I_{\Delta_2}(\tau)=D_{\Delta_1}(\pi)$ by Proposition \ref{prop strong commute imply commute} and so $I_{\Delta_2}(\tau)=I_{\Delta'}\circ D_{\Delta_1}(\pi)$. 

By the duality in Theorem \ref{thm combinatorial def}, the two conditions can be rephrased as:
\begin{enumerate}
\item $(\Delta', \Delta_1, I_{\Delta_2}(\tau))$ and $(\Delta', \Delta_2, \tau)$ are strongly LdRi-commutative triples;
\item $(\Delta', \Delta_2, I_{\Delta_1}(\tau))$ and $(\Delta', \Delta_1, \tau)$ are strongly LdRi commutative triples. 
\end{enumerate}
Now the proposition follows from the LdRi version of the integral one in Proposition \ref{prop unlinked commute integral}.
\end{proof}

\section{Intersection-union operations for integrals in commutative triples} \label{s linked strong commutation}

\subsection{Intersection-union process} \label{ss intersect union}

Let $\mathfrak m_1, \mathfrak m_2 \in \mathrm{Mult}$. We say that $\mathfrak m_2$ is obtained from $\mathfrak m_1$ by an {\it intersection-union process} if there exists a pair of linked segments $\Delta_1, \Delta_2$ such that 
\[  \mathfrak m_2=\mathfrak m_1-\left\{ \Delta_1, \Delta_2 \right\}+\Delta_1\cup \Delta_2+\Delta_1\cap \Delta_2 .
\] 
We simply drop the last term if $\Delta_1\cap \Delta_2=\emptyset$. Recall that $\leq_Z$ on multisegments is defined in Section \ref{ss gGGP relevance}.

\subsection{Strong commutativity under intersection-union process for integrals}
In this section, we shall prove:

\begin{proposition} \label{prop intersect union comm}
Let $\Delta_1, \Delta_2$ be linked segments. Let $\widetilde{\Delta}$ be another segment. Let $\Delta_1'=\Delta_1\cup \Delta_2$ and $\Delta_2'=\Delta_1\cap \Delta_2$. Let $\pi \in \mathrm{Irr}$. Suppose $I_{\Delta_2}\circ I_{\Delta_1}(\pi) \cong I_{\Delta_1'}\circ I_{\Delta_2'}(\pi)$. Then, the following statements are equivalent:
\begin{enumerate}
\item $(\widetilde{\Delta}, \Delta_1, \pi)$ and $(\widetilde{\Delta}, \Delta_2, I_{\Delta_1}(\pi))$ are strongly RdLi-commutative triples;
\item $(\widetilde{\Delta}, \Delta_1', \pi)$ and $(\widetilde{\Delta}, \Delta_2', I_{\Delta_1'}(\pi))$
are strongly RdLi-commutative triples;
\item $(\widetilde{\Delta}, \Delta_2', \pi)$ and $(\widetilde{\Delta}, \Delta_1', I_{\Delta_2'}(\pi))$ are strongly RdLi-commutative triples.
\end{enumerate}
Moreover, if any of the equivalent conditions holds, then $I_{\Delta_2}\circ I_{\Delta_1}\circ D_{\widetilde{\Delta}}(\pi)\cong I_{\Delta_2'}\circ I_{\Delta_1'}\circ D_{\widetilde{\Delta}}(\pi)$. 
\end{proposition}

\begin{proof}

We use the notations in the statement of the proposition. Suppose (1) holds. By Theorem \ref{thm combinatorial def}, 
\[ (*)\quad  \eta_{\widetilde{\Delta}}(\pi)=\eta_{ \widetilde{\Delta}}(I_{\Delta_1}(\pi))=\eta_{ \widetilde{\Delta}}(I_{\Delta_2}\circ I_{\Delta_1}(\pi))=\eta_{ \widetilde{\Delta}}(I_{\Delta_2'}\circ I_{\Delta_1'}(\pi)) .
\]
On the other hand, by Lemma \ref{lem right multi submulti},
\[      \eta_{\widetilde{\Delta}}(\pi) \leq \eta_{\widetilde{\Delta}}(I_{\Delta_1'}(\pi)) \leq \eta_{\widetilde{\Delta}}(I_{\Delta_2'}\circ I_{\Delta_1'}(\pi)) .
\]
Now, (*) forces that the inequalities are actually equations. Hence, we have that $(\widetilde{\Delta}, \Delta_1', \pi)$ and $(\widetilde{\Delta}, \Delta_2', I_{\Delta_2'}(\pi))$ are combinatorially commutative triples, and so are strongly RdLi-commutative triples by Theorem \ref{thm combinatorial def}. This proves (2).

Since $\Delta_1'$ and $\Delta_2'$ are unlinked, we have $I_{\Delta_2'}\circ I_{\Delta_1'}(\pi)\cong I_{\Delta_1'}\circ I_{\Delta_2'}(\pi)$ by Lemma \ref{lem unlinked comm der and integral basic}. Then a similar argument as above will prove (3). 

For (3) $\Rightarrow$ (2), one can argue similarly. Indeed,
\[  \eta_{\widetilde{\Delta}}(\pi) =\eta_{\widetilde{\Delta}}(I_{\Delta_1'}(\pi))=\eta_{\widetilde{\Delta}}(I_{\Delta_2'}\circ I_{\Delta_1'}(\pi)) =\eta_{\widetilde{\Delta}}(I_{\Delta_2}\circ I_{\Delta_1}(\pi)) ,
\]
where the first two equations follow from strong commutation and Theorem \ref{thm combinatorial def}; and the last equation follows from the given assumption.  We also have:
\[  \eta_{\widetilde{\Delta}}(\pi) \leq \eta_{\widetilde{\Delta}}(I_{\Delta_1}(\pi)) \leq \eta_{\widetilde{\Delta}}(I_{\Delta_2}\circ I_{\Delta_1}(\pi)) ,
\]
where the inequalities follow from Lemma \ref{lem right multi submulti}. Thus, we then must have the inequalites to be equalties and so Theorem \ref{thm combinatorial def} implies (2).

Proving (3) $\Rightarrow$ (1) is similar. The last assertion follows by applying $D_{\widetilde{\Delta}}$ on $I_{\Delta_2}\circ I_{\Delta_1}(\pi) \cong I_{\Delta_1'}\circ I_{\Delta_2'}(\pi)$ and then using Proposition \ref{prop strong commute imply commute}.
\end{proof}


\begin{remark}
An analogous statement of Proposition \ref{prop intersect union comm} does not hold in general for derivatives (while it is not so far away, c.f. Lemma \ref{lem commute and intersect union}). For example, consider $\Delta_1=[0,1]$ and $\Delta_2=[2,3]$ and $\widetilde{\Delta}=[1,2]$. Let $\pi=\mathrm{St}([0,3])$. It is clear that $D_{\Delta_2}\circ D_{\Delta_1}(\pi)=D_{[0,3]}(\pi)=\mathbb C$ as $G_0$-representation. It is clear that $([0,3], \widetilde{\Delta}, \pi)$ is a strongly RdLi-commutative triple, but $([0,1], \widetilde{\Delta}, \pi)$ is not even a pre-RdLi-commutative triple. 
\end{remark}

\subsection{Commutativity for derivatives and integrals}

We first have the following preliminary form of commutativity of minimal sequences:

\begin{lemma} \label{lem commute and intersect union} \cite[Theorem 1.4]{Ch22+bb}
Let $\pi \in \mathrm{Irr}$. Let $\Delta_1$ and $\Delta_2$ be linked segments with $\Delta_1< \Delta_2$. If $D_{\Delta_2}\circ D_{\Delta_1}(\pi) \not\cong D_{\Delta_1\cup \Delta_2}\circ D_{\Delta_1\cap \Delta_2}(\pi)$, then 
\[   D_{\Delta_2}\circ D_{\Delta_1}(\pi) \cong D_{\Delta_1}\circ D_{\Delta_2}(\pi) .
\]
\end{lemma}


\begin{remark}
It is not true in general if one switches the condition $\Delta_1<\Delta_2$ to $\Delta_2<\Delta_1$. For example, let 
\[ \pi = \langle \left\{ [0,2], [0,1], [1,2] \right\} \rangle . \]
Let $\Delta_1=[1]$ and let $\Delta_2=[2]$. Note that $D_{[1]}\circ D_{[2]}(\pi)\cong \langle \left\{ [0,1], [0], [1,2] \right\} \rangle$, but $D_{[1,2]}(\pi)\cong D_{[2]}\circ D_{[1]}(\pi)=\langle \left\{ [0,2],[0],[1] \right\} \rangle$.
\end{remark}

\begin{lemma} \label{lem commutative for integrals}
Let $\pi \in \mathrm{Irr}$. Let $\Delta_1'$ and $\Delta_2'$ be linked segments with $\Delta_1' < \Delta_2'$. If $I_{\Delta_2'}\circ I_{\Delta_1'}(\pi) \not\cong I_{\Delta_1'\cup \Delta_2'}\circ I_{\Delta_1'\cap \Delta_2'}(\pi)$, then 
\[   I_{\Delta_2'}\circ I_{\Delta_1'}(\pi) \cong I_{\Delta_1'}\circ I_{\Delta_2'}(\pi). 
\]
\end{lemma}

\begin{proof}
Let $\tau=I_{\Delta_2'}\circ I_{\Delta_1'}(\pi)$. Then $D^L_{\Delta_1'}\circ D^L_{\Delta_2'}(\tau) =\pi$ and $D^L_{\Delta_1'\cup \Delta_2'}\circ D^L_{\Delta_1'\cap \Delta_2'}(\tau) \not\cong D^L_{\Delta_1'}\circ D^L_{\Delta_2'}(\tau)$. Thus, the left version of Lemma \ref{lem commute and intersect union} implies that 
\[ D^L_{\Delta_1'} \circ D^L_{\Delta_2'}(\tau) \cong D^L_{\Delta_2'}\circ D^L_{\Delta_1'}(\tau),
\]
which implies the lemma.
\end{proof}

\begin{proposition} \label{prop commut triples ladder one}

Let $\pi \in \mathrm{Irr}$. Let $\Delta_1'$ and $\Delta_2'$ be linked segments with $\Delta_1'<\Delta_2'$. Let $\Delta$ be another segment. Suppose 
\[ I_{\Delta_2'}\circ I_{\Delta_1'}(\pi) \not\cong  I_{\Delta_1'\cup \Delta_2'}\circ I_{\Delta_1'\cap \Delta_2'}(\pi) .\]
Then $(\Delta, \Delta_1', \pi)$ and $(\Delta, \Delta_2', I_{\Delta_1'}(\pi))$ are strongly RdLi-commutative triples if and only if $(\Delta, \Delta_2', \pi)$ and $(\Delta, \Delta_1', I_{\Delta_2'}(\pi))$ are strongly RdLi-commutative triples. Moreover, if one of the equivalent conditions holds, then
\[     I_{\Delta_2'}\circ I_{\Delta_1'}\circ D_{\Delta}(\pi)\not\cong I_{\Delta_2'\cup \Delta_1'}\circ I_{\Delta_2'\cap \Delta_1'}\circ D_{\Delta}(\pi) .
\]
\end{proposition}

\begin{proof}
We first prove the if direction. The strong commutation implies 
\[ \eta_{\Delta}(\pi)=\eta_{\Delta}(I_{\Delta_1'}(\pi))=\eta_{\Delta}(I_{\Delta_2'}\circ I_{\Delta_1'}(\pi)).\]
 Moreover, Lemma \ref{lem right multi submulti} implies that $\eta_{\Delta}(\pi)\leq \eta_{\Delta}(I_{\Delta_2'}(\pi)) \leq \eta_{\Delta}(I_{\Delta_1'}\circ I_{\Delta_2'}(\pi))$. Combining the equation with Lemma \ref{lem commutative for integrals}, the above inequalities have to be equalities. Now the strong commutations of $(\Delta, \Delta_2', \pi)$ and $(\Delta, \Delta_1', I_{\Delta_2'}(\pi))$ follow from Theorem \ref{thm combinatorial def}. This proves the only if direction. For the if direction, one proceeds with a similar proof by using Lemma \ref{lem commutative for integrals}.

We now prove the last assertion. One approach is to develop those combinatorial invariants for ladder representations, in which one may use the Kret-Lapid description for the Jacquet modules of ladder representations \cite{KL12}. Here we give a proof using module structures directly. Suppose not to derive a contradiction. Let $\tau=I_{\Delta_1'}\circ I_{\Delta_2'}(\pi)(\cong I_{\Delta_2'}\circ I_{\Delta_1'}(\pi))$. Let $l=l_a(\Delta)$. We consider the following diagram:
\[ \xymatrix{ & D_{\Delta}(\tau) \boxtimes \mathrm{St}(\Delta) \ar[r]^{i'} & \tau_{N_l}  \ar[d]^{i} &    \\ 
0    \ar[r] &   (\sigma\times \pi)_{N_l} \ar[r]^{u''} \ar@{->>}[d]^{s_1} &  (\mathrm{St}(\Delta_1')\times \mathrm{St}(\Delta_2') \times \pi)_{N_l} \ar[r]^{t''} \ar@{->>}[d]^{s_2} &  (\sigma' \times \pi)_N \ar[r] \ar@{->>}[d]^{s_3} & 0 \\ 
0    \ar[r] &   \sigma \dot{\times}^1\pi_{N_l}  \ar[r]^{u'} & (\mathrm{St}(\Delta_1')\times \mathrm{St}(\Delta_2'))\dot{\times}^1\pi_{N_l} \ar[r]^{t'} & \sigma'  \dot{\times}^1\pi_{N_l} \ar[r] & 0  \\ 
0    \ar[r] &   (\sigma \times D_{\Delta}(\pi))\boxtimes \mathrm{St}(\Delta) \ar[r]^{u}\ar@{^{(}->}[u]^{j_1}  & \lambda \boxtimes \mathrm{St}(\Delta)\ar[r]^{t} \ar@{^{(}->}[u]^{j_2} & \sigma'\times D_{\Delta}(\pi)\boxtimes \mathrm{St}(\Delta) \ar[r] \ar@{^{(}->}[u]^{j_3} & 0 } ,\]
where $s_1, s_2, s_3$ are the surjections to the top layer in the geometric lemma; and $\lambda=(\mathrm{St}(\Delta_1')\times \mathrm{St}(\Delta_2') \times D_{\Delta}(\pi)$; 
\[ \sigma=\mathrm{St}(\Delta_1'+\Delta_2'), \quad \sigma'=\mathrm{St}(\Delta_1'\cup \Delta_2'+\Delta_1\cap \Delta_2').\]

Using the strong commutativity (see \cite[Proposition 5.5]{Ch22+d}), we have a map $p: D_{\Delta}(\tau)\boxtimes \mathrm{St}(\Delta) \rightarrow \lambda \boxtimes \mathrm{St}(\Delta)$ such that $j_2\circ p= s_2\circ  i\circ i'$. Using \cite[Appendix B]{Ch22+} and our assumption (and so $D_{\Delta}(\tau)\cong D_{\Delta}\circ I_{\Delta_1'}\circ I_{\Delta_2'} \cong D_{\Delta}\circ I_{\Delta_2'}\circ I_{\Delta_1'}(\pi)\cong  I_{\Delta_2'}\circ I_{\Delta_1'}\circ D_{\Delta}(\pi)$), we have that $t\circ p\neq 0$ and so $j_3\circ t\circ p\neq 0$. This implies that $t''\circ i \neq 0$. Hence, this implies that $\tau \hookrightarrow \sigma' \times \pi$. Thus, \[ \tau \cong I_{\sigma'}(\pi) ,\]
With Lemma \ref{lem commutative for integrals}, it gives a contradiction to the given condition.
\end{proof}


\section{Intersection-union operations for derivatives in commutative triples} \label{s comm minimal pairs}

\subsection{Strong commutativity under intersection-union process for derivatives}

\begin{lemma} \label{lem derivative second intersect union}
Let $\Delta_1', \Delta_2'$ be linked segments. Let $\Delta$ be another segment. Let $\widetilde{\Delta}_1'=\Delta_1'\cup \Delta_2'$ and $\widetilde{\Delta}_2'=\Delta_1'\cap \Delta_2'$. Suppose $D_{\Delta}(\pi)\neq 0$ and 
\[   I_{\Delta_2'}\circ I_{\Delta_1'}\circ D_{\Delta}(\pi) \cong I_{\widetilde{\Delta}_2'}\circ I_{\widetilde{\Delta}_1'} \circ D_{\Delta}(\pi) .
\]
If $(\Delta, \Delta_1', \pi)$ and $(\Delta, \Delta_2', I_{\Delta_1'}(\pi))$ are strongly RdLi-commutative triples, then the following conditions hold:
\begin{enumerate}
\item $(\Delta, \widetilde{\Delta}_1', \pi)$ and $(\Delta, \widetilde{\Delta}_2', I_{\widetilde{\Delta}_1'}(\pi))$ are strongly RdLi-commutative triples;
\item $(\Delta, \widetilde{\Delta}_2', \pi)$ and $(\Delta, \widetilde{\Delta}_1', I_{\widetilde{\Delta}_2'}(\pi))$ are strongly RdLi-commutative triples;
\item $I_{\Delta_2'}\circ I_{\Delta_1'}(\pi)\cong I_{\widetilde{\Delta}_2'}\circ I_{\widetilde{\Delta}_1'}(\pi)$. 
\end{enumerate}
\end{lemma}

\begin{proof}
(3) follows from  Proposition \ref{prop commut triples ladder one}. Then (1) and (2) follow from Proposition \ref{prop intersect union comm}. 
\end{proof}

The following is the derivative analog for Proposition \ref{prop intersect union comm}.

\begin{proposition} \label{prop two strong commute dual 2} 
Let $\Delta_1, \Delta_2$ be linked segments. Let $\Delta'$ be another segment. Let $\widetilde{\Delta}_1=\Delta_1\cup \Delta_2$ and let $\widetilde{\Delta}_2=\Delta_1\cap \Delta_2$. Suppose $D_{\Delta_1}(\pi)\neq 0$ and $D_{\Delta_2}(\pi)\neq 0$, and $D_{\Delta_2}\circ D_{\Delta_1}(\pi)\cong D_{\widetilde{\Delta}_2}\circ D_{\widetilde{\Delta}_1}(\pi)$. If $(\Delta_1, \Delta', \pi)$ and $(\Delta_2, \Delta', D_{\Delta_1}(\pi))$ are strongly RdLi-commutative triples, then the following holds:
\begin{enumerate}
\item $(\widetilde{\Delta}_1, \Delta', \pi)$ and $(\widetilde{\Delta}_2, \Delta', D_{\widetilde{\Delta}_1}(\pi))$ are strongly RdLi-commutative triple;
\item $(\widetilde{\Delta}_2, \Delta', \pi)$ and $(\widetilde{\Delta}_1, \Delta', D_{\widetilde{\Delta}_2}(\pi))$ are strongly RdLi-commutative triple;
\item $D_{\Delta_2}\circ D_{\Delta_1}\circ I_{\Delta'}(\pi) \cong D_{\widetilde{\Delta}_2}\circ D_{\widetilde{\Delta}_1}\circ I_{\Delta'}(\pi)$.
\end{enumerate}
\end{proposition}

\begin{proof}
Let 
\[\tau=D_{\Delta_2}\circ D_{\Delta_1}\circ I_{\Delta'}(\pi)\cong I_{\Delta'}\circ D_{\Delta_2}\circ D_{\Delta_1}(\pi) \]
Using the given condition and Lemma \ref{lem unlinked comm der and integral basic}, one also has:
\[ \tau\cong  I_{\Delta'}\circ D_{\widetilde{\Delta}_1}\circ D_{\widetilde{\Delta}_2}(\pi)\cong I_{\Delta'}\circ D_{\widetilde{\Delta}_2}\circ D_{\widetilde{\Delta}_1}(\pi) .\]
 One first reformulates into equivalent conditions and statements involving $\tau$ by Theorem \ref{thm combinatorial def} and Proposition \ref{prop strong commute imply commute}: e.g. the original two strong commutations become that $(\Delta', \Delta_1, I_{\Delta_2}(\tau))$ and $(\Delta', \Delta_2, \tau)$ are strongly LdRi-commutative triples; (1) equivalent to $(\Delta', \widetilde{\Delta}_1, I_{\widetilde{\Delta}_2}(\tau))$ and $(\Delta', \widetilde{\Delta}_2, \tau)$ are strongly LdRi-commutative triples, etc. Then the equivalent statements follow from Lemma \ref{lem derivative second intersect union}.
\end{proof}

\subsection{Commutativity for derivatives}

\begin{lemma} \label{lem dual form for commute derivatives}
Let $\pi \in \mathrm{Irr}$. Let $\Delta_1'$ and $\Delta_2'$ be linked segments with $\Delta_1'< \Delta_2'$. Let $\Delta$ be another segment. Suppose 
\[  I_{\Delta_2'}\circ I_{\Delta_1'} \circ D_{\Delta}(\pi) \not\cong I_{\Delta_2'\cup \Delta_1'}\circ I_{\Delta_2' \cap \Delta_1'}\circ D_{\Delta}(\pi) .
\] 
Then,  $(\Delta, \Delta_1', \pi)$ and $(\Delta, \Delta_2', I_{\Delta_1'}(\pi))$ are strongly RdLi-commutative triples if and only if $(\Delta, \Delta_2', \pi)$ and $(\Delta, \Delta_1', I_{\Delta_2'}(\pi))$ are strongly RdLi-commutative triples. Moreover, if one of the equivalent conditions holds, then 
\[    I_{\Delta_2'}\circ I_{\Delta_1'}(\pi) \not\cong I_{\Delta_1' \cup \Delta_2'}\circ I_{\Delta_1'\cap \Delta_2'}(\pi).
\]
\end{lemma}

\begin{proof}
The last assertion follows from Proposition \ref{prop intersect union comm} and then the equivalent condition follows from Proposition \ref{prop commut triples ladder one}.
\end{proof}

The following is the derivative analog of Proposition \ref{prop commut triples ladder one}. One may also use Definition \ref{def dual com commutative} to show the last assertion of Proposition \ref{cor strong commute linked commute case} and then use a similar proof as in Proposition \ref{prop intersect union comm}. Once we have Proposition \ref{cor strong commute linked commute case}, we can also deduce back Lemma \ref{lem dual form for commute derivatives} by the duality.

\begin{proposition} \label{cor strong commute linked commute case}
Let $\pi \in \mathrm{Irr}$. Let $\Delta_1$ and $\Delta_2$ be linked segments with $\Delta_1<\Delta_2$. Let $\Delta'$ be another segment.  Suppose 
\[  D_{\Delta_2}\circ D_{\Delta_1}(\pi) \not\cong D_{\Delta_1\cup\Delta_2}\circ D_{\Delta_1\cap \Delta_2}(\pi) .\]
Then $(\Delta_1, \Delta', \pi)$ and $(\Delta_2, \Delta', D_{\Delta_1}(\pi))$ are strongly RdLi-commutative triples if and only if $(\Delta_2, \Delta', \pi)$ and $(\Delta_1, \Delta', D_{\Delta_2}(\pi))$ are strongly RdLi-commutative triples. Moreover, if one of the equivalent conditions holds, then 
\[  D_{\Delta_2}\circ D_{\Delta_1}\circ I_{\Delta'}(\pi) \not\cong D_{\Delta_1\cap \Delta_2}\circ D_{\Delta_1\cup \Delta_2}\circ I_{\Delta'}(\pi) .
\]
\end{proposition}

\begin{proof}
We only prove the only if direction. Suppose $(\Delta_1, \Delta', \pi)$ and $(\Delta_2, \Delta', D_{\Delta_1}(\pi))$ are strongly RdLi-commutative triples. 

Let $\omega=I_{\Delta'}\circ D_{\Delta_2}\circ D_{\Delta_1}(\pi)$. By Theorem \ref{thm combinatorial def}, $(\Delta', \Delta_1, I_{\Delta_2}(\omega))$ and   $(\Delta', \Delta_2, \omega)$ are strongly LdRi-commutative triples. Note that
\[  I_{\Delta_1}\circ I_{\Delta_2}(D_{\Delta'}(\omega))\cong \pi \not\cong I_{\Delta_1\cup \Delta_2}\circ I_{\Delta_1\cap \Delta_2}(D_{\Delta'}(\omega)) .\]
Hence, by Lemma \ref{lem dual form for commute derivatives}, $(\Delta', \Delta_1, \omega)$ and $(\Delta', \Delta_2, I_{\Delta_1}(\omega))$ are strongly LdRi-commutative triples. Now, by Theorem \ref{thm combinatorial def}, $(\Delta_1, \Delta', I_{\Delta_1}\circ D_{\Delta'}(\omega))$ and $(\Delta_2, \Delta', D_{\Delta'}\circ I_{\Delta_2}\circ I_{\Delta_1}(\omega))$ are strongly RdLi-commutative triples. Now using $D_{\Delta_2}\circ D_{\Delta_1}(\pi)\cong D_{\Delta_1}\circ D_{\Delta_2}(\pi)$ in Lemma \ref{lem commute and intersect union}, we have that  $(\Delta_2, \Delta', \pi)$ and $(\Delta_1, \Delta', D_{\Delta_2}(\pi))$ are strongly LdRi-commutative triples. 

The last assertion follows by reformulating the last assertion in Lemma \ref{lem dual form for commute derivatives} (with Proposition \ref{prop strong commute imply commute} for some commutations).
\end{proof}

\section{Strong commutation for multisegments} \label{s strong commutation for sequences}

\subsection{Strongly commutative triples for multisegments}

Recall that the RdLi- commutativity is in Definition \ref{def commutative triple}. We now define the LdRi-commutative triples:

\begin{definition} \label{def strong commutation sequences}
Let $\mathfrak m, \mathfrak n \in \mathrm{Mult}$. Let $\pi \in \mathrm{Irr}$. Suppose $D_{\mathfrak m}(\pi)\neq 0$. We write $\mathfrak m=\left\{ \Delta_1, \ldots, \Delta_r\right\}$ in an ascending order and $\mathfrak n=\left\{ \Delta_1', \ldots, \Delta_s' \right\}$ in an ascending order. For $1 \leq i \leq r$, let $ \widetilde{\mathfrak m}_i=\left\{ \Delta_i, \ldots, \Delta_r\right\}$. For $1 \leq j \leq s$, let $\widetilde{\mathfrak n}_j=\left\{ \Delta_j', \ldots, \Delta_s' \right\}$, and let $\mathfrak m_0=\emptyset$ and $\mathfrak n_0=\emptyset$. We say that $(\mathfrak m, \mathfrak n, \pi)$ is a {\it strongly RdLi-commutative triple} if for any $1 \leq i \leq r$ and $1 \leq j \leq s$, 
\[ (\Delta_i, \Delta_j', I^R_{\widetilde{\mathfrak n}_{j+1}}\circ D^L_{\widetilde{\mathfrak m}_{i+1}}(\pi)) \]
is a strongly RdLi-commutative triple in the sense of Definition \ref{def strong commu intro}. We emphasis that $\mathfrak m$ and $\mathfrak n$ are written in an ascending order (c.f. $D^L$ and $I^R$ defined before in Definition \ref{def commutative triple}).

Note that, by Propositions \ref{prop unlinked commute integral} and \ref{prop unlinked commute derivative}, the strong commutation for multisegments is independent for a choice of an ascending order.
\end{definition}

We have the following duality in view of Theorem \ref{thm combinatorial def}.

\begin{proposition} \label{prop dual multi strong comm}
Let $(\mathfrak m, \mathfrak n, \pi)$ be a strongly RdLi-commutative triple. Then $(\mathfrak n, \mathfrak m, I_{\mathfrak n}\circ D_{\mathfrak m}(\pi))$ is a strongly LdRi-commutative triple.
\end{proposition}

\begin{proof}
This follows from repeated uses of Theorem \ref{thm combinatorial def} (1)$\Longleftrightarrow$(2). For more details, let $\tau=I_{\mathfrak n}\circ D_{\mathfrak m}(\pi)$. Write
\[  \mathfrak m=\left\{ \Delta_1, \ldots, \Delta_r\right\}, \quad \mathfrak n=\left\{ \Delta_1', \ldots, \Delta_s' \right\} .\]
Let 
\[ \mathfrak m_i=\left\{ \Delta_1, \ldots, \Delta_i\right\}, \quad \widetilde{\mathfrak m}_{i+1}=\left\{ \Delta_{i+1}, \ldots, \Delta_r\right\};\]
 and let 
\[ \mathfrak n_j=\left\{ \Delta_1', \ldots, \Delta_j' \right\}, \quad\widetilde{\mathfrak n}_{j+1}=\left\{ \Delta_{j+1}',\ldots, \Delta_s' \right\}. \] 
The RdLi-commutativity implies that $(\Delta_i, \Delta_j', I_{\mathfrak n_{j-1}}\circ D_{\mathfrak m_{i-1}}(\pi))$ is a RdLi-commutative triple for any $i,j$. By Theorem \ref{thm combinatorial def}, $(\Delta_j', \Delta_i, I_{\Delta_j'}\circ D_{\Delta_i}\circ  I_{\mathfrak n_{j-1}}\circ D_{\mathfrak m_{i-1}}(\pi))$ is a LdRi-commutative triple. 

By Proposition \ref{prop strong commute imply commute} multiple times, we then have that $I_{\Delta_j'}\circ D_{\Delta_i}\circ  I_{\mathfrak n_{j-1}}\circ D_{\mathfrak m_{i-1}}(\pi)) \cong I^R_{\widetilde{\mathfrak m}_{i+1}}\circ D^L_{\widetilde{\mathfrak n}_{j+1}}(\tau)$. In order words, we have that 
\[  (\Delta_j', \Delta_i, I^R_{\widetilde{\mathfrak m}_{i+1}}\circ D^L_{\widetilde{\mathfrak n}_{j+1}}(\tau))
\]
is a LdRi-commutative triple, as desired.
\end{proof}

In terms of branching laws, one may consider Proposition \ref{prop dual multi strong comm} is compatible with the duality in (\ref{eqn duality for restriction}).

\subsection{Combinatorial commutation for a sequence}

\begin{lemma} \label{lem strong commutative combin}
Let $\mathfrak m, \mathfrak n \in \mathrm{Mult}$. Let $\Delta, \Delta'$ be segments. Let $\pi \in \mathrm{Irr}$. The following conditions are equivalent:
\begin{enumerate}
\item $(\mathfrak m, \Delta', \pi)$ is strongly RdLi-commutative;
\item $\eta^L_{\Delta'}(D_{\mathfrak m}\circ I_{\Delta'}(\pi))=\eta^L_{\Delta'}(I_{\Delta'}(\pi))$.
\end{enumerate} 
Similarly, the following conditions are equivalent:
\begin{enumerate}
\item $(\Delta, \mathfrak n, \pi)$ is strongly RdLi-commutative;
\item $\eta_{\Delta}(I_{\mathfrak n}(\pi))=\eta_{\Delta}(\pi)$. 
\end{enumerate}
\end{lemma}

\begin{proof}
Write $\mathfrak m=\left\{ \Delta_1, \ldots, \Delta_r \right\}$ with the labelling in an ascending order. Let
\[  \mathfrak m_j = \left\{ \Delta_1, \ldots, \Delta_j \right\} .
\]
We only prove the former one; and the latter equivalence is similar and slightly easier by repeatedly using combinatorial commutativity (and Theorem \ref{thm combinatorial def}). 

For the former one, we first have the following inequalities (by the left version of Lemma \ref{lem right multi submulti}):
\[  \eta^L_{\Delta'}(D_{\mathfrak m} \circ I_{\Delta'}(\pi))\leq \ldots \leq \eta^L_{\Delta'}(D_{\mathfrak m_j}\circ I_{\Delta'}(\pi)) \leq \ldots \leq \eta^L_{\Delta'}(I_{\Delta'}(\pi)) .
\]

We now assume (2). Then all the inequalities above are equalities. Thus, inductively, we have that 
\[ \eta^L_{\Delta'}(D_{\Delta_j}\circ I_{\Delta'}\circ D_{\mathfrak m_{j-1}}(\pi))=\eta^L_{\Delta'}(D_{\Delta_j}\circ D_{\mathfrak m_{j-1}}\circ I_{\Delta'}(\pi))=\eta^L_{\Delta'}(D_{\mathfrak m_j}\circ I_{\Delta'} (\pi)), \]
where the first equality follows from the strong commutativity in previous $j-1$ cases and Proposition \ref{prop strong commute imply commute}. Using Theorem \ref{thm combinatorial def}, we now have $(\Delta_j,\Delta', D_{\mathfrak m_{j-1}}(\pi))$ is strongly RdLi-commutative. 

(1) $\Rightarrow$ (2) follows from Theorem \ref{thm combinatorial def} and Proposition \ref{prop strong commute imply commute}.
\end{proof}

\section{Minimal strongly commutative triples} \label{s minimal strong commutative triples for seq}

\subsection{Minimality for pairs}

\begin{lemma} \label{lem minimal multisegments int}
Let $(\Delta, \Delta_1', \pi)$ and $(\Delta, \Delta_2', I_{\Delta_1'}(\pi))$ be strongly RdLi-commutative triples. Then
$I_{\Delta_2'}\circ I_{\Delta_1'}(\pi) \not\cong I_{\Delta_1'\cup \Delta_2'} \circ I_{\Delta_1'\cap \Delta_2'}(\pi)$ if and only if 
$I_{\Delta_2'}\circ I_{\Delta_1'}\circ D_{\Delta}(\pi) \not\cong I_{\Delta_1'\cup \Delta_2'}\circ I_{\Delta_1'\cap \Delta_2'}\circ D_{\Delta}(\pi)$.
\end{lemma}

\begin{proof}
The if direction follows from Proposition \ref{prop intersect union comm}. For the only if direction, it follows from Lemma \ref{lem derivative second intersect union}.
\end{proof}

\begin{lemma} \label{lem minimal multisegments der} 
Let $(\Delta_1, \Delta', \pi)$ and $(\Delta_2, \Delta', D_{\Delta_1}(\pi))$ be strongly RdLi-commutative triples. Then $D_{\Delta_2}\circ D_{\Delta_1}(I_{\Delta'}(\pi)) \not\cong D_{\Delta_1 \cup \Delta_2} \circ D_{\Delta_1\cap \Delta_2}(I_{\Delta'}(\pi))$ if and only if $D_{\Delta_2}\circ D_{\Delta_1}(\pi) \not\cong D_{\Delta_1\cup \Delta_2}\circ D_{\Delta_1\cap \Delta_2}(\pi)$.
\end{lemma}

\begin{proof}
The if part follows from Proposition \ref{cor strong commute linked commute case}. For the only if direction, it follows from \ref{prop two strong commute dual 2}. 
\end{proof}



\subsection{Minimality for a sequence}

We recall some theory of minimal sequences established in \cite{Ch22+aa}.

\begin{theorem} \label{thm minimality of d and i} \cite{Ch22+aa} (Uniqueness of minimality)
Let $\pi \in \mathrm{Irr}$. 
\begin{enumerate}
\item For $\mathfrak m, \mathfrak n \in \mathrm{Mult}$ with $D_{\mathfrak m}(\pi)\cong D_{\mathfrak n}(\pi)$ (resp. $D^L_{\mathfrak m}(\pi)\cong D^L_{\mathfrak n}(\pi)$), if both $\mathfrak m$ and $\mathfrak n$ are Rd-minimal (resp. Ld-minimal) to $\pi$, then $\mathfrak m=\mathfrak n$. 
\item For $\mathfrak m, \mathfrak n \in \mathrm{Mult}$ with $I_{\mathfrak m}(\pi)\cong I_{\mathfrak n}(\pi)$ (resp. $I^R_{\mathfrak m}(\pi)\cong I^R_{\mathfrak n}(\pi)$), if both $\mathfrak m$ and $\mathfrak n$ are Li-minimal (resp. Ri-minimal) to $\pi$, then $\mathfrak m=\mathfrak n$.
\end{enumerate}
\end{theorem}

\begin{proof}
(1) is proved in \cite[Theorem 1.2]{Ch22+aa}. For (2), let $\tau=I_{\mathfrak m}(\pi)$. Hence $D^R_{\mathfrak m}(\tau)\cong D^R_{\mathfrak n}(\tau)$. Li-minimality to $\pi$ implies Ld-minimality to $\tau$. Thus (2) follows from (1).
\end{proof}

\begin{definition} \label{def minimal strong commutation}
A strongly RdLi-commutative triple $(\mathfrak m, \mathfrak n, \pi)$ is said to be {\it minimal} if $\mathfrak m$ is Rd-minimal to $\pi$ and $\mathfrak n$ is Li-minimal to $\pi$.
\end{definition}

While we only require that the minimality holds for $\pi$ in Definition \ref{def minimal strong commutation}, the following proposition shows that the minimality holds for "intermediate terms".

\begin{proposition} \label{prop minimal for all}
Let $(\mathfrak m, \mathfrak n, \pi)$ be a minimal strongly RdLi-commutative triple.  Write $\mathfrak m=\left\{ \Delta_1, \ldots, \Delta_r \right\}$ in an ascending order and write $\mathfrak n=\left\{ \Delta_1', \ldots, \Delta_s' \right\}$ in an ascending order. Let $\mathfrak m_i=\left\{ \Delta_1, \ldots, \Delta_i\right\}$ and let $\mathfrak n_i=\left\{ \Delta_1', \ldots, \Delta_i' \right\}$  (with $\mathfrak n_0=\emptyset$ and $\mathfrak m_0=\emptyset$). Let $\bar{\mathfrak m}_i=\mathfrak m-\mathfrak m_i$ and let $\bar{\mathfrak n}_i=\mathfrak n-\mathfrak n_i$. Then, for any $i$ and $j$, $\bar{\mathfrak m}_i$ (resp. $\bar{\mathfrak n}_j$) is Rd-minimal (resp. Li-minimal) to $I_{\mathfrak n_j}\circ D_{\mathfrak m_i}(\pi)$.
\end{proposition}

\begin{proof}
It is shown in \cite[Theorem 1.1]{Ch22+aa} that checking minimality is reduced to two segment cases. Then, for integrals, it follows from Lemma \ref{lem minimal multisegments int}; and for derivatives, it follows from Lemma \ref{lem minimal multisegments der}. 
\end{proof}

\subsection{Commutativity under minimality and RdLi-commutativity}

We first recall the following commutativity result:

\begin{lemma} \label{lem commute and minimal} \cite[Theorem 5.4]{Ch22+bb}
Let $\pi \in \mathrm{Irr}$. Let $\mathfrak m \in \mathrm{Mult}$ be minimal to $\pi$. For any submultisegment $\mathfrak m'$ of $\mathfrak m$,
\begin{enumerate}
\item $ D_{\mathfrak m}(\pi) \cong D_{\mathfrak m-\mathfrak m'}\circ D_{\mathfrak m'}(\pi)$ and $\mathfrak m-\mathfrak m'$ is minimal to $D_{\mathfrak m'}(\pi)$; and
\item $\mathfrak m'$ is still minimal to $\pi$. 
\end{enumerate}
\end{lemma}

One can then apply Lemma \ref{lem commute and minimal} multiple times to obtain that:

\begin{corollary} \cite[Corollary 1.5]{Ch22+bb} \label{cor minimal derivative in any order}
Let $\mathfrak m \in \mathrm{Mult}$ be minimal to $\pi \in \mathrm{Irr}$. Write $\mathfrak m=\left\{ \Delta_1, \ldots, \Delta_r \right\}$ in any order. Then 
\[  D_{\Delta_r}\circ \ldots \circ D_{\Delta_1}(\pi) \cong D_{\mathfrak m}(\pi).
\]
\end{corollary}

We leave the formulations of an integral version of the above two results to the reader (also see the proof of Lemma \ref{lem commutative for integrals}). 

We have an analog of above two results for strongly RdLi-commutative triples:

\begin{corollary} \label{cor commute minimal strong rdli comm}
Let $\pi \in \mathrm{Irr}$. Let $(\mathfrak m, \mathfrak n, \pi)$ be a minimal strongly RdLi-commutative triple. Write the segments in $\mathfrak m=\left\{ \Delta_1, \ldots, \Delta_r \right\}$ and write the segments in $\mathfrak n=\left\{ \Delta_1', \ldots, \Delta_s' \right\}$ in any order. For any $\delta \in S_r$ and any $\delta'\in S_s$, let $\mathfrak m_i^{\delta}=\left\{ \Delta_{\delta(1)}, \ldots, \Delta_{\delta(i)} \right\}$ and let $\mathfrak n_j^{\delta'}=\left\{ \Delta_{\delta'(1)}', \ldots, \Delta_{\delta'(j)}' \right\}$. Then, for all $\delta \in S_r, \delta' \in S_s$, and for any $i, j \geq 1$, $(\Delta_{\delta(i)}, \Delta_{\delta'(j)}', I_{\mathfrak n^{\delta'}_{j-1}} \circ D_{\mathfrak m^{\delta}_{i-1}}(\pi))$ is a strongly RdLi-commutative triple.
\end{corollary}

\begin{proof}
By relabeling and using transpositions generating $S_n$, it suffices to show that if $(\Delta_i, \Delta_j', I_{\mathfrak n_{j-1}}\circ D_{\mathfrak m_{i-1}}(\pi))$ is a strongly RdLi-commutative triple for any $i,j$, then $(\Delta_{\delta(i)}, \Delta_{j}', \pi)$ (resp. $(\Delta_i, \Delta_{\delta'(j)}', \pi)$) is also a strongly RdLi-commutative triple for a transposition $\delta \in S_r$ (resp. $\delta'\in S_s$). We consider the statement for $\delta$ and suppose $\delta$ switches $x$ and $x+1$. Then we only have to prove the strong commutation of the following pairs for any $j$:
\begin{align} \label{eqn consecutive minimal pair}
  (\Delta_x, \Delta_j, \omega), \quad (\Delta_{x+1}, \Delta_j, D_{\Delta_x}(\omega)),
\end{align}
where $\omega=I_{\mathfrak n_{j-1}}\circ D_{\mathfrak m_{x-1}}(\pi)$. But, we have minimality of $\left\{ \Delta_x, \Delta_{x+1} \right\}$ to $\omega$ by Lemmas \ref{lem commute and minimal} and \ref{lem minimal multisegments der}, and so (\ref{eqn consecutive minimal pair}) follows from Proposition \ref{cor strong commute linked commute case} (for linked case) as well as Proposition \ref{prop unlinked commute derivative} (for unlinked case). For the statement for $\delta'$, one uses Lemma \ref{lem minimal multisegments int} and integral version of Lemma \ref{lem commute and minimal}.
\end{proof}

In particular, Proposition \ref{prop minimal for all}, Lemma \ref{lem commute and minimal} and Corollary \ref{cor commute minimal strong rdli comm} give the following two special cases:

\begin{corollary} (Subsequent property) \label{cor seq strong commut under minimal and sub}
Let $\mathfrak m'$ be a submultisegment of $\mathfrak m$ and let $\mathfrak n'$ be a submultisegment of $\mathfrak n$. Then $(\mathfrak m', \mathfrak n', \pi)$ is still a minimal strongly RdLi-commutative triple.
\end{corollary}

\begin{corollary} \label{cor seq strong commut under minimal} 
Let $\mathfrak m'$ be a submultisegment of $\mathfrak m$ and let $\mathfrak n'$ be a submultisegment of $\mathfrak n$. Then $(\mathfrak m-\mathfrak m', \mathfrak n-\mathfrak n', I_{\mathfrak n'}\circ D_{\mathfrak m'}(\pi))$ is also a minimal strongly RdLi-commutative triple.
\end{corollary}

\subsection{Duality under minimal commutative triples}

\begin{proposition} \label{prop minimal strong commute}
  The triple $(\mathfrak m, \mathfrak n, \pi)$ is a minimal strongly RdLi-commutative triple if and only if $(\mathfrak n, \mathfrak m, I_{\mathfrak n}\circ D_{\mathfrak m}(\pi))$ is a minimal strongly LdRi-commutative triple.
\end{proposition}

\begin{proof}
The commutative part follows from Proposition \ref{prop dual multi strong comm}. The minimality part follows from the duality between derivatives and integrals.
\end{proof}

\section{Generalized GGP relevant pairs and their uniqueness}

\subsection{Generalized relevant pairs}

\begin{definition} \label{def relevant pair content}
Let $\pi_1, \pi_2 \in \mathrm{Irr}$. We say that a pair $(\pi_1, \pi_2)$ is {\it relevant} if there exist multisegments $\mathfrak m$ and $\mathfrak n$ such that $I_{\mathfrak n}\circ D_{ \mathfrak m}(\nu^{1/2}\cdot\pi_1) \cong \pi_2$ and $(\mathfrak m, \mathfrak n, \nu^{1/2}\cdot \pi_1)$ is a strongly RdLi-commutative triple. We further say that $(\pi_1, \pi_2)$ is {\it $i^*$-relevant} if such $\mathfrak m$ satisfies $i^*=l_a(\mathfrak m)$.

\end{definition}

Our first main result is the following uniqueness statement, which shows that the $i^*$-relevance in Definition \ref{def relevant pair content} is well-defined. The main idea is to compare the invariant for $|\eta|_{\Delta}$ for a suitable choice of $\Delta$. 

\begin{theorem} \label{thm unique of relevant pairs}
Let $\pi_1, \pi_2 \in \mathrm{Irr}$. Suppose $(\pi_1, \pi_2)$ is relevant. There exist unique multisegments $\mathfrak m, \mathfrak n$ such that
\begin{itemize}
\item $(\mathfrak m, \mathfrak n, \nu^{1/2} \cdot\pi_1)$ is a minimal strongly RdLi-commutative triple; and
\item $D^R_{\mathfrak m}(\nu^{1/2} \cdot \pi_1)\cong D^L_{\mathfrak n}(\pi_2)$.
\end{itemize}
\end{theorem}

\begin{proof}
The existence part follows from the definition of relevance with Propositions \ref{prop intersect union comm} and \ref{prop two strong commute dual 2}, and the closedness property shown in \cite[Theorem 1.1]{Ch22+}.

We now prove the uniqueness part in the following steps: \\
\noindent
{\bf Step 1: Case of both $\mathfrak m$ and $\mathfrak m'$ to be non-empty}

Let $\pi=\nu^{1/2}\cdot \pi_1$. If both $\mathfrak m$ and $\mathfrak m'$ are empty, then the equality $\mathfrak n=\mathfrak n'$ follows from the uniqueness of a minimal multisegment.


\ \\
\noindent
{\bf Step 2: Choose an appropriate $\Delta_*$}


From now on, we assume at least one of $\mathfrak m, \mathfrak m'$ is non-empty. Let 
\[  \mathcal B=\left\{ b(\Delta) : \Delta \in \mathfrak m+\mathfrak m' \right\} .
\]
Choose a $\leq$-maximal $\rho$ in $\mathcal B$. 

Set $\mathfrak p=\mathfrak m_{b=\rho}$ and $\mathfrak p'=\mathfrak m'_{b=\rho}$. We shall only consider the case that both $\mathfrak p$ and $\mathfrak p'$ are non-empty. When one of $\mathfrak p$ and $\mathfrak p'$ is empty, it can be dealt similarly. Let $\Delta_*$ (resp. $\Delta_*'$) be the shortest segment in $\mathfrak p$ (resp. $\mathfrak p'$).  \\

\noindent
{\bf Step 3: A reduction using inductive hypothesis.}
Suppose $\Delta_* = \Delta_*'$. Then 
\[ \pi_2\cong  I^R_{\mathfrak n}\circ D_{\mathfrak m}^L(\pi) \cong I^R_{\mathfrak n}\circ D_{\Delta_*}^L\circ D_{\mathfrak m-\Delta_*}^L(\pi) \cong D_{\Delta_*}^L\circ  I^R_{\mathfrak n}\circ D_{\mathfrak m-\Delta_*}^L(\pi_1) 
\]
One can similarly obtain the expression by replacing $\mathfrak m$ and $\mathfrak n$ respectively by $\mathfrak m'$ and $\mathfrak n'$. Thus, by cancelling the term $D_{\Delta_*}^L$, we have:
\[ I^R_{\mathfrak n}\circ D_{\mathfrak m-\Delta_*}^L(\pi_1) \cong I^R_{\mathfrak n'}\circ D_{\mathfrak m'-\Delta_*}^L(\pi_1) .
\]
Since $\mathfrak m-\Delta_*$ and $\mathfrak m'-\Delta_*$ are still Rd-minimal to $\pi$ (see e.g. Corollary \ref{cor seq strong commut under minimal and sub}), the induction gives that $\mathfrak n=\mathfrak n'$ and $\mathfrak m-\Delta_*=\mathfrak m'-\Delta_*$. Thus we also have $\mathfrak m=\mathfrak m'$. \\

\noindent
{\bf Step 4: Setup notations for computing $\eta_{\Delta_*}$ in another case}

Thus it remains to show that $\Delta'_* \neq \Delta_*$ is not possible. By switching the labelling if necessary, we assume that $ \Delta_* \subsetneq \Delta'_*$. 

\begin{itemize}
\item Let $\mathfrak q$ (resp. $\mathfrak q'$) be the submultisegment of $\mathfrak m-\mathfrak p$ (resp. $\mathfrak m'-\mathfrak p'$) containing all segments $\widetilde{\Delta}$ with $\widetilde{\Delta} \subset \Delta_*$.
\end{itemize}

\ \\

\noindent
{\bf Step 5: Begin to compare $|\eta|_{\Delta}(\pi)$ and $|\eta|_{\Delta}(I_{\mathfrak n}\circ D_{\mathfrak m}(\pi))$.}

Write segments in $\mathfrak n$ as $\overline{\Delta}_1, \ldots, \overline{\Delta}_s$ in an ascending order. For $1\leq j \leq s$, write $\mathfrak n_j=\left\{ \overline{\Delta}_1, \ldots, \overline{\Delta}_j\right\}$. We first note that the commutativity of linked segments gives the following equations: for any $j$,
\begin{align}
  D_{\mathfrak m}(I_{\mathfrak n_j}(\pi)) &= D_{\mathfrak p}\circ D_{\mathfrak m-\mathfrak p}(I_{\mathfrak n_j}(\pi)) \\
	                & = D_{\mathfrak p} \circ D_{\mathfrak q} \circ D_{\mathfrak m-\mathfrak p-\mathfrak q}(I_{\mathfrak n_j}(\pi)) \\
									& = D_{\mathfrak q} \circ D_{\mathfrak p}\circ D_{\mathfrak m-\mathfrak p-\mathfrak q}(I_{\mathfrak n_j}(\pi))  ,
\end{align}
where the first equation follows by arranging the segments in an ascending order of $b(\Delta)$ and the second equation follows by arranging the segments in an ascending order of $a(\Delta)$.

Since $(\Delta_*, \overline{\Delta}_j, D_{\mathfrak m-\mathfrak p-\mathfrak q}\circ I_{\mathfrak n_{j-1}}(\pi))$ is a strongly RdLi-commutative triple by Corollary \ref{cor seq strong commut under minimal and sub}, we have that, by Theorem \ref{thm combinatorial def},  
\[  \eta_{\Delta_*}( D_{\mathfrak m-\mathfrak p-\mathfrak q}\circ I_{\mathfrak n_{j-1}}(\pi)))=\eta_{\Delta_*}(D_{\mathfrak m-\mathfrak p-\mathfrak q}\circ I_{\mathfrak n_{j}}(\pi)) .
\]
Thus inductively, we have the first equation: for any $j$,
\[ (\star) \quad  \eta_{\Delta_*}( D_{\mathfrak m-\mathfrak p-\mathfrak q}\circ I_{\mathfrak n_{j}}(\pi)))=\eta_{\Delta_*}(D_{\mathfrak m-\mathfrak p-\mathfrak q}(\pi)) .
\]

\noindent
{\bf Step 6: Compute $\eta_{\Delta_*}(D_{\mathfrak m-\mathfrak p-\mathfrak q}(\pi))$} 

It follows from Lemma \ref{lem minimal eta criterai} that
\[ (\star\star)\quad  \eta_{\Delta_*}(\pi)= \eta_{\Delta_*}(D_{\mathfrak m-\mathfrak p-\mathfrak q}(\pi)) .\]

\noindent
{\bf Step 7: Compute $|\eta|_{\Delta_*}(D_{\mathfrak m-\mathfrak q}\circ I_{\mathfrak n}(\pi))$} \\

We now consider $|\eta|_{\Delta_*}(D_{\mathfrak p}\circ D_{\mathfrak m-\mathfrak p-\mathfrak q}\circ I_{\mathfrak n}(\pi))$. By using Proposition \ref{prop eta function derivative}(1) and (2), we have that
\[  |\eta|_{\Delta_*}(D_{\mathfrak p}\circ D_{\mathfrak m-\mathfrak p-\mathfrak q}\circ I_{\mathfrak n}(\pi)) =|\eta|_{\Delta_*}(D_{\mathfrak m-\mathfrak p-\mathfrak q}\circ I_{\mathfrak n}(\pi))- l,
\]
where $l$ is the number of $\Delta_*$ in $\mathfrak p$ (which is at least one from our choice). Thus we have the following strict inequality, by Proposition \ref{prop eta function derivative}(2),
\[  (\star\star\star)\quad |\eta|_{\Delta_*}(D_{\mathfrak m-\mathfrak p-\mathfrak q}\circ I_{\mathfrak n}(\pi)) > |\eta|_{\Delta_*}(D_{\mathfrak p}\circ D_{\mathfrak m-\mathfrak p-\mathfrak q}\circ I_{\mathfrak n}(\pi)) .\]

\noindent
{\bf Step 8: Compute $|\eta|_{\Delta_*}(D_{\mathfrak m}\circ I_{\mathfrak n}(\pi))$ and get the comparison} \\
We finally have that, by Proposition \ref{prop eta function derivative}(3),
\[ (\star\star\star\star) \quad |\eta|_{\Delta_*}(D_{\mathfrak q}\circ D_{\mathfrak p}\circ D_{\mathfrak m-\mathfrak p-\mathfrak q}\circ I_{\mathfrak n}(\pi)) =|\eta|_{\Delta_*}(D_{\mathfrak p}\circ D_{\mathfrak m-\mathfrak p-\mathfrak q}\circ I_{\mathfrak n}(\pi)) . \]

Now combining $(\star)-(\star\star\star\star)$, we have:
\[(\bullet)\quad   |\eta|_{\Delta_*}(\pi)>|\eta|_{\Delta_*}(D_{\mathfrak m}\circ I_{\mathfrak n}(\pi)) .\]

\ \\
\noindent 
{\bf Step 9: Compare $|\eta|_{\Delta_*}(I_{\mathfrak n'}(\pi))$ and $|\eta|_{\Delta_*}(\pi)$}

On the other hand, by Lemma \ref{lem right multi submulti},
\[ (\ast)\quad   |\eta|_{\Delta_*}(I_{\mathfrak n'}(\pi))\geq |\eta|_{\Delta_*}(\pi) .
\]
Now, by Proposition \ref{prop eta function derivative}(4), 
\[  (\ast\ast)\quad |\eta|_{\Delta_*}(D_{\mathfrak m'-\mathfrak p'}\circ I_{\mathfrak n'}(\pi)) \geq |\eta|_{\Delta_*}( I_{\mathfrak n'}(\pi))
\]
Now, since we are assuming $ \Delta_* \subsetneq \Delta_*'$, Proposition \ref{prop eta function derivative}(1) gives that:
\[ (\ast\ast\ast)  \quad  |\eta|_{\Delta_*}(D_{\mathfrak p'}\circ D_{\mathfrak m'-\mathfrak p'}\circ I_{\mathfrak n'}(\pi)) = |\eta|_{\Delta_*}( D_{\mathfrak m'-\mathfrak p'}\circ I_{\mathfrak n'}(\pi)) .
\]
Thus combining $(\ast), (\ast\ast), (\ast\ast\ast)$, we have that 
\[ (\bullet \bullet) \quad |\eta|_{\Delta_*}(D_{\mathfrak m'}\circ I_{\mathfrak n'}(\pi))\geq |\eta|_{\Delta_*}(\pi) .
\]

\noindent
{\bf Step 10: Arrive a contradiction} \\
Since $D_{\mathfrak m'}\circ I_{\mathfrak n'}(\pi)\cong D_{\mathfrak m}\circ I_{\mathfrak n}(\pi)$, we arrive a contradiction from $(\bullet)$ and $(\bullet \bullet)$ as desired.
\end{proof}



\begin{remark}
An alternative way to show the well-definedness of $i^*$-relevance is to use Proposition \ref{prop smallest integer in strong triple} below. Granting that, once one proves Theorem \ref{thm sufficiency} (whose proof is independent of Theorem \ref{thm unique of relevant pairs}), and one can use the multiplicity one of branching laws to deduce the uniqueness in Theorem \ref{thm unique of relevant pairs} i.e. an alternate (indirect) proof. 

\end{remark}

\section{Double derivatives and double integrals} \label{s double der int}

\subsection{Level preserving integrals for highest derivative multisegments}

\begin{lemma} \label{lem level preserve hd}
Let $\pi \in \mathrm{Irr}$. Let $\Delta$ be a segment. Suppose $\mathrm{lev}(I_{\Delta}(\pi))=\mathrm{lev}(\pi)$. Then 
\[  \mathfrak{hd}(\pi) =\mathfrak{hd}(I_{\Delta}(\pi)) .
\]
\end{lemma}
\begin{proof}
Since $\mathrm{lev}(\pi)=\mathrm{lev}(I_{\Delta}(\pi))$, $l_a(\mathfrak{hd}(\pi))=l_a(\mathfrak{hd}(I_{\Delta}(\pi))$, the inequality in Lemma \ref{lem right multi submulti} must be an equality. Thus $\mathfrak{hd}(\pi)=\mathfrak{hd}(I_{\Delta}(\pi))$ as desired (see the proof of \cite[Corollary 11.2]{Ch22+b}). 
\end{proof}

\subsection{Level preserving integrals}

\begin{proposition} \label{level preserving strong comm}
Let $\pi \in \mathrm{Irr}$. Let $\Delta$ be a segment. Suppose $\mathrm{lev}(I_{\Delta}(\pi))=\mathrm{lev}(\pi)$. Let $\mathfrak n$ be a multisegment such that $D_{\mathfrak n}(\pi)\neq 0$ Then $(\mathfrak n, \Delta, \pi)$ is a strongly RdLi-commutative triple.
\end{proposition}

\begin{proof}
Write $\mathfrak n=\left\{ \Delta_1, \ldots, \Delta_r \right\}$ in an ascending order. Let $\mathfrak n_j=\left\{ \Delta_1, \ldots, \Delta_j \right\}$. By Theorem \ref{thm combinatorial def}, it suffices to prove that those pairs are combinatorially RdLi-commutative triples. By Proposition \ref{prop unlinked commute derivative}, it suffices to consider that $\Delta_1, \ldots, \Delta_r$ are arranged such that for any $i<j$, 
\[ a(\Delta_i) \not > a(\Delta_j)  .\]
Note that the combinatorial commutativity of $(\Delta_1, \Delta, \pi)$ follows from Lemma \ref{lem level preserve hd}. (The subtly on the general case is that the level of $I_{\Delta}\circ D_{\mathfrak n_j}(\pi)$ may not be equal to $D_{\mathfrak n_j}(\pi)$ and we cannot prove inductively by using Lemma \ref{lem level preserve hd}.)


Since $\mathfrak{hd}(\pi)=\mathfrak{hd}(I_{\Delta}(\pi))$, by Theorem \ref{thm change in highest derivative after d}, 
\[  \varepsilon_{\Delta'}(D_{\mathfrak n_{j-1}}(\pi))=\varepsilon_{\Delta'}(D_{\mathfrak n_{j-1}}\circ I_{\Delta}(\pi)) 
\]
for any $\Delta_j$-saturated segment $\Delta'$. This in turns gives that $\eta_{\Delta_j}(D_{\mathfrak n_{j-1}}(\pi))=\eta_{\Delta_j}(D_{\mathfrak n_{j-1}}\circ I_{\Delta}(\pi))$ which is also equal to $\eta_{\Delta_j}( I_{\Delta} \circ D_{\mathfrak n_{j-1}}(\pi))$ by inductively using Proposition \ref{prop strong commute imply commute}. Then the proposition now follows from Theorem \ref{thm combinatorial def}.
\end{proof}



\subsection{Double derivatives and integrals}

One key result of \cite{Ch22+} is the following double derivative:

\begin{theorem} \cite[Theorem 1.4]{Ch22+} \label{thm double derivative}
Let $\pi \in \mathrm{Irr}$. Let $\mathfrak m \in \mathrm{Mult}$. Suppose $D_{\mathfrak m}(\pi)\neq 0$. Then there exists a multisegment $\mathfrak n$ such that 
\[  D_{\mathfrak n}\circ D_{\mathfrak m}(\pi) \cong \pi^- .
\]
Moreover, one can take $\mathfrak n=\mathfrak r(\mathfrak m, \mathfrak{hd}(\pi))$.
\end{theorem}

We now deduce an integral version from the double derivative one. A simple idea is the following. Let $\pi \in \mathrm{Irr}$ and let $\mathfrak m \in \mathrm{Mult}$ such that $\mathrm{lev}(I_{\mathfrak m}(\pi))=\mathrm{lev}(\pi)$ and $\pi$ is thickened i.e. all segments in the associated multisegment have length $2$.

\begin{theorem} \label{thm double integral}
Let $\pi \in \mathrm{Irr}$. Let $\mathfrak m \in \mathrm{Mult}$. There exists a multisegment $\mathfrak n$ such that 
\[   {}^-(I_{\mathfrak n}\circ I_{\mathfrak m}(\pi)) \cong \pi 
\]
and $\mathrm{lev}(I_{\mathfrak n}\circ I_{\mathfrak m}(\pi))=\mathrm{lev}(I_{\mathfrak m}(\pi))$. 
\end{theorem}

\begin{proof}

For a segment $\Delta=[a,b]_{\rho}$, define $\Delta^+=[a,b+1]_{\rho}$. Let $\mathfrak p \in \mathrm{Mult}$ be such that $I_{\mathfrak m}(\pi)\cong \langle \mathfrak p \rangle$. Define
\[  \mathfrak p^+ = \sum_{\Delta \in \mathfrak p} \Delta^ + .
\]
Let $\tau=I_{\mathfrak m}(\pi)$ and $\tau^+=\langle \mathfrak p^+\rangle$.

Let $\mathfrak h \in \mathrm{Mult}$ such that $D_{\mathfrak h}(\tau^+)\cong (\tau^+)^-\cong \tau$ (see Definition \ref{def highest der multi}). Thus we can rewrite as 
\begin{align} \label{eqn level unchanged in double integral}
   \tau^+ \cong I^{R}_{\mathfrak h}(\tau) .
	\end{align}

\noindent
{\it Claim  1:} $ I^R_{\mathfrak h}(\pi) \cong D^L_{\mathfrak m}\circ I^R_{\mathfrak h}(\tau)$. \\

\noindent
{\it Claim 2:} There exists $\mathfrak n \in \mathrm{Mult}$ such that 
\[  D^L_{\mathfrak n}\circ I^R_{\mathfrak h}(\pi) \cong \nu\cdot \tau .
\]

Suppose Claim 2 holds in the meanwhile. Then, rewriting Claim 2, we have:
\[    \nu^{-1}\cdot I^R_{\mathfrak h}(\pi) \cong I_{\nu^{-1}\mathfrak n}(\tau) .
\]
Now, we have that $\mathrm{lev}(\nu^{-1}\cdot I^R_{\mathfrak h}(\pi))=\mathrm{lev}(D^L_{\mathfrak m}\circ I^R_{\mathfrak h}(\tau))=\mathrm{lev}(\tau^+)=\mathrm{lev}(\tau)$, which checks the level condition. (Here the first equality follows from Claim 1, and the second equality follows from that all segments for $\tau^+$ has at least of relative length $2$, also c.f. Lemma \ref{lem multisegment simple q bz} below.) Since 
\[ \mathrm{lev}(\nu^{-1}\cdot I^R_{\mathfrak h}(\pi))=|\mathfrak h|, \] 
\[ (\nu^{-1}\cdot I^R_{\mathfrak h}(\pi))^-=D_{\nu^{-1}\cdot \mathfrak h}\circ (\nu^{-1}\cdot I^R_{\mathfrak h}(\pi))=\nu^{-1}\cdot \pi.
\]
Switching to the highest left derivative, we obtain the statement. \\

We now prove Claim 1. Since $\mathrm{lev}(I^R_{\mathfrak h}(\tau))=\mathrm{lev}(\tau)$ (see (\ref{eqn level unchanged in double integral})), using Propositions \ref{level preserving strong comm} and \ref{prop strong commute imply commute} multiple times, we have:
\[  D^L_{\mathfrak m}\circ I^R_{\mathfrak h}(\tau)\cong I^R_{\mathfrak h}\circ D^L_{\mathfrak m}(\tau),
\]
which is equivalent to Claim 1.

We now prove Claim 2. By the left version of Theorem \ref{thm double derivative}, we have a multisegment $\mathfrak n$ such that 
\[  D^L_{\mathfrak n}\circ D^L_{\mathfrak m} \circ I^R_{\mathfrak h}(\tau)\cong {}^-(I^R_{\mathfrak h}(\tau)) \cong \nu \cdot (I^R_{\mathfrak h}(\tau))^-\cong \nu\cdot (\tau^+)^-\cong \nu\cdot \tau .
\]
Combining with Claim 1, we have Claim 2.
\end{proof}

\section{Symmetry property of relevant pairs}

We now prove a symmetry property for relevance. Using the symmetry property with Proposition \ref{prop dual multi strong comm}, we also see that it is not necessary to introduce some left-right terminologies for relevance. 

\begin{theorem} \label{thm symmetric property of relevant}
Let $\pi, \pi' \in \mathrm{Irr}$. Then $(\pi, \pi')$ is relevant if and only if $(\pi', \pi)$ is relevant.
\end{theorem}

\begin{proof}

We only prove one direction and the other direction is similar. Since $(\pi, \pi')$ is relevant, there exist multisegments $\mathfrak m$ and $\mathfrak n$ such that 
\[   D_{\mathfrak m}^R(\nu^{1/2}\cdot\pi)  \cong D_{\mathfrak n}^L(\pi') .\]  
This implies that $I_{\mathfrak n}^L\circ D_{\mathfrak m}^R(\nu^{1/2}\cdot \pi)\cong \pi'$. 

\noindent
{\bf Step 1: Construct dual multisegments by double derivatives and double integrals.}

Set $\widetilde{\pi} =\nu^{1/2}\cdot \pi$. By Theorem \ref{thm double integral}, there exists a multisegment $\mathfrak n'$ such that 
\[ (*)\quad  {}^-(I_{\mathfrak n'}\circ I_{\mathfrak n} (\widetilde{\pi})) \cong \widetilde{\pi} 
\]
and 
\[ (**) \quad \mathrm{lev}(I_{\mathfrak n'}\circ I_{\mathfrak n}(\widetilde{\pi}))=\mathrm{lev}(I_{\mathfrak n}(\widetilde{\pi})) . \]
On the other hand, by Theorem \ref{thm double derivative}, there exists a multisegment $\mathfrak m'$ such that 
\[ (***)\quad   D_{\mathfrak m'}\circ D_{\mathfrak m}(I_{\mathfrak n}(\widetilde{\pi})) \cong (I_{\mathfrak n}(\widetilde{\pi}))^- .
\]

\noindent
{\bf Step 2: Show a commutation using level preserving.} \\
Write the segments in $\mathfrak n'$ in an ascending order: $\underline{\Delta}_1, \ldots, \underline{\Delta}_s$ . Let $\mathfrak n'_j=\left\{ \underline{\Delta}_1, \ldots, \underline{\Delta}_j\right\}$ and let $\pi_j=I_{\mathfrak n_{j}'}\circ I_{\mathfrak n}(\widetilde{\pi})$ and $\pi_0=I_{\mathfrak n}(\widetilde{\pi})$. Thus (**) implies that $\mathrm{lev}(I_{\underline{\Delta}_j}(\pi_{j-1}))=\mathrm{lev}(\pi_{j-1})$ and so by Proposition \ref{level preserving strong comm}, $(\mathfrak m,  \underline{\Delta}_j , \pi_{j-1})$ is a strongly RdLi-commutative triple for all $j$. In particular, we have that 
\begin{align} \label{eqn level commute sym proof}  D_{\mathfrak m} \circ I_{\mathfrak n'_j}\circ I_{\mathfrak n}(\widetilde{\pi})\cong I_{\mathfrak n_j'}\circ D_{\mathfrak m}\circ I_{\mathfrak n}(\widetilde{\pi}) ,
\end{align}
where the isomorphism follows from repeatedly using Proposition \ref{prop strong commute imply commute}.

(The point of (\ref{eqn level commute sym proof}) is to exploit $(\star)$ shown in Step 3 below.)

\noindent
{\bf Step 3: Check strong commutation.} \\
\noindent
{\it Claim:} Let $\omega=D_{\mathfrak m}\circ I_{\mathfrak n}(\widetilde{\pi})$. We have $(\mathfrak m', \mathfrak n', \omega)$ is strongly RdLi-commutative.

\noindent
{\it Proof of claim:} By using (**) and Lemma \ref{lem level preserve hd}, 
\[ \mathfrak{hd}(\pi_j)=\mathfrak{hd}(I_{\mathfrak n}(\widetilde{\pi})) .\]
 Thus, 
\[   \mathfrak r(\mathfrak m, \pi_j)=\mathfrak r(\mathfrak m, I_{\mathfrak n}(\widetilde{\pi})) .
\]
Hence, by (the second assertion of) Theorem \ref{thm double derivative},
\[ (\star)\quad  D_{\mathfrak m'} \circ D_{\mathfrak m}(\pi_j) \cong (\pi_j)^- . \]


Now, by the left version of Lemma \ref{lem right multi submulti},
\begin{align*}
 \eta^L_{\underline{\Delta}_j}(D_{\mathfrak m'} \circ D_{\mathfrak m}(\pi_j)) 
&\leq \eta^L_{\underline{\Delta}_j}(D_{\mathfrak m'_{r-1}}\circ D_{\mathfrak m}(\pi_j)) \\
& \leq \ldots \\
&\leq \eta^L_{\underline{\Delta}_j}(D_{\mathfrak m}(\pi_j)) \\
& \leq \eta^L_{\underline{\Delta}_j}(\pi_j) ,
\end{align*}
On the other hand, by Proposition \ref{level preserving strong comm}, 
\[  \eta^L_{\underline{\Delta}_j}((I_{\mathfrak n_j'}\circ I_{\mathfrak n}(\widetilde{\pi}))^-) = \eta^L_{\underline{\Delta}_j}(I_{\mathfrak n_j'}\circ I_{\mathfrak n}(\widetilde{\pi})) .
\]
Thus, with $(\star)$, the above inequalities are equalities and so
\[  \eta^L_{\underline{\Delta}_j}( D_{\mathfrak m'}\circ D_{\mathfrak m}\circ I_{\mathfrak n_j'}\circ I_{\mathfrak n}(\widetilde{\pi})) = \eta^L_{\underline{\Delta}_j}(D_{\mathfrak m}\circ I_{\mathfrak n_j'}\circ I_{\mathfrak n}(\widetilde{\pi})) ,
\] 
and hence
\[  
\]

Thus, we  use Proposition \ref{level preserving strong comm} and Proposition \ref{prop strong commute imply commute} again, we have 
\[   I_{\underline{\Delta}_j'}\circ D_{\mathfrak m} (\pi_{j-1}) \cong D_{\mathfrak m}\circ I_{\underline{\Delta}_j'}\circ \pi_{j-1} \cong D_{\mathfrak m}(\pi_j) 
\]
and so, combining above inequalities, 
\[  \eta^L_{\underline{\Delta}_j}(D_{\mathfrak m'}\circ I_{\underline{\Delta}_j'}\circ D_{\mathfrak m} \circ I_{\mathfrak n_{j-1}'} \circ I_{\mathfrak n}(\widetilde{\pi})) = \eta^L_{\underline{\Delta}_j}(I_{\underline{\Delta}_j'}\circ D_{\mathfrak m} \circ I_{\mathfrak n_{j-1}'}\circ I_{\mathfrak n}(\widetilde{\pi})) .
\]
By Lemma \ref{lem strong commutative combin}, we have that $(\mathfrak m', \underline{\Delta}_j',  D_{\mathfrak m} \circ I_{\mathfrak n_{j-1}'}\circ I_{\mathfrak n}(\widetilde{\pi})))$ is a strongly RdLi-commutative triple. With $ D_{\mathfrak m}\circ I_{\mathfrak n_{j-1}'} \circ I_{\mathfrak n}(\widetilde{\pi})= I_{\mathfrak n_{j-1}'}\circ D_{\mathfrak m}\circ I_{\mathfrak n}(\widetilde{\pi})$ (see \ref{eqn level commute sym proof}), we have that $(\mathfrak m', \mathfrak n',  D_{\mathfrak m}\circ I_{\mathfrak n}(\widetilde{\pi}))$ is a strongly RdLi-commutative triple, proving the claim. \\

 Recall that $\pi' \cong D_{\mathfrak m}\circ I_{\mathfrak n}(\widetilde{\pi})$ by the relevance. Thus, Step 3 implies that $(\mathfrak m', \mathfrak n',\pi')$ is a strongly RdLi-commutative triple.


\noindent
{\bf Step 4: Check the isomorphism condition.} \\

 To check $(\pi', \pi)$ is relevant, it remains to show that $\nu^{1/2}\cdot D_{\mathfrak m'}\circ I_{\mathfrak n'}(\pi') \cong \pi$. To this end, we consider
\begin{align*}
     \nu^{1/2}\cdot D_{\mathfrak m'}\circ I_{\mathfrak n'}(\pi')  & \cong \nu^{1/2}\cdot D_{\mathfrak m'}\circ I_{\mathfrak n'} \circ D_{\mathfrak m}\circ I_{\mathfrak n}(\widetilde{\pi}) \\
		&  \cong \nu^{1/2}\cdot D_{\mathfrak m'} \circ D_{\mathfrak m}\circ I_{\mathfrak n'}\circ I_{\mathfrak n}(\widetilde{\pi}) \\
		&  \cong \nu^{1/2}\cdot(I_{\mathfrak n'}\circ I_{\mathfrak n}(\widetilde{\pi}))^- \\
		& \cong  \nu^{-1/2}\cdot {}^-(I_{\mathfrak n'}\circ I_{\mathfrak n}(\widetilde{\pi})) \\
		&  \cong \nu^{-1/2} \cdot \widetilde{\pi} \\
		& \cong \pi
\end{align*}
where the second isomorphism follows from Step 2, the third isomorphism follows from $(\star)$, the forth isomorphism follows from the highest derivatives of Zelevinsky \cite{Ze80}, the fifth isomorphism follows from (*). This completes the proof.
\end{proof}

\begin{corollary}
Let $\pi_1, \pi_2 \in \mathrm{Irr}$. Then $(\pi_1, \pi_2)$ is a relevant pair if and only if $(\pi_1^{\vee}, \pi_2^{\vee})$ is a relevant pair.
\end{corollary}

\begin{proof}
For a segment $\Delta=[a,b]_{\rho}$, let $\Delta^{\vee}=[-b,-a]_{\rho^{\vee}}$. For a multisegment $\mathfrak m=\left\{ \Delta_1, \ldots, \Delta_k \right\}$, let $\mathfrak m^{\vee}=\left\{ \Delta_1^{\vee}, \ldots, \Delta_k^{\vee}\right\}$. In general, we have 
\[ \eta_{\Delta}^L(\nu^{1/2}\pi_1) = \eta_{\Delta^{\vee}}^R(\nu^{-1/2}\pi_1^{\vee}) , \]
and $(I_{\Delta}^R(\nu^{1/2}\pi_1))^{\vee}=I_{\Delta^{\vee}}^L(\nu^{-1/2}\pi_1^{\vee})$. This implies that $(\mathfrak m^{\vee}, \mathfrak n^{\vee}, \nu^{-1/2}\cdot \pi_1^{\vee})$ is a strongly LdRi-commutative triple. Hence, $(\mathfrak n^{\vee}, \mathfrak m^{\vee}, \pi_2^{\vee})$ is a strongly RdLi-commutative triple by Proposition \ref{prop dual multi strong comm} and
\[  \nu^{1/2}\cdot D_{\mathfrak n^{\vee}}\circ I_{\mathfrak m^{\vee}}(\pi_2^{\vee}) \cong \pi_1^{\vee} .
\]
Thus $(\pi_2^{\vee}, \pi_1^{\vee})$ is relevant and so is $(\pi_1^{\vee}, \pi_2^{\vee})$ by Theorem \ref{thm symmetric property of relevant}.
\end{proof}


\begin{corollary} \label{cor relevant under theta}
Let $\pi_1, \pi_2 \in \mathrm{Irr}$. Then $(\pi_1, \pi_2)$ is a relevant pair if and only if $(\theta(\pi_1), \theta(\pi_2))$ is a relevant pair.
\end{corollary}

\begin{proof}
It follows from a result of Gelfand-Kazhdan that $\theta(\pi_1)\cong \pi_1^{\vee}$ and $\theta(\pi_2)\cong \pi_2^{\vee}$. 
\end{proof}








\part{Proof of sufficiency of generalized relevance} \label{part proof sufficiency}

For an overview of this part, see Section \ref{ss outline sufficient}. The first goal of this part is to compute a certain smallest derivative for achieving the relevance in Section \ref{s smallest derivative strongly}, for which we need tools in Sections \ref{s delta reduced repn} and \ref{s complete to reduced}. Section \ref{s characterize bz filtration for bl} explains connections of branching laws with the smallest derivatives. Section \ref{s construct branching via rs} studies a construction of branching law from Rankin-Selberg integrals. Section \ref{s smallest der control branching law} studies a BZ filtration and analyzes which layers could contribute a branching law by using the smallest derivatives. We prove the sufficiency of the relevance in Section \ref{s proof sufficiency}.

\section{ $\Delta$-reduced representations} \label{s delta reduced repn}

\subsection{$\Delta$-reduced representation}

\begin{lemma} \label{lem produce more strong rdli commutative triples} (c.f. \cite[Section 8.3]{Ch22+d})
Let $\pi \in \mathrm{Irr}$. Let $\Delta$ be a segment. Let $\mathfrak p=\mathfrak{mx}(\pi, \Delta)$. Suppose $(\Delta, \mathfrak n, \pi)$ is a strongly RdLi-commutative triple. Then,
\begin{enumerate}
\item $(\mathfrak p, \mathfrak n, \pi)$ is also a strongly RdLi-commutative triple.
\item $D_{\mathfrak p}\circ I_{\mathfrak n}(\pi)$ is R-$\Delta$-reduced (see Section \ref{ss standard trick intro}).
\end{enumerate}
\end{lemma}

\begin{proof}
By the strong commutativity of $(\Delta, \mathfrak n, \pi)$ and Lemma \ref{lem strong commutative combin}, 
\[   \eta_{\Delta}(\pi)=\eta_{\Delta}(I_{\mathfrak n}(\pi)) .
\]
Hence, by Theorem \ref{thm change in highest derivative after d},
\[  \eta_{\Delta}(D_{\mathfrak p}\circ I_{\mathfrak n}(\pi))=\eta_{\Delta}(D_{\mathfrak p}(\pi))=0 .
\]
This shows (2).

Let $\Delta'$ be a $\Delta$-saturated segment. Since both $\eta_{\Delta'}(D_{\Delta}(\pi))$ and $\eta_{\Delta'}(D_{\Delta}\circ I_{\mathfrak n}(\pi))$ are obtained by reducing the factor $\varepsilon_{\Delta}$ by $1$ and all other $\varepsilon_{\Delta'}$ unchanged, we still have that:
\[   \eta_{\Delta'}(D_{\Delta}(\pi)) = \eta_{\Delta'}(D_{\Delta}\circ I_{\mathfrak n}(\pi)) .
\]
Now by Proposition \ref{prop strong commute imply commute}, we have:
\[  \eta_{\Delta'}(D_{\Delta}(\pi)) = \eta_{\Delta'}(I_{\mathfrak n}\circ D_{\Delta}(\pi)) ,
\]
and so we have the strong commutativity for $(\Delta', \mathfrak n, D_{\Delta}(\pi))$ by Lemma \ref{lem strong commutative combin}. Thus, we inductively have that $(\mathfrak p-\Delta, \mathfrak n, D_{\Delta}(\pi))$ is still strongly RdLi-commutative triple. Thus, this combines to have $(\mathfrak p, \mathfrak n, \pi)$ to be a strongly RdLi-commutative triple by definitions. This proves (1).
\end{proof}

\subsection{Reduced triples} \label{ss reduced and saturated parts}

For segments $\Delta_1, \Delta_2$, we write $\Delta_1 \prec^R \Delta_2$ if either $b(\Delta_1)<b(\Delta_2)$; or $b(\Delta_1)\cong b(\Delta_2)$ and $a(\Delta_2) <a(\Delta_1)$. We write $\Delta_1\preceq^R \Delta_2$ if $\Delta_1\prec^R\Delta_2$ or $\Delta_1=\Delta_2$. In particular, if $\Delta$ is a $\preceq^R$-maximal segment in $\mathfrak m \in \mathrm{Mult}$, then $b(\Delta)$ is $\leq$-maximal in $\left\{ b(\widetilde{\Delta}): \widetilde{\Delta} \in \mathfrak m \right\}$ and $\Delta$ is the longest segment in $\mathfrak m_{b=b(\Delta)}$. 

\begin{definition}
Let $\mathfrak m \in \mathrm{Mult}$. Let $\Delta$ be a segment. Let $\Delta$ be a $\preceq^R$-maximal segment in $\mathfrak m$. We say that $(\mathfrak m, \Delta, \pi)$ is a {\it R-reduced triple} if $\eta_{\Delta}(D_{\mathfrak m}(\pi))=0$ and $\mathfrak m$ is minimal to $\pi$. 
\end{definition}

We first prove a slightly technical lemma.

\begin{lemma} \label{lem delta reduced subset}
Suppose $(\mathfrak m, \Delta, \pi)$ is a R-reduced triple. Let $\mathfrak p=\mathfrak{mx}(\pi, \Delta)$. Then $\mathfrak p \subset \mathfrak m$. 
\end{lemma}

\begin{proof}
 We write $\mathfrak m$ in an ascending sequence: $\Delta_1, \ldots, \Delta_k$. When there is only one segment in $\mathfrak m$, $\mathfrak p=\mathfrak m=\left\{ \Delta \right\}$ from $R$-reducedness of $(\mathfrak m, \Delta, \pi)$. Then we are done. We now proceed on the number of segments in $\mathfrak m$.

If $\Delta_1$ is $\Delta$-saturated, then $\mathfrak{mx}(D_{\Delta_1}(\pi), \Delta)=\mathfrak{mx}(\pi, \Delta)-\Delta_1=\mathfrak p-\Delta_1$. (This can be proved by either using the removal sequence or some direct computations of the geometric lemma.) Then, by induction on the number of segments in $\mathfrak m$, we have 
\[  \mathfrak p-\Delta_1 \subset \mathfrak m-\Delta_1 
\]
as desired.

Suppose $\Delta_1$ is not $\Delta$-saturated. Thus, we have $\eta_{\Delta}(D_{\Delta_1}(\pi))=\eta_{\Delta}(\pi)$ by considering the following cases:
\begin{enumerate}
\item If $b(\Delta_1)$ is not a $\nu$-integral shift of $b(\Delta)$, this case is simple.
\item If $b(\Delta_1)<b(\Delta)$ and $a(\Delta_1)<a(\Delta)$, then, by the subsequent property (Lemma \ref{lem commute and minimal}), $\Delta_1+\Delta$ is minimal to $\pi$. Then the equality follows by using minimality and Lemma \ref{lem minimal eta criterai}.
\item If $b(\Delta_1)\cong b(\Delta)$, then one uses Proposition \ref{prop eta function derivative}(1).
\item If $b(\Delta_1)<b(\Delta)$ and $a(\Delta) \leq a(\Delta_1)$, then one further considers two cases: let $\widetilde{\Delta}$ be the first segment in the removal sequence for $(\Delta_1, \mathfrak{hd}(\pi))$.
\begin{itemize}
\item If $b(\widetilde{\Delta})< b(\Delta)$, then the equality follows from Proposition \ref{prop removal sequence for intermediate case}(1);
\item If $b(\widetilde{\Delta}) \geq b(\Delta)$, then Proposition \ref{prop removal sequence for intermediate case}(2) gives a certain segment $\overline{\Delta}$ satisfying the properties in that proposition. Then, by using $\overline{\Delta}$-reducedness (from $\Delta$-reducedness) for $D_{\mathfrak m}(\pi)$ and minimality of $\mathfrak m-\Delta_1$ to $D_{\Delta_1}(\pi)$, induction gives that $\mathfrak{mx}(D_{\Delta_1}(\pi), \overline{\Delta}) \subset \mathfrak m-\Delta_1$. In particular, $\overline{\Delta} \in \mathfrak m$ (by Proposition \ref{prop removal sequence for intermediate case}(2)). Then, $\Delta_1+\overline{\Delta}$ is minimal to $\pi$ by the subsequent property (Lemma \ref{lem commute and minimal}) and the inequality in Proposition \ref{prop removal sequence for intermediate case}(2) then contradicts to Lemma \ref{lem minimal eta criterai}. (In other words, this case could not happen.)
\end{itemize}
\end{enumerate}
 This implies $\mathfrak{mx}(D_{\Delta_1}(\pi), \Delta)=\mathfrak p$. Then, by induction on the number of segments in $\mathfrak m$, we have 
\[ \mathfrak p \subset \mathfrak m-\Delta_1 \subset \mathfrak m.
\]

\end{proof}

\subsection{More minimality and commutativity results}

We shall now show more minimality and commutativity results, and some are useful when we prove the necessity part and the exhaustion part in Part \ref{part necessity part}.

\begin{lemma} \label{lem commute max rho}
Let $\pi \in \mathrm{Irr}$. Let $\rho$ be a $\leq$-maximal element in $\mathrm{csupp}(\pi)$ and let $\mathfrak p=\mathfrak{mxpt}(\pi, \rho)$. Let $\Delta$ be a segment such that $\Delta+\mathfrak p$ is admissible to $\pi$. Then $\Delta+\mathfrak p$ is minimal to $\pi$ and in particular, 
\[   D_{\Delta+\mathfrak p}(\pi) \cong D_{\mathfrak p}\circ D_{\Delta}(\pi) \cong D_{\Delta}\circ D_{\mathfrak p}(\pi) .
\]
\end{lemma}

\begin{proof}
We find a minimal multisegment $\mathfrak n$ such that $D_{\mathfrak n}(\pi) \cong D_{\mathfrak p+\Delta}(\pi)$. Suppose $\mathfrak n$ is not minimal. Then $\mathfrak n$ is obtained from $\mathfrak p+\Delta$ by an intersection-union process between a segment in $\mathfrak p$ and $\Delta$. However, by Lemma \ref{lem commute and minimal}, 
\[  D_{\mathfrak n_{b=\rho}}(\pi) \neq 0  \]
and so this contradicts to the maximality of $\mathfrak{mxpt}(\pi, \rho)$.

For the second assertion, the first isomorphism follows from the asecending ordering for $\Delta+\mathfrak p$, and the second isomorphism follows from Lemma \ref{lem commute and minimal}.
\end{proof}

\begin{lemma} \label{lem commute max rho 2}
Let $\pi \in \mathrm{Irr}$. Let $\rho$ be a $\leq$-maximal element in $\mathrm{csupp}(\pi)$ and let $\mathfrak p=\mathfrak{mxpt}(\pi, \rho)$. Let $\Delta$ be a segment such that $\Delta+\mathfrak p$ is admissible to $\pi$. Then 
\[   \mathfrak{mxpt}(D_{\Delta}(\pi), \rho) =\mathfrak p.
\]
\end{lemma}

\begin{proof}
Since $\Delta+\mathfrak p$ is admissible to $\pi$, we must have that $b(\Delta)\not\geq \rho$. Now $D_{\mathfrak p}(D_{\Delta}(\pi))\neq 0$ implies that $\mathfrak p\subset \mathfrak{mxpt}(D_{\Delta}(\pi), \rho)$. By \cite[Theorem 9.3]{Ch22+}, $\mathfrak{mxpt}(D_{\Delta}(\pi), \rho)=\mathfrak r(\Delta, \pi)_{b=\rho}$.  Now it follows from the removal process for $\mathfrak r(\Delta, \pi)$, the only possibility is that $\mathfrak{mxpt}(D_{\Delta}(\pi), \rho) =\mathfrak p$.
\end{proof}

\begin{corollary} \label{cor commute max rho pre} (c.f. \cite[Proposition 15.2]{Ch22+bb})
Let $\pi \in \mathrm{Irr}$. Let $\rho$ be a $\leq$-maximal element in $\mathrm{csupp}(\pi)$ and let $\mathfrak p=\mathfrak{mxpt}(\pi, \rho)$. Let $\mathfrak m$ be a multisegment such that $\mathfrak m+\mathfrak p$ is admissible to $\pi$. Then 
\[ D_{\mathfrak p+\mathfrak m}(\pi)\cong D_{\mathfrak m}\circ D_{\mathfrak p}(\pi) \neq 0. \]
Futhermore, if $\mathfrak m$ is minimal to $D_{\mathfrak p}(\pi)$, then $\mathfrak m$ is minimal to $\pi$.
\end{corollary}

\begin{proof}
Let $\mathfrak m=\left\{ \Delta_1, \ldots , \Delta_r\right\}$ be in an ascending order. Note that if $\mathfrak m+\mathfrak p$ is admissible to $\pi$, then $\left\{ \Delta_1, \ldots, \Delta_i\right\}+\mathfrak p$ is admissible to $\pi$ for any $i$. 

Now, by repeatedly applying Lemmas \ref{lem commute max rho} and \ref{lem commute max rho 2}, we have the first assertion. For the second assertion, suppose $\mathfrak n$ is minimal to $\pi$ such that $D_{\mathfrak n}(\pi)\cong D_{\mathfrak m}(\pi)$. Then, the same reasoning using Lemmas \ref{lem commute max rho} and \ref{lem commute max rho 2}, one has 
\[   D_{\mathfrak p}\circ D_{\mathfrak n}(\pi) \cong D_{\mathfrak n}\circ D_{\mathfrak p}(\pi) \neq 0 .
  \]
Hence, combining equations, we have
\[  D_{\mathfrak n}\circ D_{\mathfrak p}(\pi) \cong D_{\mathfrak m}\circ D_{\mathfrak p}(\pi) .
\]
Since $\mathfrak n$ is obtained from $\mathfrak m$ by a sequence of intersection-union processes, the minimality of $\mathfrak m$ to $D_{\mathfrak p}(\pi)$ and the uniquenss in Theorem \ref{thm minimality of d and i} imply that $\mathfrak n=\mathfrak m$ as desired.
\end{proof}

\subsection{Consequences on commutative triples}

\begin{lemma} \label{lem induced strongly commutate}
Let $\pi \in \mathrm{Irr}$. Let $\rho$ be a $\leq$-maximal element in $\mathrm{csupp}(\pi)$ and let $\mathfrak p=\mathfrak{mxpt}(\pi, \rho)$. Let $(\Delta, \Delta', D_{\mathfrak p}(\pi))$ be a minimal strongly RdLi-commutative triple such that for any $\rho' \in \mathrm{csupp}(\Delta')$, $\rho'\not\geq \rho$. Suppose $\Delta+\mathfrak p$ is admissible to $\pi$. Then $(\Delta+\mathfrak p, \Delta', \pi)$ is a minimal strongly RdLi-commutative triple.
\end{lemma}

\begin{proof}
Let $\tau=D_{\Delta+\mathfrak p} \circ I_{\Delta'}(\pi)$. From Example  \ref{example pre commutative}  and Lemma \ref{lem strong commutative combin}, 
\[ \mathfrak{mxpt}(I_{\Delta'}(\pi), \rho)=\mathfrak p . \]
We also have that $\Delta+\mathfrak p$ is admissible to $\pi$ and so is admissible to $I_{\Delta'}(\pi)$. Thus, we also have:
\begin{align} \label{eqn commute derivatives}
\tau \cong D_{\Delta+\mathfrak p}\circ I_{\Delta'}(\pi) \cong D_{\mathfrak p}\circ D_{\Delta}\circ I_{\Delta'}(\pi) \cong D_{\Delta}\circ D_{\mathfrak p}\circ I_{\Delta'}(\pi) . 
\end{align}

The strong commutativity condition implies that:
\[   \eta^L_{\Delta'}(I^R_{\Delta}(\tau))=\eta^L_{\Delta'}(\tau) .\]
On the other hand, $(\mathfrak p, \Delta', I^R_{\Delta}(\tau))$ is LdRi-commutative triple by the LdRi-version of Example \ref{example pre commutative} and Theorem \ref{thm pre implies strong}. Hence, by Theorem \ref{thm combinatorial def},
\[  \eta^L_{\Delta'}(I^R_{\mathfrak p}\circ I^R_{\Delta}(\tau))=\eta^L_{\Delta'}(I^R_{\Delta}(\tau)) .
\]
Similarly, 
\[  \eta^L_{\Delta'}(I^R_{\mathfrak p}(\tau))=\eta^L_{\Delta'}(\tau) .
\]
Thus, we have $\eta^L_{\Delta'}(I^R_{\mathfrak p}(\tau))=\eta^L_{\Delta'}(I_{\mathfrak p}^R\circ I_{\Delta}^R(\tau))$. 

Combining with (\ref{eqn commute derivatives}), we have:
\[  \eta^L_{\Delta'}(D_{\Delta}\circ I_{\Delta'}(\pi))=\eta^L_{\Delta'}(I_{\Delta'}(\pi)) .\]
By Theorem \ref{thm combinatorial def}, we have that $(\Delta, \Delta', \pi)$ is a RdLi-commutative triple. To show $(\Delta+\mathfrak p, \Delta', \pi)$ is a RdLi-commutative triple, one again applies Example \ref{example pre commutative} and Theorem \ref{thm pre implies strong} several times. This completes the proof.
\end{proof}

\begin{corollary} \label{cor commute max rho}
Let $\pi \in \mathrm{Irr}$. Let $\rho$ be a $\leq$-maximal element in $\mathrm{csupp}(\pi)$ and let $\mathfrak p=\mathfrak{mxpt}(\pi, \rho)$. Let $(\mathfrak m, \mathfrak n, D_{\mathfrak p}(\pi))$ be a minimal strongly RdLi-commutative triple such that for any $\rho' \in \mathrm{csupp}(\mathfrak n)$, $\rho'\not\geq \rho$. Suppose $\mathfrak m+\mathfrak p$ is admissible to $\pi$. Then $(\mathfrak m+\mathfrak p, \mathfrak n, \pi)$ is also a minimal strongly RdLi-commutative triple. 
\end{corollary}

\begin{proof}
This follows from a multiple use of Proposition \ref{cor strong commute linked commute case} and Corollary \ref{cor commute max rho pre}. In more details, write
\[  \mathfrak m=\left\{ \Delta_1, \ldots, \Delta_r\right\}
\]
in an ascending order. Let $\mathfrak m_i=\left\{ \Delta_1, \ldots, \Delta_{i-1}\right\}$. By Corollary \ref{cor commute max rho pre}, we have $\mathfrak m_{i}+\Delta_i+\mathfrak p$ is admissible to $\pi$, and so by Corollary \ref{cor commute max rho pre}, $\mathfrak m_i+\Delta_i$ is also minimal to $\pi$.

On the other hand, we also have 
\begin{align} \label{eqn commute multisegments} 
D_{\mathfrak p}\circ D_{\mathfrak m_i}(\pi) \cong D_{\mathfrak m_i}\circ D_{\mathfrak p}(\pi). 
\end{align} 

Now, we have the following:
\begin{enumerate}
\item $(\mathfrak p, \mathfrak n, D_{\mathfrak m_i}(\pi))$ is a strongly RdLi-commutative triple by Example \ref{example pre commutative}(2);
\item $(\Delta_i, \mathfrak n, D_{\mathfrak p}\circ D_{\mathfrak m_i}(\pi))$ is a strongly RdLi-commutative triple by given assumption and (\ref{eqn commute multisegments});
\item $\Delta_i+\mathfrak p$ is minimal to $D_{\mathfrak m_i}(\pi)$ by above discussion.
\end{enumerate}

In other words, $(\mathfrak m_i+\mathfrak p, \mathfrak n, \pi)$ is a strongly RdLi-commutative triple. Now, by Corollaries \ref{cor seq strong commut under minimal and sub} and \ref{cor seq strong commut under minimal}, $(\Delta_i, \mathfrak n, D_{\mathfrak m_i}(\pi))$ is a strongly RdLi-commutative triple. 
\end{proof}




\section{Completing to $\Delta$-reduced triples} \label{s complete to reduced}

Recall that the notion of reducedness and saturatedness is defined in Section \ref{ss standard trick intro}. For $\pi \in \mathrm{Irr}$ and a segment $\Delta$, let $\mathfrak p=\mathfrak{mx}(\pi, \Delta)$ and one has an embedding
\[  \pi  \hookrightarrow D_{\mathfrak p}(\pi)\times \mathrm{St}(\mathfrak p) 
\]
(see e.g. \cite[Section 4]{Ch22+}, \cite[Section 7]{LM22}). As also seen in \cite{Ch22+b}, such the decomposition is particularly useful in some reduction process. We first explain how to do such reduction in our context of strong commutativity.

The main idea is that if one starts with a strongly commutative triple $(\mathfrak m, \mathfrak n, \pi)$, one can then easily produce another strongly commutative triple by adding a segment on $\mathfrak m$ (with $\Delta_1, \ldots, \Delta_k$ in an ascending ordering) so that the ordering $\Delta_1, \ldots, \Delta_k, \Delta$ is still ascending, and $(\Delta, \mathfrak n, D_{\mathfrak m}(\pi))$ is still a strongly commutative triple. For this, one may use Lemma \ref{lem produce more strong rdli commutative triples}. Then one uses the intersection-union processes to produce a strongly commutative triple, and this is addressed in Corollary \ref{cor reduced one strong commutation}.

\subsection{Reduced multisegments}

\begin{definition}
Let $\mathfrak m \in \mathrm{Mult}$ and let $\pi \in \mathrm{Irr}$ with $D_{\mathfrak m}(\pi)\neq 0$. We say that $\mathfrak n$ is the {\it minimized multisegment} for $(\mathfrak m, \pi)$ if $\mathfrak n$ is minimal to $\pi$ and $D_{\mathfrak n}(\pi)\cong D_{\mathfrak m}(\pi)$. (Again, the uniqueness of minimality is shown in \cite[Theorem 1.2]{Ch22+aa}.)
\end{definition}

\begin{definition}
Let $\mathfrak m \in \mathrm{Mult}$. Let $\pi \in \mathrm{Irr}$ and let $\Delta$ be a $\preceq^R$-maximal segment in $\mathfrak m$. We say that a $\Delta$-saturated multisegment $\mathfrak r$ is {\it completing $(\mathfrak m, \Delta, \pi)$} if $D_{\mathfrak m+\mathfrak r}(\pi)$ is $\Delta$-reduced. 

We say that $\mathfrak q$ is the {\it reduced multisegment for $(\mathfrak m, \Delta, \pi)$} if $\mathfrak q+\mathfrak{mx}(\pi, \Delta)$ is the minimized multisegment for $(\mathfrak m+\mathfrak r, \pi)$. (Note that the definition is well-defined by Lemma \ref{lem delta reduced subset}.)
\end{definition}

\begin{example}
Let $\pi=\mathrm{St}([0,4]+[1,3])$. Let $\mathfrak m=[0,2]+[1,3]$. The minimized multisegment $\mathfrak n$ for $(\mathfrak m, \pi)$ is $[0,3]+[1,2]$. Let $\Delta=[0,3]$. Then $[3]$ is completing $(\mathfrak m, \Delta, \pi)$. The reduced multisegment for $(\mathfrak m, \Delta, \pi)$ is $\emptyset$. 
\end{example}

\begin{corollary} \label{cor reduced one strong commutation}
Let $(\mathfrak m, \mathfrak n, \pi)$ be a strongly RdLi-commutative triple. Let $\Delta$ be a $\preceq^R$-maximal segment in $\mathfrak m$. Let $\mathfrak p=\mathfrak{mx}(\pi, \Delta)$ and let $\mathfrak q$ be the reduced multisegment for $(\mathfrak m, \Delta, \pi)$. Then $(\mathfrak q, \mathfrak n, D_{\mathfrak p}(\pi))$ is a strongly RdLi-commutative triple. 
\end{corollary}

\begin{proof}
Let $\mathfrak s=\mathfrak{mx}(D_{\mathfrak m}(\pi), \Delta)$. Then, by Lemma \ref{lem produce more strong rdli commutative triples}, $(\mathfrak s, \mathfrak n, D_{\mathfrak m}(\pi))$ is a strongly RdLi-commutative triple and so $(\mathfrak s+\mathfrak m, \mathfrak n, \pi)$ is also a strongly RdLi-commutative triple. Thus, if $\mathfrak s+\mathfrak m$ is not minimal to $\pi$, the closedness property in \cite[Theorem 1.1]{Ch22+aa} implies that one can find a consecutive pair (see \cite{Ch22+aa}) for doing the intersection-union operation to get a new sequence, which still gives a strongly RdLi-commutative triple by Proposition \ref{prop two strong commute dual 2}. Repeatedly using this reasoning, we have that $(\mathfrak p+\mathfrak q, \mathfrak n, \pi)$ is also a strongly RdLi-commutative triple. By Corollary \ref{cor seq strong commut under minimal} (or more directly by Corollary \ref{cor commute minimal strong rdli comm}), we then have that $(\mathfrak q, \mathfrak n, D_{\mathfrak p}(\pi))$ is also a strongly RdLi-commutative triple.
\end{proof}

\subsection{Producing a sequence of strong commutations} \label{ss produce strong commute}

Let $\pi \in \mathrm{Irr}$. Let $\mathfrak m \in \mathrm{Mult}$ be minimal to $\pi$. Let $\mathfrak q_0=\mathfrak m$ and let $\pi_0=\pi$. We recursively produce the following data: for $i \geq 1$, 
\begin{enumerate}
\item Let $\Delta_i$ be a $\preceq^R$-maximal segment in $\mathfrak q_{i-1}$.
\item Let $\mathfrak p_i =\mathfrak{mx}(\pi_{i-1}, \Delta_i)$ and let $\mathfrak q_i$ be a reduced multisegment for $(\mathfrak q_{i-1}, \Delta_i, \pi_{i-1})$.
\item Set $\pi_i=D_{\mathfrak p_i}(\pi_{i-1})$. 
\end{enumerate}
The process terminates when $\mathfrak q_k=\emptyset$. Note that the data generated above depend on the choices of maximal segments in $\mathfrak q_{i-1}$. In the above data, we also let $\mathfrak r_i$ be the multisegment for minimizing $(\mathfrak q_{i-1}, \Delta_i, \pi_{i-1})$. We shall call
\[  (\mathfrak r_1, \mathfrak p_1, \mathfrak q_1, \pi_1), \ldots, (\mathfrak r_k, \mathfrak p_k, \mathfrak q_k, \pi_k) 
\]
to be a {\it sequence of minimized data} for $(\mathfrak m, \pi)$.

\begin{lemma} \label{lem strong commut triples}
Let $(\mathfrak m, \mathfrak n, \pi)$ be a strongly RdLi-commutative triple. Denote a sequence of minimized data for $(\mathfrak m, \pi)$ by:
\[ (\mathfrak r_1, \mathfrak p_1, \mathfrak q_1, \pi_1), \ldots, (\mathfrak r_k, \mathfrak p_k, \mathfrak q_k, \pi_k)  \]
(with $\mathfrak q_k=\emptyset$). Then 
\[  (\mathfrak q_1, \mathfrak n, \pi_1), \ldots, (\mathfrak q_{k-1}, \mathfrak n, \pi_{k-1})
\]
are strongly RdLi-commutative triples.
\end{lemma}

\begin{proof}
This follows multiple uses of Corollary \ref{cor reduced one strong commutation}.
\end{proof}

\begin{lemma} \label{lem length of m in terms of p r}
We use the notations in the previous lemma. Then $l_a(\mathfrak m)=\sum_x l_a(\mathfrak p_x)-\sum_y l_a(\mathfrak r_y)$. 
\end{lemma}

\begin{proof}
We have that $l_a(\mathfrak m)+l_a(\mathfrak r_1)=l_a(\mathfrak p_1)+l_a(\mathfrak q_1)$ and $l_a(\mathfrak q_i)+l_a(\mathfrak r_{i+1})=l_a(\mathfrak p_{i+1})+l_a(\mathfrak q_{i+1})$. Note that $l_a(\mathfrak q_k)=0$. Combining the equations, we have the lemma.
\end{proof}

\begin{lemma} \label{lem some expression in reduced multisegments}
Let $\tau_0=I_{\mathfrak n}\circ D_{\mathfrak m}(\pi)$. Let $\tau_i=D_{\mathfrak r_i} \circ \ldots \circ D_{\mathfrak r_1}(\tau_0)$. Then 
\[  \tau_i \cong I_{\mathfrak n}\circ D_{\mathfrak q_i}(\pi_i) \cong I_{\mathfrak n}\circ D_{\mathfrak p_i+\mathfrak q_i}(\pi_{i-1}).
\]
Moreover, $\mathfrak{mx}(\Delta_i, \tau_i)=\emptyset$, where $\Delta_i$ is the segment involved in the minimizing process.
\end{lemma}

\begin{proof}
We shall prove inductively on $i$. By definition, $\tau_i=D_{\mathfrak r_i}(\tau_{i-1})$ and so we have $\tau_i=D_{\mathfrak r_i}\circ I_{\mathfrak n}\circ D_{\mathfrak q_{i-1}}(\pi_{i-1})$. 

We have the strong commutation for $(\mathfrak r_i, \mathfrak n, \pi_{i-1})$ by Lemma \ref{lem strong commut triples} and so we have the strong commutation for $(\mathfrak r_i, \mathfrak n, D_{\mathfrak q_{i-1}}(\pi_{i-1}))$ by Lemma \ref{lem produce more strong rdli commutative triples}. Thus $\tau_i=I_{\mathfrak n}\circ D_{\mathfrak r_i}\circ D_{\mathfrak q_{i-1}}(\pi_{i-1})$ and so, by $\mathfrak r_i+\mathfrak q_{i-1}=\mathfrak p_i+\mathfrak q_i$, we have:
\[   \tau_i = I_{\mathfrak n}\circ D_{\mathfrak p_i}\circ D_{\mathfrak q_i}(\pi_{i-1}) .
\]
Now, by Lemma \ref{lem commute and minimal}, 
\[  \tau_i=I_{\mathfrak n}\circ D_{\mathfrak q_i}\circ D_{\mathfrak p_i}(\pi_{i-1})=I_{\mathfrak n}\circ D_{\mathfrak q_i}(\pi_i) .
\]

It remains to prove the last assertion. Note that $\mathfrak{mx}(\Delta_i, D_{\mathfrak q_i}(\pi_{i-1}))=\mathfrak p_i$ and so 
\[\mathfrak{mx}(\Delta_i, I_{\mathfrak n}\circ D_{\mathfrak q_i}(\pi_{i-1}))=\mathfrak p_i \]
 by the subsequent property of Corollary \ref{cor seq strong commut under minimal and sub} and Theorem \ref{thm combinatorial def}. This implies that 
\[  \mathfrak{mx}(\Delta_i, I_{\mathfrak n}\circ D_{\mathfrak p_i}\circ D_{\mathfrak q_i}(\pi_{i-1}))=\mathfrak{mx}(\Delta_i, D_{\mathfrak p_i}\circ I_{\mathfrak n}\circ D_{\mathfrak q_i}(\pi_{i-1}))=\emptyset ,
\]
where the equality follows from Proposition \ref{prop strong commute imply commute}. Now the last assertion follows from the above expression of $\tau_i$.
\end{proof}

The following lemma follows from Frobenius reciprocity:

\begin{lemma} \label{lem inductive embeddings in completions}
We use the notations above. Let $\tau_i=I_{\mathfrak n}\circ D_{\mathfrak q_i}(\tau_{i-1})$. Then 
\[        \pi_{i-1} \hookrightarrow \pi_i \times \mathrm{St}(\mathfrak p_{i}) ,\quad       \tau_{i-1} \hookrightarrow \tau_i \times \mathrm{St}(\mathfrak r_i) .\]
\end{lemma}

\section{Smallest derivatives for a strongly commutative triple} \label{s smallest derivative strongly}

We illustrate the idea of computing the smallest derivative in a segment case. Let $\pi'=D_{\Delta_1}(\pi)$. Our goal is to determine the smallest integer $i^*$ such that $\mathrm{Hom}(\pi^{(i^*)},{}^{(j)}\pi') \neq 0$. Roughly speaking, one compares $\eta_{\Delta_1}(\pi)$ and $\eta_{\Delta_1}({}^{(j)}\pi') \leq \eta_{\Delta_1}(\pi')$. The $\eta_{\Delta_1}(\pi)$ and $\eta_{\Delta_1}(\pi')$ are differed by $1$ on the coordinate for $\varepsilon_{\Delta_1}$. Thus in order to get from $\pi$ to $\pi'$, one has to take at least $l_a(\Delta_1)$-derivative. Such idea can be extended to a strongly commutative triple $(\Delta_1, \Delta_2, \pi)$ due to the combinatorial criteria (Definition \ref{def combinatorial comm triple}), and then to strongly commutative triples $(\mathfrak m, \mathfrak n, \pi)$, but then one needs tools from Section \ref{s complete to reduced}.

\subsection{Leibniz's rule}

So far, we mainly deal with the derivatives $D_{\Delta}$. We now start to discuss more on BZ derivatives. Again, a standard tool is the geometric lemma, which we shall also refer to Leibniz's rule. 

For a $G_{n_1}$-representation $\pi_1$ and a $G_{n_2}$-representation $\pi_2$, Leibniz's rule asserts that $(\pi_1\times \pi_2)^{(i)}$ (resp. ${}^{(i)}(\pi_1 \times \pi_2)$) admits a filtration with layers isomorphic to 
\[   \pi_1^{(i_1)}\times \pi_2^{(i_2)}, \quad \mbox{(resp. ${}^{(i_1)}\pi_1\times {}^{(i_2)}\pi_2$)}
\]
where $i_1$ and $i_2$ run for all non-negative integers satisfying $i_1+i_2=i$. 

\subsection{A lemma} 

\begin{lemma} \label{lem inequality on delta red}
Let $\Delta$ be a segment. Let $\omega \in \mathrm{Irr}$ be $\Delta$-reduced and let $\tau \in \mathrm{Irr}$. Let $\mathfrak p \in \mathrm{Mult}$ be $\Delta$-saturated. Let $\lambda \in \mathrm{Alg}(G_n)$. Suppose $\mathrm{Hom}( \mathrm{St}(\mathfrak p) \times  \lambda, \omega \times \tau) \neq 0$. Then $l_a(\mathfrak p)\leq n(\tau)$. Moreover, the inequality is strict if there exists $\rho$ in $\mathrm{csupp}(\tau)$, but not in $\mathrm{csupp}(\mathfrak p)$. 
\end{lemma}

\begin{proof}
Applying Frobenius reciprocity, we have that:
\[   \mathrm{Hom}( \pi \boxtimes \mathrm{St}(\mathfrak p) , (\omega \times \tau)_{N_{n(\pi), l_a(\mathfrak p)}} ) \neq 0 . 
\]
Since $\omega$ is $\Delta$-reduced, the only layer in the geometric lemma that can contribute the Hom is $\omega \dot{\times}^1 (\tau_{N_{l_a(\mathfrak p)}})$. But this then implies the required inequality.

The strict part of the inequality follows from $\mathrm{csupp}(\mathfrak p)\subset \mathrm{csupp}(\tau)$ for getting a contribution of that layer to a non-zero Hom.
\end{proof}

\begin{proposition} \label{prop smallest integer in strong triple}
Let $(\mathfrak m, \mathfrak n, \pi)$ be a strongly RdLi-commutative triple. Let $s_1=l_a(\mathfrak m)$, let $s_2=l_a(\mathfrak n)$. Let $\tau= I_{\mathfrak n}\circ D_{\mathfrak m}(\pi)$. Let $n=n(\pi)$. Then $n(\tau)=n-s_1+s_2$. Then the smallest integer $i$ such that 
\[  \mathrm{Hom}_{G_{n-i}}(\pi^{(i)}, {}^{(i-s_1+s_2)}\tau) \neq 0
\]
is $l_a(\mathfrak m)$.
\end{proposition}

\begin{proof}
We first remark that $\mathfrak n$ does not play much role in the proof, thanks to the properties of strongly RdLi-commutative triples.

Let a sequence of minimized data for $(\mathfrak m, \pi)$ be:
\[ (\mathfrak r_1, \mathfrak p_1, \mathfrak q_1, \pi_1), \ldots, (\mathfrak r_k, \mathfrak p_k, \mathfrak q_k, \pi_k) .
\] 
Let $n'=n(\tau)$. Let $\Pi=\Pi_{i-s_1+s_2}$. We first rewrite the inequality, by using the second adjointness, as:
\[  \mathrm{Hom}_{G_{n'}}(\pi^{(i)} \times \Pi, \tau) \neq 0. \]
Then, using Lemma \ref{lem inductive embeddings in completions} (the reducedness part of $\tau_1$ follows from Lemma \ref{lem some expression in reduced multisegments}),
\[  \mathrm{Hom}_{G_{n'}}( (\mathrm{St}(\mathfrak p_1)\times \pi_1)^{(i)} \times \Pi, \tau_1\times \mathrm{St}(\mathfrak r_1)) \neq 0 . \]
Thus, there exists $i_1$ such that 
\[  \mathrm{Hom} (\mathrm{St}(\mathfrak p_1)^{(i_1)} \times \pi_1^{(i-i_1)} \times \Pi, \tau_1 \times \mathrm{St}(\mathfrak r_1) ) \neq 0 .
\]
Then we have that $i_1 \geq l_a(\mathfrak p_1)-l_a(\mathfrak r_1)$ by Lemma \ref{lem inequality on delta red}. 

Now, by Frobenius reciprocity again, we have:
\[ \mathrm{Hom}( \pi_1^{(i-i_1)} \times \Pi\boxtimes \mathrm{St}(\mathfrak p_1)^{(i_1)}, (\tau_1\times \mathrm{St}(\mathfrak r_1))_{N}) \neq 0, 
\]
where $N=N_{l_a(\mathfrak p_1)-i_1}$. Using the reducedness of $\tau_1$, we then have that: 
\[  \mathrm{Hom}(\pi_1^{(i-i_1)} \times \Pi \boxtimes \mathrm{St}(\mathfrak p_1)^{(i_1)}, \tau_1 \dot{\times}^1 (\mathrm{St}(\mathfrak r_1)_{N_{l_a(\mathfrak p_1)-i_1}}) \neq 0 .
\]
Using adjointness, we then have that:
\[ \mathrm{Hom}(\pi_1^{(i-i_1)} \times \Pi, \tau_1 \times \kappa_1)\neq 0,
\]
for some $\kappa_1$ in $\mathrm{Alg}(G_{l_a(\mathfrak r_1)-l_a(\mathfrak p_1)+i_1})$.  

We now use Lemma \ref{lem inductive embeddings in completions} again to obtain:
\[  \mathrm{Hom}((\mathrm{St}(\mathfrak p_2) \times \pi_2)^{(i-i_1)} \times \Pi, \tau_2 \times \mathrm{St}(\mathfrak r_2) \times \kappa_1) \neq 0. 
\]
With the same reasoning as above, we have that: there exists $i_2$ such that
\[            l_a(\mathfrak r_2)+l_a(\mathfrak r_1)-(l_a(\mathfrak p_1)-i_1)         \geq l_a(\mathfrak p_2) -i_2,
\]
(here $l_a(\mathfrak r_2)$ comes from the largest possible contribution in a Jacquet module of $\mathrm{St}(\mathfrak r_2)$, and $l_a(\mathfrak r_1)-(l_a(\mathfrak p)-i_1)$ comes from the largest possible contribution in a Jacquet module from $\kappa_1$), equivalently
\[ i_2 \geq l_a(\mathfrak p_2)-l_a(\mathfrak r_2)+l_a(\mathfrak p_1)-l_a(\mathfrak r_1)-i_1 \]
such that
\[  \mathrm{Hom}( \pi_2^{(i-i_1-i_2)} \times \Pi, \tau_2 \times \kappa_2 ) \neq 0
\]
for some $\kappa_2 \in \mathrm{Irr}(G_{l_a(\mathfrak r_2)+l_a(\mathfrak r_1)-l_a(\mathfrak p_1)-l_a(\mathfrak p_2)+i_1+i_2})$. 

Now inductively, we have that:
\[   i_k \geq \sum_x l_a(\mathfrak p_x) -\sum_y l_a(\mathfrak r_y) -i_1-\ldots -i_{k-1} .
\]
Thus, $i \geq i_1+\ldots +i_k \geq l_a(\mathfrak m)$ by Lemma \ref{lem length of m in terms of p r}. 
\end{proof}

\begin{corollary} \label{cor strict inequality for integer db}
We use the notations in Proposition \ref{prop smallest integer in strong triple}, and in particular, let $(\mathfrak m, \mathfrak n, \pi)$ be a strongly RdLi-commutative triple. Let $(\mathfrak r_1, \mathfrak p_1, \mathfrak q_1, D_{\mathfrak p_1}(\pi))$ be the first term of a sequence of minimized data for $(\mathfrak m, \pi)$ in Section \ref{ss produce strong commute}. Let $\tau=I_{\mathfrak n}\circ D_{\mathfrak m}(\pi)$. Suppose
\[ \mathrm{Hom}_{G_{n-i'-i''}}(\mathrm{St}(\mathfrak p_1)^{(i')} \times (D_{\mathfrak p_1}(\pi))^{(i'')}, {}^{(i-s_1+s_2)}\tau) \neq 0
\]
for some $i' >l_a(\mathfrak p_1)-l_a(\mathfrak r_1)$ and some $i''$ with $i'+i''=i$. Then $i'+i''>l_a(\mathfrak m)$. 
\end{corollary}

\begin{proof}
It follows from the previous proof that 
\[   i_k \geq l_a(\mathfrak p_k)-\left(\sum l_a(\mathfrak r_x) -\sum_{y=1}^{k-1} (l_a(\mathfrak p_y)-i_y) \right)
\]
If $k=1$, there is nothing to prove. Now, for $k \geq 2$, our choice guarantees that $b(\Delta) \not\geq \rho$ for any $ \Delta \in \mathfrak q_k$. Thus the inequality is strict by looking at a cuspidal representation in $\mathrm{csupp}(\kappa_k)$ (see the notations in Proposition \ref{prop smallest integer in strong triple} and $\kappa_k$ is defined analogously) which contains $b(\Delta)$. 
\end{proof}

\section{Characterizing the supporting layer in BZ filtration} \label{s characterize bz filtration for bl}

\subsection{Multiplicity-one results}

We first recall the following multiplicity one result:

\begin{theorem} \cite[Theorem 1.1]{Ch21+} \label{thm standard multiplicity one}
Let $\lambda$ be a standard module of $G_{n+1}$ and let $\lambda'$ be a standard module of $G_n$. Then 
\[  \mathrm{Hom}_{G_n}(\lambda, \lambda'{}^{\vee}) \cong \mathbb C.
 \]

\end{theorem}

This immediately gives the following:

\begin{corollary} \label{cor quotient of standard mult one}
Let $\omega \in \mathrm{Alg}(G_{n+1})$ be a quotient of $\lambda$ and let $\omega' \in \mathrm{Alg}(G_n)$ be a quotient of $\lambda'$. Then
\[ \mathrm{dim}~\mathrm{Hom}_{G_n}(\omega, \omega'{}^{\vee}) \leq 1. 
\]
\end{corollary}

When $\omega$ and $\omega'$ are irreducible, this in particular recovers the multiplicity-one theorem \cite{AGRS10} (see \cite{Ch21+}).

\begin{definition} \label{def integer determining branching}
\begin{enumerate}
\item Let $\pi \in \mathrm{Alg}(M_{n+1})$ and let $\pi' \in \mathrm{Alg}(G_n)$. Let $f$ be a non-zero element in $\mathrm{Hom}_{G_n}(\pi, \pi')$. We say that the {\it right layer supporting $f$} is $i^*$ (or $i^*$-th) if $i^*$ is the largest integer such that 
\[ f|_{\Lambda_{i^*-1}(\pi)} \in \mathrm{Hom}_{G_n}(\Lambda_{i^*-1}(\pi), \pi')\neq 0. \]  
We denote such integer $i^*$ by $\mathcal L_{rBL}(f)$.
\item Let $\lambda$ and $\lambda'$ be standard modules of $G_{n+1}$ and $G_n$ respectively. Let $\omega$ and $\omega'$ be (not necessarily irreducible) quotients of $\lambda$ and $\lambda'$. Suppose $\mathrm{Hom}_{G_n}(\omega, \omega'^{\vee})\neq 0$. We define $\mathcal L_{rBL}(\omega, \omega'{}^{\vee})=\mathcal L_{rBL}(f)$ for the unique (up to a scalar) non-zero element $f$ in $\mathrm{Hom}_{G_n}(\omega, \omega'{}^{\vee})$, see Corollary \ref{cor quotient of standard mult one}.   
\item For $\pi \in \mathrm{Alg}(G_{n+1})$ or $\pi \in \mathrm{Alg}(M_{n+1}^t)$. One defines the left layer supporting $f$ analogously. Then one also defines the {\it left layer supporting the branching law} and $\mathcal L_{lBL}$ analogously by using the filtration ${}_i\Lambda(\pi)$.
\end{enumerate}
\end{definition}

As hinted from \cite[Page 137]{CPS17}, for generic representations, the Whittaker functions from those layers in Definition \ref{def integer determining branching} contribute poles of $L$-functions.

For auxiliary results, we shall only prove the right version and we shall sometimes drop the term "right" if there is no confusion.

\subsection{Supporting layer in terms of BZ derivatives}

Recall that a standard module is discussed in Section \ref{ss basic notations}.

\begin{lemma} \label{lem std module filtration}
Let $\pi$ be a standard module of $G_n$. Then $\pi^{(i)}$ admits a filtration whose successive subquotients are also standard modules.
\end{lemma}

\begin{proof}
We have that $\pi \cong \lambda(\mathfrak m)$ for some multisegment $\mathfrak m$. We write the segments in $\mathfrak m$ as $\Delta_1, \ldots, \Delta_r$ such that $b(\Delta_k)\not\leq b(\Delta_l)$for any $k<l$. By using the Leibniz's rule, $\pi^{(i)}$ admits a filtration whose successive subquotients are 
\[  \mathrm{St}(\Delta_1)^{(i_1)}\times \ldots \times \mathrm{St}(\Delta_r)^{(i_r)} ,
\]
with $i_1+\ldots+i_r=i$. Each $\mathrm{St}(\Delta_k)^{(i_k)}$ is isomorphic to $\mathrm{St}({}^{(i_k)}\Delta_k)$ (here ${}^{(i_k)}\Delta_k$ is a segment from $\Delta_k$ by truncating cuspidal representations on the left), and so the subquotient is still a standard module since $b({}^{(i_k)}\Delta_k) \not\leq  b({}^{(i_l)}\Delta_l)$ for any $k<l$ (drop the term if ${}^{(i_k)}\Delta$ is empty). This proves the lemma.
\end{proof}

The following homological property is the key for deducing a characterization of branching laws.

\begin{lemma} \label{lem vanishing ext filtration}
Let $\lambda$ and $\lambda'$ be standard modules of $G_n$ and $G_m$ respectively. Let $k=n-m$. For $i, i+k \geq 0$, label the standard modules in the filtration (see Lemma \ref{lem std module filtration}) of $\lambda^{(i+k)}$ by $\kappa_1, \ldots, \kappa_p$ and label the standard modules in the filtration (see Lemma \ref{lem std module filtration}) of $\lambda'{}^{(i)}$ by $\kappa_1', \ldots, \kappa_q'$. 
\begin{enumerate}
\item If $\kappa_s \cong \theta(\kappa_t')$ for some $s,t$, then $\mathrm{Hom}_{G_{m-i}}(\lambda^{(i+k)}, {}^{(i)}(\lambda'{}^{\vee})) \neq 0$.
\item If $\kappa_s \not\cong \theta(\kappa_t')$ for any $s,t$, then, for any $j$, 
\[  \mathrm{Ext}^j_{G_{m-i}}(\lambda^{(i+k)}, {}^{(i)}(\lambda'{}^{\vee})=0 .
\]
\end{enumerate}
\end{lemma}

\begin{proof}
Recall that for two standard modules $\lambda_1, \lambda_2$ of $G_k$, $\mathrm{Ext}^j_{G_k}(\lambda_1, \lambda_2^{\vee})=0$ if $\lambda_1 \not\cong \theta(\lambda_2)$ (see \cite[Lemma 5.5]{Ch22+b}, which is stated and shown for Zelevinsky classification, but one for Langlands classification can be proved similarly or deduced by applying Bernstein's cohomological duality). Now, the lemma follows from a standard application of long exact sequences.
\end{proof}

\begin{corollary} \label{cor integer determining branching law and layers}
Let $\lambda$ and $\lambda'$ be standard modules of $G_{n+1}$ and $G_n$ respectively. Let $\pi$ and $\pi'$ be (not necessarily irreducible) quotients of $\lambda$ and $\lambda'$ respectively. Suppose $\mathrm{Hom}_{G_n}(\pi, \pi'{}^{\vee})\neq 0$. Let $i^*=\mathcal L_{rBL}(\lambda, \lambda'{}^{\vee})$. Then, for $i<i^*$,
\[    \mathrm{Hom}_{G_n}(\Sigma_i(\pi), \pi'{}^{\vee}) =0, \quad \mbox{equivalently} \quad \mathrm{Hom}_{G_{n+1-i}}(\pi^{[i]}, {}^{(i-1)}(\pi'{}^{\vee})) =0 ,
\]
and 
\[   \mathrm{Hom}_{G_n}(\Sigma_{i^*}(\pi), \pi'{}^{\vee}) \cong \mathbb C, \quad \mbox{equivalently} \quad \mathrm{Hom}_{G_{n+1-i^*}}(\pi^{[i^*]}, {}^{(i^*-1)}(\pi'{}^{\vee})) \cong \mathbb C .
\]
In particular, $\mathcal L_{rBL}(\pi, \pi'{}^{\vee})=i^*$.
\end{corollary}

\begin{proof}
Set $\pi_i=\Lambda_i(\pi)$ and $\lambda_i=\Lambda_i(\lambda)$.

\noindent
{\it Claim:} For $i<i^*$, $\mathrm{Hom}_{G_n}(\pi_i, \pi'{}^{\vee}) \cong \mathbb C$ and $\mathrm{Hom}_{G_n}(\Sigma_i(\pi), \pi'{}^{\vee}) \cong 0$; and $\mathrm{Hom}_{G_n}(\lambda_i,\lambda'{}^{\vee})\cong \mathbb C$ and $\mathrm{Hom}_{G_n}(\Sigma_i(\lambda), \lambda'{}^{\vee}) \cong 0$. \\

\noindent
{\it Proof of Claim:} We consider the following commutative diagram:
\[ \xymatrix{ 0 \ar[r] & \mathrm{Hom}_{G_n}(\Sigma_i(\pi), \pi'{}^{\vee}) \ar[r]\ar@{^{(}->}[d] & \mathrm{Hom}_{G_n}(\pi_{i-1}, \pi'{}^{\vee}) \ar[r]^{f_i''} \ar@{^{(}->}[d] & \mathrm{Hom}_{G_n}(\pi_i, \pi'{}^{\vee}) \ar@{^{(}->}[d] &  \\
              0 \ar[r] & \mathrm{Hom}_{G_n}(\Sigma_i(\lambda), \pi'{}^{\vee})\ar[r] \ar@{^{(}->}[d] & \mathrm{Hom}_{G_n}(\lambda_{i-1}, \pi'{}^{\vee})\ar[r]^{f_i'} \ar@{^{(}->}[d] & \mathrm{Hom}_{G_n}(\lambda_i, \pi'{}^{\vee}) \ar@{^{(}->}[d] \\
							0 \ar[r] & \mathrm{Hom}_{G_n}(\Sigma_i(\lambda), \lambda'{}^{\vee}) \ar[r] & \mathrm{Hom}_{G_n}(\lambda_{i-1}, \lambda'{}^{\vee}) \ar[r]^{f_i} & \mathrm{Hom}_{G_n}(\lambda_i, \lambda'{}^{\vee})  \ar[r] & \mathrm{Ext}^1_{G_n}(\Sigma_i(\lambda), \lambda'{}^{\vee}) }
\]
When $i=0$, $\pi_0=\pi$ and so it follows from the assumption and Theorem \ref{thm standard multiplicity one}. We now consider $i\geq 1$. By considering the middle line, the injectivity and assumptions imply all Hom's in the middle are one-dimensional. Then from the assumption on $\mathcal L_{rBL}(\lambda, \lambda'{}^{\vee})=i^*$, $f_i$ is non-zero and so must be injective by one-dimensionality (from induction). Then, $\mathrm{Hom}_{G_n}(\Sigma_i(\lambda), \lambda'{}^{\vee})=0$. By Lemma \ref{lem vanishing ext filtration}, the last $\mathrm{Ext}^1$ is zero and so the map $f_i$ is an isomorphism and so \[ \mathrm{Hom}_{G_n}(\lambda_i, \lambda'{}^{\vee})\cong \mathbb C,\]
 and then by the commutativity diagram, we also have $\mathrm{Hom}_{G_n}(\lambda_i, \pi'{}^{\vee}) \cong \mathbb C$ and $f'_i$ is also an isomorphism. Now we repeat the argument and obtain that $\mathrm{Hom}_{G_n}(\pi_i, \pi'{}^{\vee})\cong \mathbb C$ and $f_i''$ is an isomorphism. This then also implies that $\mathrm{Hom}_{G_n}(\Sigma_i(\pi), \pi'{}^{\vee})=0$. This proves the claim. \\

We now return to the proof. It remains to establish the case for $i=i^*$. To this end, we consider the following exact sequence:
\[ 0 \rightarrow \mathrm{Hom}_{G_n}(\Sigma_{i^*}(\pi), \pi'{}^{\vee}) \rightarrow \mathrm{Hom}_{G_n}(\pi_{i^*-1}, \pi'{}^{\vee}) \rightarrow \mathrm{Hom}_{G_n}(\pi_{i^*}, \pi'{}^{\vee}) .
\]
The last term is zero by $\mathrm{Hom}_{G_n}(\lambda_{i^*}, \lambda'{}^{\vee})=0$. Thus, the first Hom is isomorphic to the second one, and then the corollary follows from the claim. 
\end{proof}




\begin{corollary} \label{cor integer for branching law equal layer}
Let $\pi \in \mathrm{Irr}(G_{n+1})$ and let $\pi'\in \mathrm{Irr}(G_n)$. Suppose $\mathrm{Hom}_{G_n}(\pi, \pi')\neq 0$.  Let $i^*$ be the smallest integer such that 
\[ \mathrm{Hom}_{G_n}(\Sigma_{i^*}(\pi), \pi')\neq 0, \mbox{ equivalently } \quad \mathrm{Hom}_{G_{n+1-i^*}}(\pi^{[i^*]}, {}^{(i^*-1)}\pi') \neq 0.\]
 Then $\mathcal L_{rBL}(\pi, \pi')=i^*$. 
\end{corollary}

\begin{proof}
Let $i^{\#}=\mathcal L_{rBL}(\pi, \pi')$. The zeroness part for $i<i^{\#}$ of Corollary \ref{cor integer determining branching law and layers} implies $i^*\geq i^{\#}$.  The multiplicity one part of Corollary \ref{cor integer determining branching law and layers} then implies $i^{\#}\geq i^*$.  
\end{proof}

\section{Some results on computing the layer supporting a branching law}

\subsection{Submodule and quotient constraints}

We first have a submodule constraint:

\begin{lemma} \label{lem integer branching law submodule}
Let $\pi \in \mathrm{Alg}(G_{n+1})$ and let $\pi' \in \mathrm{Alg}(G_n)$. Let $f$ be a non-zero element in $\mathrm{Hom}_{G_n}(\pi, \pi')$. Suppose $\pi$ admits a filtration as $M_{n+1}$-module:
\[  0 \rightarrow \omega \rightarrow \pi \rightarrow \omega' \rightarrow 0 .
\]
Suppose $f|_{\omega} \neq 0$. Then $\mathcal L_{rBL}(f) \geq \mathcal L_{rBL}(f|_{\omega})$.

\end{lemma}

\begin{proof}

Let $\pi_i=\Lambda_i(\pi)$, the BZ filtration of $\pi$. Set $\omega_i=\omega \cap \pi_i$, which gives the BZ filtration of $\omega$. For any $i$, we have the following commutative diagram from the functoriality of Bernstein-Zelevinsky filtrations:
\[  \xymatrix{     \omega_i  \ar@{^{(}->}[r]^{\iota' } \ar[d]^{h_i} &   \pi_i \ar[d]^{k_i}  &           \\
                   \omega    \ar@{^{(}->}[r]^{\iota}        &   \pi    \ar[r]^f & \pi'  
                  }
\]
where $k_i$ is the embedding map in the Bernstein-Zelevinsky filtration for $\pi_i$, and $h_i=k_i|_{\omega_i}$. Suppose $(f \circ \iota) \circ h_i \neq 0$. Then the commutative diagram implies $f\circ k_i \neq 0$. This implies the lemma.
\end{proof}

We next have a quotient constraint:

\begin{lemma} \label{lem integer determine quotient}
Let $\pi \in \mathrm{Alg}(G_{n+1})$ and let $\widetilde{\pi}$ be a quotient of $\pi$ with the quotient map $q$. Let $\pi' \in \mathrm{Alg}(G_n)$. Suppose there exists a non-zero map $f: \widetilde{\pi} \rightarrow \pi'$. Then 
\[ \mathcal L_{rBL}(f \circ q)=\mathcal L_{rBL}(f). \]
\end{lemma}

\begin{proof}
Set $\pi_i=\Lambda_i(\pi)$ and $\widetilde{\pi}_i=\Lambda_i(\widetilde{\pi})$ for simplicity. Let $\widetilde{\pi}_i = \mathrm{im}~q_i$. We consider the following commutative diagram:
\[  \xymatrix{ \pi_i \ar@{^{(}->}[d]^{\iota_i} \ar[r]^{q_i} &   \widetilde{\pi}_i \ar@{^{(}->}[d]^{\widetilde{\iota}_i}  &          \\
               \pi \ar[r]^{q} & \widetilde{\pi} \ar[r]^f              &  \pi'
                     } .
\]
Hence, if $f\circ  q \circ \iota_i\neq 0$, the commutative diagram gives that $f\circ \widetilde{\iota}_i \neq 0$. Conversely, if $f\circ \widetilde{\iota}_i\neq 0$, then $f\circ \widetilde{\iota}_i \circ q_i \neq 0$ and so $f\circ q \circ \iota_i \neq 0$. Thus we now have the lemma.
\end{proof}

\subsection{Refined standard trick for the supporting layer}

\begin{lemma} \label{lem construct one more non-zero element}
Let $\pi \in \mathrm{Irr}(G_{n+1})$ and let $\pi' \in \mathrm{Irr}(G_n)$. Suppose there exists a $\rho \in \mathrm{Irr}^c$ good to $\nu^{1/2}\pi$ and $\rho$ is $\leq$-minimal in $\mathrm{csupp}(\pi')$. Let $\mathfrak p=\mathfrak{mxpt}^L(\pi', \rho)$ (see Definition \ref{def max multisegment}). Let $\sigma \in \mathrm{Irr}^c(G_{l_a(\mathfrak p)})$ be good to $\nu^{1/2}\pi$ and $\pi'$. Then
\[   \mathrm{Hom}_{G_{n}}(\pi, \sigma \times D^L_{\mathfrak p}(\pi') )\cong \mathbb C
 \]
and $\mathcal L_{rBL}(\pi, \sigma\times D^L_{\mathfrak p}(\pi'))=i^*$ if and only if
\[  \mathrm{Hom}_{G_{n}}(\pi, \mathrm{St}(\mathfrak p)\times D^L_{\mathfrak p}(\pi')) \cong \mathbb C 
\]
and $\mathcal L_{rBL}(\pi, \mathrm{St}(\mathfrak p)\times D^L_{\mathfrak p}(\pi'))=i^*$.
\end{lemma}

\begin{proof}
We only prove the only if direction and the if direction is very similar. We first prove the Hom result. The argument is more standard now and so will be slightly sketchy. By a duality  (see \cite[Proposition 4.1]{Ch22}), we have that:
\[   \mathrm{Hom}_{G_{n+1}}(\mathrm{St}(\mathfrak p)^{\vee} \times \sigma' \times D^L_{\mathfrak p}(\pi')^{\vee} , \pi^{\vee})  \cong \mathrm{Hom}_{G_{n}}(\pi,  \mathrm{St}(\mathfrak p)\times D^L_{\mathfrak p}(\pi'))
\]
for some cuspidal representation $\sigma'$ of $G_2$. Thus one then applies Proposition \ref{prop filtration on parabolic} on $(\mathrm{St}(\mathfrak p) \times \sigma') \times D^L_{\mathfrak p}(\pi')^{\vee},$ and the only possible layer (by an argument of comparing cuspidal supports) that can contribute the Hom is the bottom layer of certain Rankin-Selberg model involving $D^L_{\mathfrak p}(\pi')^{\vee}$ taking the form:
\begin{align} \label{eqn layer contribute in standard trick}
   \mathrm{RS}_{l_a(\mathfrak p)}(D^L_{\mathfrak p}(\pi')^{\vee})
\end{align}
(see Section \ref{ss bz filtrations variant}). 

 On the other hand, if one considers $\mathrm{Hom}_{G_{n}}(\pi, \sigma \times D^L_{\mathfrak p}(\pi')^{\vee})$, then one applies a similar duality and then the layer that can contribute the Hom is again of such form. Thus, now we have
\[ \mathrm{Hom}_{G_{n}}(\pi, \sigma \times D^L_{\mathfrak p}(\pi'))\cong \mathbb C  \Longrightarrow \  \mathrm{Hom}_{G_{n}}(\pi, \mathrm{St}(\mathfrak p) \times D^L_{\mathfrak p}(\pi'))\cong \mathbb C .
\]

For convenience, for representations $\omega_1\in \mathrm{Alg}(G_k)$ and $\omega_2 \in \mathrm{Alg}(G_l)$ of finite lengths, set
\begin{align} \label{eqn def of smallest derivative}
  \mathcal L_{smD}(\omega_1, \omega_2) =\mathrm{min}\left\{ i : \mathrm{Hom}_{G_{k-i}}(\omega_1^{[i]}, {}^{(i-k+l)}\omega_2) \neq 0 \right\} .
\end{align}

We now turn to the part of the layer supporting the branching law. By using Corollary \ref{cor integer determining branching law and layers}, $i^*=\mathcal L_{smD}(\pi, \sigma \times D^L_{\mathfrak p}(\pi'))$, and so a cuspidal consideration also implies that $i^*=\mathcal L_{smD}(\pi, D^L_{\mathfrak p}(\pi'))$. This then, by a cuspidal consideration, also implies that $i^*=\mathcal L_{smD}(\pi, \mathrm{St}(\mathfrak p)\times D^L_{\mathfrak p}(\pi'))$. Thus, now the lemma follows from Corollary \ref{cor integer determining branching law and layers} (or some variants of Corollary \ref{cor integer for branching law equal layer}).
\end{proof}

Note that similar lines of arguments also prove Lemma \ref{lem standard trick first form} as well as its refinement analogue to the above one. 

\section{Constructing branching laws via Ranking-Selberg integrals} \label{s construct branching via rs}



In this section, we consider the Rankin-Selberg integrals \cite[Theorem 2.7]{JPSS83} and the work \cite{CPS17} of Cogdell--Piatetski-Shapiro.

\subsection{Preliminaries for Rankin-Selberg integrals} \label{ss prelim rs integral}
We consider an induced representation of the form:
\[  \lambda := \mathrm{St}(\Delta_1) \times \ldots \times \mathrm{St}(\Delta_r) 
\]
for some segments $\Delta_1, \ldots, \Delta_r$. Let $n+1=l_a(\Delta_1)+\ldots+l_a(\Delta_r)$. Let $u$ be an $r$-tuple $(u_1,\ldots, u_r)$, considered as indeterminants. We simply write $\mathbb C[q^{\pm u}]$ for $\mathbb C[q^{\pm u_1}_1, \ldots, q_r^{\pm u_r}]$. and let $\widetilde{\lambda}_u$ be the space of all smooth functions from $G_{n+1}$ to $\mathbb C[q^{\pm u}] \otimes \pi$ such that for any $f \in \mathbb C[q^{\pm u}]\otimes \pi$, 
\[  f(pg)= \nu^{u_1}(m_1)\ldots \nu^{u_r}(m_r) p.f(g)  .
\]
One can also embed $\pi$ into $\mathbb C[q^{\pm u}]\otimes \pi$ via the so-called compact realization, say the map to be $\mathcal F$, such that
\[  \mathcal F(f)(k)= f(k) 
\]
for all $k$ in the maximal compact subgroup $\mathrm{GL}_n(\mathcal O)$ of $\mathrm{GL}_n(F)$. Let $\lambda_u$ be the $G_n$-representation generated by $\mathrm{im}(\mathcal F)$. We fix a Whittaker function from $\psi: U_n \rightarrow \mathbb C$. Following \cite{CPS17}, for each $f \in \mathrm{im}(\mathcal F)$, Casselman-Shalika constructed a $\psi$-Whittaker linear functional $W(f): G_n \rightarrow \mathbb C[q^{\pm u}]$ i.e. satisfying, for $n' \in U_n$,
\[    W(f)(n'g)= \psi(n')f(g) .
\]
Explicitly, let $N$ be the unipotent radical of $P_{l_1,\ldots, l_r}$ ($l_k=l_a(\Delta_k)$) and let $\left\{ N_i \right\}$ be an exhaustive filtration of $N_Q$ by open compact subgroups. We have that the Whittaker functional is defined as:
\begin{align} \label{eqn whittaker function}
 W(f)(g)=  \int_{N_i} (\lambda(ng)f)(\dot{w}) \psi^{-1}(n)~dn , 
\end{align}
for sufficiently large $i$ so that the integral is stabilized, where $\dot{w}$ is a fixed representative of a Weyl group element in $S_{n+1}$. 

Let $\widetilde{\mathcal W}_{\lambda}$ be the space spanned by all those $\psi$-Whittaker linear functionals. The space $\widetilde{\mathcal W}_{\lambda}$ is not invariant under $G_n$-actions.  Let $\mathcal W_{\lambda}$ be the $G_n$-representation generated by $\widetilde{\mathcal W}_{\lambda}$. From our construction, we have an isomorphism as $G_n$-representations:
\[   \lambda_u \cong \mathcal W_{\lambda}  
\]
and so realize elements in $\lambda_u$ as $\psi$-Whittaker functionals.

We now consider a $G_n$-representation $\lambda'$
\[  \lambda' =\mathrm{St}(\Delta_1')\times \ldots \times \mathrm{St}(\Delta_t') .
\]
Let $v=(v_1,\ldots, v_t)$ be a $t$-tupble, and we similarly construct a $G_n$-representation $\lambda'_v$ and $\mathcal W_{\lambda'}$ in analog of $\lambda_u$ and $\mathcal W_{\lambda}$ above, respectively.

For $W \in \mathcal W_{\lambda}$ and $W'\in \mathcal W_{\lambda'}$, the Rankin-Selberg integral is defined as:
\begin{align} \label{eqn rankin selberg integral}
  I(s;W, W') = \int_{N_{m}\setminus G_{m}}     W\begin{pmatrix} g & \\ & 1 \end{pmatrix} W'(g) |\mathrm{det}(g)|^{s-1/2}    dg  ,
\end{align}
where $dg$ is a right-invariant measure on $N_m\setminus G_m$. It is shown in \cite{CPS17} that there exists a linear function $\ell_{\lambda, \lambda'}(u,v,s)$  such that $I(s;W, W')$ absolutely converges for 
\begin{align} \label{eqn real positive linear}
\mathrm{Re}(\ell_{\lambda, \lambda'}(u,v,s))>0.
\end{align}

\begin{lemma} \cite[Proposition 3.2]{CPS17} \label{lem rational rs integral}
Let $W \in \mathcal W_{\lambda}$ and let $W' \in \mathcal W_{\lambda'}$. Then $I(s; W, W')$ is a rational function of $q^{-u}$, $q^{-v}$ and $q^{-s}$. 
\end{lemma}

For two essentially discrete series $\mathrm{St}(\Delta)$ and $\mathrm{St}(\Delta')$, we denote by $L(s, \mathrm{St}(\Delta) \times \mathrm{St}(\Delta'))$ the corresponding Rankin-Selberg $L$-function, see discussions in the beginning of \cite[Page 167]{CPS17}.

\begin{lemma}
Let $W \in \mathcal W_{\lambda}$ and let $W' \in \mathcal W_{\lambda'}$. Then 
\[   \frac{ I(s; W, W')}{\prod_{1\leq i \leq r, 1\leq j \leq t} L(s+u_i+v_j, \mathrm{St}(\Delta_i)\times \mathrm{St}(\Delta_j)) } 
\]
has no poles and lies in the Laurent polynomial ring $\mathbb C[q^{\pm u}, q^{\pm v}, q^{\pm s}]$. Moreover, the expression is not identically zero.
\end{lemma}

\begin{proof}
This is shown in \cite[Page 167]{CPS17} that the above expression has no poles. Hence, the expression is analytic and so is in $\mathbb C[q^{\pm u}, q^{\pm v}, q^{\pm s}]$. Moreover, it follows from the construction of $I(s, W, W')$ that the expression is non-zero.
\end{proof}

\begin{corollary} \label{cor nonzero G equivariant map}
There exists a non-zero $G_n$-equivariant map from $\mathcal W_{\lambda}\otimes \mathcal W_{\lambda'}$ to $\mathbb C[q^{\pm u}, q^{\pm v}]$. 
\end{corollary}

\begin{proof}
Let $\widetilde{B}_s: \mathcal W_{\lambda} \times \mathcal W_{\lambda'} \rightarrow \mathbb C[q^{\pm u}, q^{\pm v}, q^{\pm s}]$ given by:
\[   \widetilde{B}_s(W, W') = \frac{I(s; W, W')}{\prod_{1\leq i \leq r, 1\leq j \leq t} L(s+u_i+v_j, \mathrm{St}(\Delta_i)\times \mathrm{St}(\Delta_j)) } .
\]
Now, we divide $(q^s-q^{1/2})^a$ for some non-negative $a$ on $\widetilde{B}_s$ such that the map $B_s: \mathcal W_{\lambda} \times \mathcal W_{\lambda'} \rightarrow \mathbb C[q^{\pm u}, q^{\pm v}, q^{\pm s}]$ given by
\[   (W, W') \mapsto \frac{B_s(W,W')}{(q^s-q^{1/2})^a} 
\]
is not an identically zero map when $s$ is specialized to $\frac{1}{2}$ and still has no pole at $s=\frac{1}{2}$.

Then, one checks that 
\[ B_s(g.W, g.W')= \nu(g)^{-s+1/2} B_s(W, W')  .\]
Specializing to $s=\frac{1}{2}$, we obtain a non-zero $G_n$-equivariant map from $\mathcal W_{\lambda} \times \mathcal W_{\lambda'}$ to $\mathbb C[q^{\pm u}, q^{\pm v}]$, and so lifts to a non-zero $G_n$-equivariant map in the corollary. 
\end{proof}

\begin{corollary} \label{cor G equivariant linear map}
Let $1 \leq i \leq r$ and $1\leq j \leq t$. Let $u_i^*, \ldots, u_r^*, v_j^*, \ldots, v_t^* \in \mathbb C$. Let $\widetilde{\mathrm{sp}}: \mathbb C[q^{\pm u}, q^{\pm v}, q^{\pm s}] \rightarrow \mathbb C$ be the specialization map for evaluating $u_1=u_1^*,\ldots, u_r=u_r^*, v_1=v_1^*, \ldots, v_t=v_t^*$. This induces a map $\mathrm{sp}: \mathbb C[q^{\pm u}]\otimes \lambda \rightarrow \mathbb C[q^{\pm u_1}, \ldots, q^{\pm u_i}]\otimes \lambda$ and a map $\mathrm{sp}': \mathbb C[q^{\pm v}]\otimes \lambda \rightarrow \mathbb C[q^{\pm v_1}, \ldots, q^{\pm v_j}]\otimes \lambda'$. There exists a non-zero $G_n$-equivariant linear map from $\mathrm{sp}(\lambda_u) \otimes \mathrm{sp}'(\lambda_u')$ to $\mathbb C[q^{\pm u_1}, \ldots q^{\pm u_{i-1}}, q^{\pm v_1}, \ldots, q^{\pm v_{j-1}}]$.  
\end{corollary}

\begin{proof}
Let $C: \mathcal W_{\lambda} \otimes \mathcal W_{\lambda'} \rightarrow \mathbb C[q^{\pm u}, q^{\pm v}]$ be the $G_n$-equivariant map obtained in Corollary \ref{cor nonzero G equivariant map}. Now, by dividing 
\[ (q^{u_i}-q^{u_i^*})^{a_i}\ldots (q^{u_r}-q^{u_r^*})^{a_r}(q^{v_j}-q^{v_j^*})^{b_j}\ldots (q^{v_t}-q^{v_t^*})^{b_t} \]
for some non-negative integers $a_i, \ldots, a_r, b_j, \ldots, b_t$ if necessary, we can then obtain a $G_n$-equivariant map, denoted by $C'$, from $\mathcal W_{\lambda} \otimes \mathcal W_{\lambda'}$ to $\mathbb C[q^{\pm u}, q^{\pm v}]$ which has no pole and is not an identically zero map. In other words, the composition
\[  \widetilde{\mathrm{sp}}\circ C'
\]
gives a non-zero and well-defined $G_n$-equivariant map. 

It remains to show that the map $\widetilde{\mathrm{sp}} \circ C'$ can descend to a map from $\mathrm{sp}(\lambda_u) \otimes \mathrm{sp}'(\lambda_u')$ to $\mathbb C[q^{\pm u_1}, \ldots q^{\pm u_{i-1}}, q^{\pm v_1}, \ldots, q^{\pm v_{j-1}}]$. In other words, for $f_1, f_2 \in \lambda_u$ and $f_1', f_2' \in \lambda_v'$, if $\mathrm{sp}(f_1)=\mathrm{sp}(f_2)$ and $\mathrm{sp}'(f_1')=\mathrm{sp}'(f_2')$, then 
\[   \widetilde{\mathrm{sp}}\circ C'(f_1\otimes f_1')=\widetilde{\mathrm{sp}}\circ C' (f_2\otimes f_2') .
\]
Now, one observes that $\widetilde{\mathrm{sp}}\circ W(f_1) =\widetilde{\mathrm{sp}}\circ W(f_2)$, and $\widetilde{\mathrm{sp}}\circ W'(f_1')=\widetilde{\mathrm{sp}}\circ W'(f_2')$ from the construction (\ref{eqn whittaker function}). One now chooses sufficiently large $u_1, \ldots, u_{i-1}, v_1, \ldots, v_{j-1}, s$ so that the condition (\ref{eqn real positive linear}) is satisfied, and so the corresponding specializations on $I(s; W(f_1), W'(f_1'))-I(s;W(f_2), W'(f_2'))$ then are zero. But now, since $I(s; W(f_1), W'(f_1'))-I(s;W(f_2), W'(f_2))$ is in $\mathbb C(q^{\pm u}, q^{\pm v}, q^{\pm s})$ by Lemma \ref{lem rational rs integral}, and so must be identically equal to zero, as desired.
\end{proof}

\begin{corollary} \label{cor vanishing for infinite many}
We use the notations in  Corollary \ref{cor G equivariant linear map}, and fix $i=2$ and $j=1$. Let $\widetilde{C}$ be the $G_n$-equivariant linear map in Corollary \ref{cor G equivariant linear map}. Let $f \in \lambda_u$ and $f'\in \lambda'_v$. If $\widetilde{C}(f\otimes f')=0$ for infinitely many values of $u_1$, then $\widetilde{C}(f\otimes f')=0$ for all values of $u_1$. 
\end{corollary}

\begin{proof}
This follows from that a non-zero Laurent polynomial in $\mathbb C[q^{\pm u}]$ has only finitely many roots and from Corollary \ref{cor G equivariant linear map}.
\end{proof}

\subsection{A construction}

\begin{lemma} \label{lem construct branching from deformation}
Let $\pi \in \mathrm{Irr}(G_n)$.  Let $\Delta$ be a segment of absolute length $k$. Let $\pi' \in \mathrm{Irr}(G_{n+k-1})$. Let $\sigma \in \mathrm{Irr}^c(G_k)$ such that $\sigma$ is good to $\pi$ and $\nu^{-1/2}\pi'$. Suppose
\[  \mathrm{Hom}_{G_{n+k-1}}(\sigma \times \pi, \pi') \neq 0
\]
with $\mathcal L_{rBL}(\sigma\times \pi, \pi')=i^*$. Then there exists a non-zero $f$ in $\mathrm{Hom}_{G_{n+k-1}}(\mathrm{St}(\Delta)\times \pi, \pi')$ such that $\mathcal L_{rBL}(f) \leq i^*$. 
\end{lemma}

\begin{proof}
\noindent
{\bf Step 1: Construct a branching law} \\

From the non-vanishing Hom hypothesis in the lemma and a variation of the standard trick (see Lemma \ref{lem construct one more non-zero element}), we have that
\begin{align} \label{eqn nonzero hom in construction}
  \mathrm{Hom}_{G_{n+k-1}}(\nu^{u_0}\mathrm{St}(\Delta) \times \pi, \pi') \neq 0
\end{align}
for all except finitely many $u_0 \in \mathbb C$. 

Now we find standard representations $\lambda$ and $\lambda'$ for $\pi$ and $\pi'{}^{\vee}$ respectively:
\[ \lambda:=  \mathrm{St}(\Delta_1) \times \ldots \times \mathrm{St}(\Delta_r) \twoheadrightarrow \pi
\]
and
\[  \lambda':= \mathrm{St}(\Delta_1')\times \ldots \times \mathrm{St}(\Delta_s') \twoheadrightarrow \pi'{}^{\vee} .\]
Let $n_k=l_a(\Delta_k)$ for $k=1, \ldots, r$ and let $n_k'=l_a(\Delta_k')$ for $k=1, \ldots, s$. 

Let $u=(u_1, \ldots, u_r)$ be a $r$-tuple and let $v=(v_1, \ldots, v_s)$ be a $s$-tuple. Let $\widetilde{\lambda}_u$ be the space of functions from $G_n$ to $\mathbb C[q^{\pm u}, q^{\pm v}]\otimes (\mathrm{St}(\Delta_1)\boxtimes \ldots \boxtimes \mathrm{St}(\Delta_r))$ satisfying 
\[ f(pg)=\nu(m_1)^{u_1}\ldots \nu(m_r)^{u_r}p.f(g) ,\]
where $p=\mathrm{diag}(m_1, \ldots, m_r)\cdot n$ for $m_i \in G_{n_i}$ and $n \in N_{n_1, \ldots, n_r}$ and $p$ acts via the action on $ \mathrm{St}(\Delta_1) \boxtimes \ldots \boxtimes \mathrm{St}(\Delta_r)$. Let $\lambda_u$ be the subspace of $\widetilde{\lambda}_u$ described in the beginning of Section \ref{ss prelim rs integral}. We similarly define for $\lambda'_v$.



Let $\mathrm{sp}: \mathbb C[q^{\pm u_0}, q^{\pm u_1}, \ldots, q^{\pm u_r}, q^{\pm v_1}, \ldots, q^{\pm v_t}, q^{\pm s}]\otimes \lambda\rightarrow \mathbb C[q^{\pm u_1}]\otimes \lambda$ be the specialization map for $u_1=\ldots=u_r=v_1=\ldots=v_t=0$ and $s=\frac{1}{2}$. By Corollary \ref{cor G equivariant linear map}, we obtain a non-zero linear functional 
\[  \mu_{u_0}: \mathrm{sp}(\lambda_{u_0}) \otimes \lambda' \rightarrow \mathbb C[q^{\pm u_0}] .
\]

Now, we have an embedding
\[   \mathrm{Hom}(\nu^{u^*}\mathrm{St}(\Delta)\times \pi, \pi')\hookrightarrow \mathrm{Hom}(\nu^{u^*}\mathrm{St}(\Delta)\times \lambda, \lambda'{}^{\vee}),
\]
which is an isomorphism for infinitely many $u_0$ by using the multiplicity one, see Corollary \ref{cor quotient of standard mult one}, for both spaces (recall that the former space is non-zero by (\ref{eqn nonzero hom in construction})). This implies that $\mu_{u^*}$ vanishes on the kernel of the natural map
\[   (\nu^{u^*}\mathrm{St}(\Delta)\times \lambda) \otimes \lambda' \rightarrow (\nu^{u^*}\mathrm{St}(\Delta) \times \pi)\otimes \pi'{}^{\vee}
\]
induced from the surjections $\lambda \rightarrow \pi$ and $\lambda' \rightarrow \pi'{}^{\vee}$, for infinitely many $u^*$ and so, by Corollary \ref{cor vanishing for infinite many}, for all $u^*$. Thus now $\mu_0$ descends to a linear functional 
\[   (\nu^{u_0}\mathrm{St}(\Delta)\times \pi) \otimes \pi'{}^{\vee} \rightarrow \mathbb C[q^{\pm u_0}].
\]
Let $u^*\in \mathbb C$. Again, by multiplying a normalizing factor if necessary, it descends to a non-zero map from \[ \mu_{u^*}':(\nu^{u^*}\mathrm{St}(\Delta)\times \pi)\otimes \pi'{}^{\vee} \rightarrow \mathbb C . \]
When $u^*=0$, we obtain a desired map in $\mathrm{Hom}_{G_n}((\mathrm{St}(\Delta)\times \pi)\otimes \pi'{}^{\vee}, \mathbb C)$. \\

\noindent
{\bf Step 2: Check the layer supporting the branching law} \\

 Set the Bernstein-Zelevinsky filtration, as a $M_{n+k}$-representation, on $\widetilde{\lambda}_{u_0}$ by
\[  \omega_{i} = \Lambda_i(\widetilde{\lambda}_{u_0}) .
\]
It remains to show that the layer supporting the branching law from $\mu'_0$ is at most $i^*$ i.e.
\[  \mu_0'(\omega_{i^*}\otimes \pi'^{\vee})=0 .
\]

Let $u^* \in \mathbb C$. We now let $\widetilde{\lambda}_{u^*} =(\nu^{u^*}\cdot\mathrm{St}(\Delta)) \times \pi$, that is the specialization of $\widetilde{\lambda}_{u_0}$ at $u^*$. Set
\[ \omega^*_{i} =\Lambda_i(\widetilde{\lambda}_{u^*}) .
\]

Now, we define 
\[ \mathrm{sp}_{u^*}: \widetilde{\lambda}_{u_0}\otimes \pi'{}^{\vee} \rightarrow \widetilde{\lambda}_{u^*}\otimes \pi'^{\vee} \]
by specializing $u_0=u^*$.  

Indeed, one has that, as a consequence of Proposition \ref{prop preserve bz filtration} (see Section \ref{rmk special parameters step by step} in the Appendix) and using that the exactness in the functoriality of Bernstein-Zelevinsky filtration, the specialization of $\omega_i$ for $u^*=u_0$ gives $\omega_i^*$. Thus, by tensoring with $\otimes \pi'{}^{\vee}$ (which is exact), 
\[  (*)\quad  \mathrm{sp}_{u^*}(\omega_i \otimes \pi'{}^{\vee})=\omega_i^*\otimes \pi'{}^{\vee} .
\]
Now, by a similar argument in Lemma \ref{lem construct one more non-zero element} and $\mathcal L_{rBL}(\sigma \times \pi, \pi')=i^*$, we have that 
\[  \mathcal L_{rBL}(\nu^{u^*}\mathrm{St}(\Delta) \times \pi, \pi') =i^* \]
 for infinitely many $u^*$. Thus, 
\begin{align*}
 \mu_{u_0}(\omega_{i^*} \otimes \pi'{}^{\vee})|_{u_0=u^*} &= \mu_{u^*}'(\mathrm{sp}_{u^*}(\omega_{i^*}\otimes \pi'{}^{\vee})) \\
   &= \mu_{u^*}'(\omega_{i^*}^* \otimes \pi'{}^{\vee}) \\
	 &=0 
\end{align*}
for infinitely many $u^*$. However, this then implies that those equations also hold for all $u^*$, particularly when $u^*=0$. This now proves the lemma by Definition \ref{def integer determining branching}.
\end{proof}

\section{The smallest derivative controlling some BZ-layers supporting branching laws} \label{s smallest der control branching law}

\subsection{Computation for a layer}

We need the following slightly technical computation later:

\begin{lemma} \label{lem compute layers by geometric lemma}
Let $\Delta$ be a segment and let $\mathfrak p \in \mathrm{Mult}$ be $\Delta$-saturated. Let $\rho' \in \mathrm{Irr}^c$ such that $\rho' \not\geq \widetilde{\rho}$ for any $\widetilde{\rho} \in \mathrm{csupp}(\mathfrak p)$ or $\widetilde{\rho} \in \mathrm{csupp}(\tau)$. Let $\mathfrak p'$ be a strongly L-$\rho'$-saturated multisegment (see Section \ref{ss standard trick intro}). Let $\tau \in \mathrm{Alg}(G_n)$ and let $\tau' \in \mathrm{Irr}(G_{n'})$, where $l_a(\mathfrak p)+n-k=l_a(\mathfrak p')+n'$. Let $\mathfrak r=\mathfrak{mx}(\tau', \Delta)$. Suppose $l_a(\mathfrak p)\geq l_a(\mathfrak r)$. Then
\begin{enumerate}
\item If $k<l_a(\mathfrak p)-l_a(\mathfrak r)$, then 
 \[ \mathrm{Hom}(\mathrm{St}(\mathfrak p)^{(k)}\times \tau, \mathrm{St}(\mathfrak p')\times \tau' ) =0 .\]
\item If $k=l_a(\mathfrak p)-l_a(\mathfrak r)$, then 
\[  \mathrm{Hom}(\mathrm{St}(\mathfrak p)^{(k)}\times \tau, \mathrm{St}(\mathfrak p') \times \tau') \cong \mathrm{Hom}(\tau, \mathrm{St}(\mathfrak p') \times D_{\mathfrak r}(\tau') ).
\]
 \end{enumerate}

\end{lemma}

\begin{proof}

One applies the second adjointness of Frobenius reciprocity to have that:
\[  \mathrm{Hom}_{G_n}(\mathrm{St}(\mathfrak p)^{(k)} \times \tau,   \mathrm{St}(\mathfrak p')\times \tau')\cong \mathrm{Hom}( \tau \boxtimes \mathrm{St}(\mathfrak p)^{(k)}, (  \mathrm{St}(\mathfrak p') \times \tau')_N) ,
\]
where $N$ is the corresponding unipotent radical for the parabolically induced module $\mathrm{St}(\mathfrak p)^{(k)}\times \tau$. Now one applies the geometric lemma to analyse the layers. 

We first consider (1). In such case, one can conclude that those related Hom's are zero by either
\begin{itemize}
\item using the cuspidal representation $b(\Delta')$; or
\item using $D_{\mathfrak t}(\pi')=0$ if $\mathfrak t$ is not a submultisegment of $\mathfrak r$ (e.g. by using \cite[Corollary 7.23]{Ch22+} and the removal process in Section \ref{def removal process}) and so, in particular, when a multisegment $\mathfrak t$ is of the form 
\[  \sum_{\Delta \in \mathfrak p}   {}^{(j_{\Delta})}\Delta \]
with sum of $j_{\Delta}$ equal to $l_a(\mathfrak p)-k$. 
\end{itemize}

Now we consider (2) and so $k=l_a(\mathfrak p)-l_a(\mathfrak r)$. Again, analysing the layers in the geometric lemma, it reduces to compute
\[  \mathrm{dim}~ \mathrm{Hom}( \tau \boxtimes \mathrm{St}(\mathfrak p)^{(k)}, \mathrm{St}(\mathfrak p')\dot{\times}^1 \tau'_N) ,
\]
where $N=N_{l_a(\mathfrak r)}$. By using $\mathfrak{mx}(\tau',\Delta)$, the only composition factor of $\mathrm{St}(\mathfrak p)^{(k)}$ that can contribute a non-zero Hom is $\mathrm{St}(\mathfrak r)$. Let $\zeta$ be the indecomposable component in $\mathrm{St}(\mathfrak p)^{(k)}$ that has the cuspidal support as $\mathrm{St}(\mathfrak r)$ . It is possibly zero, but if it is non-zero, $\zeta$ has a unique simple quotient isomorphic to $\mathrm{St}(\mathfrak r)$ (see e.g. \cite[Proposition 2.5 and Corollary 2.6]{Ch21}). Thus, it further reduces to compute
\[ \mathrm{dim}~\mathrm{Hom}(\tau \boxtimes \zeta, \mathrm{St}(\mathfrak p')\dot{\times}^1\tau'_N ).
\]

On the other hand, since $\mathfrak{mx}(\tau, \Delta)=\mathfrak r$, $D_{\mathfrak r}(\tau)\boxtimes \mathrm{St}(\mathfrak r)$ is a direct summand in $\tau_N$ (see \cite[Proposition 2.1]{Ch22+b}) and no other composition factors in $\tau_N$ takes the form $\omega \boxtimes \mathrm{St}(\mathfrak r)$. Hence, now with K\"unneth formula, we now reduce to compute:
\[\mathrm{dim}~\mathrm{Hom}_{G_n}(\tau, \mathrm{St}(\mathfrak p')\times D_{\mathfrak r}(\tau')) .
\]
This completes the proof.
\end{proof}

\subsection{Relation to the layer supporting the branching law}

\begin{lemma} \label{lem a branching factor through quotient}
Let $\pi \in \mathrm{Irr}(G_{n+1})$ and let $\pi' \in \mathrm{Irr}(G_n)$. Suppose $(\pi, \pi')$ is relevant with the associated strongly RdLi-commutative triple $(\mathfrak m, \mathfrak n, \pi)$. Let $\Delta$ be a $\preceq^R$-maximal element in $\mathfrak m$ and let $\mathfrak p=\mathfrak{mx}(\pi, \Delta)$. Suppose there exists a non-zero $f$ in
\[  \mathrm{Hom}_{G_n}(\mathrm{St}(\mathfrak p)\times D_{\mathfrak p}(\pi), \pi')  
\] 
with $\mathcal L_{rBL}(f)\leq l_a(\mathfrak m)$. Let $\mathfrak r$ be the multisegment completing $(\mathfrak m, \Delta, \pi)$. Let $\omega$ be the submodule of $\mathrm{St}(\mathfrak p)\times D_{\mathfrak p}(\pi)$ containing all the layers 
\[    \mathrm{St}(\mathfrak p)^{[k]} \times \mathrm{RS}_{k-2}(D_{\mathfrak p}(\pi))
\]
for $k > l_a(\mathfrak p)-l_a(\mathfrak r)$. Then $f|_{\omega}=0$. (Note that $\omega$ is possibly zero, which happens when $\mathfrak r=\emptyset$.)
\end{lemma}

\begin{proof}
Suppose not to derive a contradiction. Then there exists some submodule $\omega'$ of $\omega$ (from the filtration in Proposition \ref{prop filtration on parabolic}) such that we have the following commutative diagram:
\[  \xymatrix{  \omega' \ar@{^{(}->}[r] \ar@{->>}[d]^s  &   \omega  \ar[r] & \pi \ar[r]^f & \pi'  \\
               \mathrm{St}(\mathfrak p)^{[k^*]} \times \mathrm{RS}_{k^*-2}(D_{\mathfrak p}(\pi))   \ar@{->}[urrr]^{\widetilde{f}}   &   &   &
                          }
\]
for some $k^* >l_a(\mathfrak p)-l_a(\mathfrak r)$ and some $\widetilde{f}$; and the composition $\widetilde{f}\circ s\neq 0$. Set $\tau=D_{\mathfrak p}(\pi)$. Now, $\mathrm{St}(\mathfrak p)^{[k^*]}\times \mathrm{RS}_{k^*-2}(\tau)$ admits Bernstein-Zelevinsky filtration whose successive subquotients of the form:
\[   \mathrm{St}(\mathfrak p)^{[k^*]} \times \tau^{[l]} \times \Pi_{k^*+l-1} 
\]
with $l$ varying. Thus, in order to get $\mathrm{Hom}( \mathrm{St}(\mathfrak p)^{[k^*]} \times \tau^{[l]} \times \Pi_{k^*+l-1} , \pi')\neq 0$, we must have that $k^*+l>l_a(\mathfrak m)$ by Corollary \ref{cor strict inequality for integer db}. Then, $\mathcal{L}_{rBL}(\widetilde{f}) > l_a(\mathfrak m)$. Then, by Lemma \ref{lem integer determine quotient}, $\mathcal L_{rBL}(\widetilde{f}\circ s)> l_a(\mathfrak m)$. By Lemma \ref{lem integer branching law submodule}, $\mathcal L_{rBL}(f)>l_a(\mathfrak m)$, giving a contradiction. 
\end{proof}


\section{Proof of sufficiency of the generalized relevant pairs} \label{s proof sufficiency}

\subsection{Strategy} 

We shall extract the branching law from some reducible modules and the following lemma will be useful . 

\begin{lemma} \label{lem general strategy}
Let $\tau$ and $\omega$ be $G_n$-modules with short exact sequences:
\[  0 \rightarrow \tau_2 \rightarrow \tau \rightarrow \tau_1 \rightarrow 0 
\]
and
\[  0 \rightarrow \omega_1 \rightarrow \omega \rightarrow \omega_2 \rightarrow 0.
\]
Suppose the following conditions hold:
\[ \mathrm{Hom}_{G_n}(\tau, \omega) \cong \mathbb C ,
\]
\[  \mathrm{Hom}_{G_n}(\tau, \omega_1) \cong \mathbb C, \quad \mathrm{Hom}_{G_n}(\tau_1, \omega) \cong \mathbb C. 
\]
Then $\mathrm{Hom}_{G_n}(\tau_1, \omega_1) \cong \mathbb C$. 
\end{lemma}

\begin{proof}
It is clear that $\mathrm{dim}~\mathrm{Hom}_{G_n}(\tau_1, \omega_1) \leq 1$. Suppose the dimension is zero to derive a contradiction. Let $f \in \mathrm{Hom}_{G_n}(\tau_1, \omega)$ be non-zero. Then $\mathrm{im}~f \not\subset \omega_1$. Let $f' \in \mathrm{Hom}_{G_n}(\tau, \omega_1)$ be non-zero. Then $\mathrm{im}~f'\subset \omega_1$. Both $f$ and $f'$ can lift to a map in $\mathrm{Hom}_{G_n}(\tau, \omega)$. But $f$ and $f'$ are not a scalar multiple of each other. This implies that $\mathrm{dim}~\mathrm{Hom}_{G_n}(\tau, \omega) \geq 2$.
\end{proof}



\subsection{Zero relative rank case}
We shall do an induction and need the following definition:

\begin{definition}
Let $\pi_1 \in \mathrm{Irr}(G_{n+1})$ and let $\pi_2 \in \mathrm{Irr}(G_n)$. A cuspidal representation $\rho$ in $\mathrm{csupp}(\pi_1)$ is said to be {\it relevant} to $\pi_2$ if $\nu^{1/2}\rho \in \mathrm{csupp}_{\mathbb Z}(\pi_2)$. The {\it relevant relative rank} of $(\pi_1, \pi_2)$, denoted by $\mathcal{RR}(\pi_1, \pi_2)$, is 
\[  \sum_{\sigma \in \mathrm{csupp}(\pi_1): \mbox{ relevant to $\pi_2$}} n(\sigma) +\sum_{\sigma \in \mathrm{csupp}(\pi_2): \mbox{ relevant to $\pi_1$}} n(\sigma) .\]
\end{definition}

We first prove a basic case:

\begin{lemma} \label{lem basic thm branching}
Let $\pi_1 \in \mathrm{Irr}(G_{n+1})$ and let $\pi_2 \in \mathrm{Irr}(G_n)$. Suppose the relative rank of $(\pi_1,\pi_2)$ is $0$. Then the following statements are equivalent:
\begin{enumerate}
\item $\mathrm{Hom}_{G_n}(\pi_1, \pi_2)\neq 0$; 
\item both $\pi_1$ and $\pi_2$ are generic;
\item $(\pi_1, \pi_2)$ is $(n+1)$-relevant;
\item $(\pi_1, \pi_2)$ is relevant.
\end{enumerate}
\end{lemma}

\begin{proof}
(2) $\Rightarrow$ (1) is well-known \cite{JPSS83}. (1) $\Rightarrow$ (2) follows from an application of BZ filtrations and comparing cuspidal supports. 

(2)$\Leftrightarrow$(3) is straightforward from definitions. (3)$\Leftrightarrow$(4) follows from definitions and note that one has to use the relative rank to be zero to obtain the level of layer supporting via comparison of cuspidal supports.
\end{proof}



\noindent
\subsection{Proof of sufficiency}

\begin{theorem} \label{thm sufficiency}
Let $\pi \in \mathrm{Irr}(G_{n+1})$ and let $\pi' \in \mathrm{Irr}(G_n)$. Suppose $(\pi, \pi')$ is $i^*$-relevant. Then 
\[  \mathrm{Hom}_{G_n}(\pi, \pi') \neq 0
\]
and $\mathcal L_{rBL}(\pi, \pi')=i^*$. 
\end{theorem}

\begin{proof}

\noindent
{\bf Step 1 Inductive setup.} We shall prove by an induction on the relative rank. When the relative rank $\mathcal{RR}(\pi,\pi')$ is zero, it is shown in Lemma \ref{lem basic thm branching}.



\noindent

We now consider the relative rank is non-zero. By the relevance condition, there exists a minimal strongly RdLi-commutative triple $(\mathfrak m, \mathfrak n, \nu^{1/2}\cdot \pi)$ such that 
\[    D_{\mathfrak m}^R(\nu^{1/2}\cdot \pi) \cong D^L_{\mathfrak n}(\pi') .
\]
By definition, we have $i^*=l_a(\mathfrak m)$. Note that the statement $\mathcal L_{rBL}(\pi, \pi')=i^*$ would follow from Proposition \ref{prop smallest integer in strong triple} and Corollary \ref{cor integer for branching law equal layer} once we show that $\mathrm{Hom}_{G_n}(\pi, \pi')\neq 0$.

Now we consider two cases: \\

\noindent
{\bf Case 1:} There exists a $\leq$-minimal element $\rho' \in \mathrm{csupp}(\pi')$ such that $\rho' \not\in \mathrm{csupp}(\nu^{1/2}\cdot \pi)$; and $\rho'$ is relevant to $\pi'$. \\ 

\noindent
{\bf Step 2:  Compute $\mathrm{Hom}_{G_n}(\pi, \omega)$ via the standard trick ($\omega$ is defined below)} \\

 Let $\mathfrak p'=\mathfrak{mxpt}^L( \pi', \rho')$. By using $D^R_{\mathfrak m}(\nu^{1/2}\cdot \pi)\cong D^L_{\mathfrak n}(\pi')$, we must have $D^L_{\mathfrak n} (\pi')$ to be strongly L-$\rho'$-reduced. Then, we have $\mathfrak p' \subset \mathfrak n$ by (the left version of) Lemma \ref{lem delta reduced subset}. Then, by Corollary \ref{cor seq strong commut under minimal and sub}, we still have that $(\mathfrak m, \mathfrak n-\mathfrak p', D_{\mathfrak p'}^L(\pi))$ is still a strongly RdLi-commutative triple and
\[   D^R_{\mathfrak m}(\nu^{1/2}\cdot \pi) \cong  D^L_{\mathfrak n-\mathfrak p'}\circ D^L_{\mathfrak p'}(\pi') .
\]
 Thus $(\pi, D^L_{\mathfrak p'}(\pi'))$ is still relevant. 

Now we pick a cuspidal representation $\sigma$ of $G_{l_a(\mathfrak p')}$ good to $\nu^{1/2}\cdot \pi$ and $\pi'$. Then we have that $(\pi, \sigma \times D^L_{\mathfrak p'}(\pi'))$ is still relevant  (determined by the triple $(\mathfrak m, \mathfrak n-\mathfrak p'+[\sigma], \nu^{1/2}\cdot \pi)$ as shown above). Now, $\mathcal{RR}(\pi, \sigma \times D^L_{\mathfrak p'}(\pi'))<\mathcal{RR}(\pi, \pi')$, by induction, we have that:
\[  \mathrm{Hom}_{G_n}(\pi, \sigma \times D^L_{\mathfrak p'}(\pi')) \neq 0 .
\]
Let $\omega = \mathrm{St}(\mathfrak p') \times D^L_{\mathfrak p'}(\pi') $. The standard trick of Lemma \ref{lem construct one more non-zero element} now implies that
\[  \mathrm{Hom}_{G_n}(\pi, \omega )\cong \mathbb C  
\]
and $\mathcal L_{rBL}(\pi, \omega)=i^*$. \\

\noindent
{\bf Step 3: Define $\tau'$ and its quotient $\tau$.} \\

On the other hand, let $\Delta$ be a $\preceq^R$-maximal (see Section \ref{ss reduced and saturated parts}) segment in $\mathfrak m$. Let $\mathfrak p=\mathfrak{mx}(\pi, \Delta)$ and let $\tau'= \mathrm{St}(\mathfrak p) \times D_{\mathfrak p}(\pi)$. We now consider the filtration of $\pi$ as given in Proposition \ref{prop filtration on parabolic}. Let $\kappa$ be the submodule of $\tau'$ containing all the layers 
\[    \mathrm{St}(\mathfrak p)^{[k]} \times \mathrm{RS}_{k-2}(D_{\mathfrak p}(\pi)) 
\]
for $k > l_a(\mathfrak p)-l_a(\mathfrak r) \geq 1$. 

Now let $\tau=\tau'/\kappa$. We pick $\sigma' \in \mathrm{Irr}^c$ good to $\pi$ and $\nu^{-1/2}\pi'$. By Corollary \ref{cor seq strong commut under minimal}, $(\sigma'\times D_{\Delta}(\pi), \pi')$ is still relevant. Hence, we can apply induction. By using Lemma \ref{lem construct branching from deformation} on $\mathrm{Hom}_{G_n}(\sigma' \times D_{\Delta}(\pi), \pi')\neq 0$ with $\mathcal L_{rBL}(\sigma'\times D_{\Delta}(\pi), \pi')=l_a(\mathfrak m)$, we have a non-zero $\widetilde{f}$ in 
\[   \mathrm{Hom}_{G_n}(\mathrm{St}(\Delta)\times D_{\Delta}(\pi), \pi') 
\]
with $\mathcal L_{rBL}(\widetilde{f})\leq l_a(\mathfrak m)$. Thus,  from the quotient map $\mathrm{St}(\mathfrak p)\times D_{\mathfrak p}(\pi)$ to $\mathrm{St}(\Delta)\times D_{\Delta}(\pi)$, we also obtain a non-zero map $\bar{f}$ in $\mathrm{Hom}_{G_n}(\mathrm{St}(\mathfrak p)\times D_{\mathfrak p}(\pi), \pi')$ with $\mathcal L_{rBL}(\bar{f}) \leq l_a(\mathfrak m)$. By Lemma \ref{lem a branching factor through quotient}, we then have
\[ (*) \quad  \mathrm{Hom}_{G_n}(\tau, \pi') \neq 0 .
\]
We want to show that the Hom has dimension one (see Claim 2 below). \\

\noindent
{\bf Step 4.} {\it Claim 1:} $\mathrm{dim}~\mathrm{Hom}_{G_n}(\tau, \omega) \leq 1$. \\

\noindent
{\it Proof of claim 1:} Note that $(\Delta, \mathfrak p', I_{\Delta}^R\circ D^L_{\mathfrak p'}(\pi'))$ is strongly RdLi-commutative (by Example \ref{example pre commutative}(2)), and so, by Lemma \ref{lem strong commutative combin}, 
\[  \eta_{\Delta}(I_{\Delta}^R\circ D^L_{\mathfrak p'}(\pi))=\eta_{\Delta}(I_{\mathfrak p'}\circ I_{\Delta}^R\circ D^L_{\mathfrak p'}(\pi))=\eta_{\Delta}(I_{\Delta}^R(\pi)) 
\]
and so
\[ \eta_{\Delta}(D^L_{\mathfrak p'}(\pi')) =\eta_{\Delta}(\pi') .
\]
Thus $\mathfrak{mx}(D^L_{\mathfrak p'}(\pi'), \Delta)=\mathfrak{mx}(\pi, \Delta)$ and we let $\mathfrak r=\mathfrak{mx}(\pi, \Delta)$. In Proposition \ref{prop filtration on parabolic}, the layers for $(\mathrm{St}(\mathfrak p)\times D^L_{\mathfrak p}(\pi))|_{G_n}$ take the form:
\[   \mathrm{St}(\mathfrak p)^{[k]} \times  \lambda\]
and so, by Lemma \ref{lem compute layers by geometric lemma} (1), the only possible layer that can contribute to a non-zero element in $\mathrm{Hom}_{G_n}(\tau, \omega)$ is 
\[  \mathrm{St}(\mathfrak p)^{[k^*]}\times  \mathrm{RS}_{k^*-2}(D_{\mathfrak p}(\pi)) ,
\]
where $k^*=l_a(\mathfrak p)-l_a(\mathfrak r)$.

Recall that $\omega=\mathrm{St}(\mathfrak p')\times D^L_{\mathfrak p'}(\pi')$. Now, combining with Lemma \ref{lem compute layers by geometric lemma} (2),
\[  \mathrm{dim}~\mathrm{Hom}_{G_n}(\tau, \omega) \leq \mathrm{dim}~\mathrm{Hom}_{G_{n-l_a(\mathfrak p)}}( \mathrm{RS}_{k^*-2}(D_{\mathfrak p}(\pi)), \mathrm{St}(\mathfrak p') \times D_{\mathfrak r}\circ D^L_{\mathfrak p'}(\pi')) .
\]
The last term has dimension at most one by Lemma \ref{lem multiplicity one of standard module} and so this proves the claim. \\

\noindent
{\bf Step 5.} {\it Claim 2:} Let $\omega_1 =\pi'$. Then 
\[ \mathrm{dim}~\mathrm{Hom}_{G_n}( \tau , \omega_1) = \mathrm{dim}~\mathrm{Hom}_{G_n}(\tau, \omega)=1 .
\]

\noindent
{\it Proof of claim 2:} This immediately follows from Claim 1 that 
\[  \mathrm{dim}~\mathrm{Hom}_{G_n}(\tau, \omega_1) \leq \mathrm{dim}~\mathrm{Hom}_{G_n}(\tau, \omega) \leq 1 .
\]
The equality part now follows from (*). \\

\noindent 
{\bf Step 6.} {\it Claim 3:} Let $\tau_1$ be the quotient of $\tau$ coming from the projection of $\tau'$ to $\pi$ i.e. $\tau_1$ is the pushout of two surjections $\tau' \rightarrow \pi$ and $\tau' \rightarrow \tau$. Then 
\[  \mathrm{dim}~\mathrm{Hom}_{G_n}(\tau_1, \omega) =1 .
\]

\noindent
{\it Proof of claim 3:} By Claim 1, we only have to show that $\mathrm{Hom}_{G_n}(\tau_1, \omega)\neq 0$. To this end, we appeal to Lemma \ref{lem construct one more non-zero element} that there exists a non-zero element $f$ in $\mathrm{Hom}_{G_n}(\pi, \omega)$ with $\mathcal L_{rBL}(\pi, \omega)=l_a(\mathfrak m)$. 

We consider the following commutative diagram:
\[  \xymatrix{ \kappa \ar@{^{(}->}[r]     \ar@{->>}[d]   & \tau'  \ar@{->>}[d]^q  &     \\
               \widetilde{\kappa} \ar@{^{(}->}[r]  & \pi    \ar[r]^f      & \omega                                        },
\]
where $\widetilde{\kappa}=q(\kappa)$. By Lemma \ref{lem integer determine quotient}, $\mathcal L_{rBL}(f\circ q)=l_a(\mathfrak m)=i^*$. Thus, by Lemma \ref{lem a branching factor through quotient}, $(f\circ q)|_{\kappa}=0$. This implies that $f|_{\widetilde{\kappa}} =0$. Since $\tau_1=\pi/\widetilde{\kappa}$, $f$ lifts to a non-zero map in $\mathrm{Hom}_{G_n}(\tau_1, \omega)$ as desired. \\

\noindent
{\bf Step 7: Complete Case 1 by the strategy.} Now, we return to the proof. By using Claims 2 and 3, Lemma \ref{lem general strategy} implies that $\mathrm{Hom}_{G_n}(\tau_1, \omega_1) \neq 0$. Since $\tau_1$ is a quotient of $\pi$, we then have that $\mathrm{Hom}_{G_n}(\pi, \pi') \neq 0$ (recall $\omega_1=\pi'$). This proves Case 1. \\

\noindent
{\bf Step 8: Use the duality to transfer Case 2 to Case 1.} \\
\noindent
{\bf Case 2:} Suppose we are not in Case 1. Recall that we are assuming  $\mathcal{RR}(\pi, \pi')>0$. Then there exists a $\leq$-minimal $\rho' \in \mathrm{csupp}(\pi)$ such that $\rho' \not\in \mathrm{csupp}(\nu^{1/2}\pi')$ and $\rho'$ is relevant to $\pi'$. (Otherwise, if we have $\rho' \in \mathrm{csupp}(\nu^{1/2}\pi')$, then $\nu^{-1/2}\rho' \in \mathrm{csupp}(\pi')$. But we also have $\nu^{-1/2}\rho' \not\in \mathrm{csupp}(\nu^{1/2}\cdot \pi)$ by minimality and then it gives a contradiction.)

Now applying $\vee$ and $\theta$, we have that there exists a $\leq$-minimal element $\rho'' \in \mathrm{csupp}(\theta(\pi)^{\vee})$ such that $\rho'' \not\in \mathrm{csupp}(\nu^{1/2}\theta(\pi')^{\vee})$. 

On the other hand,
\begin{align}
 \mathrm{Hom}_{G_n}(\pi, \pi') & \cong \mathrm{Hom}_{G_{n+1}}(\sigma \times \pi'{}^{\vee}, \pi^{\vee})       \\ 
                       & \cong \mathrm{Hom}_{G_{n+1}}(\theta(\sigma) \times \theta(\pi'{}^{\vee}), \theta(\pi^{\vee}))  \\
  \label{eqn hom in case 2}    & \cong \mathrm{Hom}_{G_{n+1}}(\theta(\sigma)\times \pi', \pi) .
\end{align}

Here $\sigma$ is a certain cuspidal representation of $G_2$ such that $\sigma \times \pi'{}^{\vee}$ is irreducible and $\nu^{1/2} \cdot \sigma$ is good to $\pi^{\vee}$. The first isomorphism follows from a duality \cite[Proposition 4.1]{Ch22} and the second isomorphism follows by applying $\theta$-action, and $\theta(\sigma \times \pi'{}^{\vee})\cong \theta(\pi'{}^{\vee})\times \theta(\sigma)\cong \theta(\sigma)\times \theta(\pi'{}^{\vee})\cong \theta(\sigma)\times \pi'$. The last isomorphism follows by the isomorphism for the Gelfand-Kazhdan involution.

We have that $\mathcal{RR}(\theta(\sigma) \times \pi'{}, \pi)=\mathcal{RR}(\pi, \pi')$. We also have that $(\theta(\sigma)\times \pi', \pi)$ is relevant by Theorem \ref{thm symmetric property of relevant}. Now the above discussions justify that we can use Case 1 to conclude the Hom in (\ref{eqn hom in case 2}) is non-zero. This implies that $\mathrm{Hom}_{G_n}(\pi, \pi')$ is non-zero.
\end{proof}


\part{Proof of necessity of generalized relevance} \label{part necessity part}
For an overview of this section, see Sections \ref{ss outline necc} and \ref{ss outline exh}. Section \ref{s branching from bz derivatives} studies a special type of branching laws related to derivatives, in which a technique in Section \ref{s taking highest derivatives in BL} is needed. Section \ref{s deform and thicken simple quotients} studies a characterization of the module whose highest derivative gives a simple quotient of a BZ derivative (Corollary \ref{cor unique highest derivative to give submodule}). Sections \ref{s exhaustion relev conditions} and \ref{s necessarity and exhaustion} prove our main results.

\section{Taking highest derivatives in branching laws} \label{s taking highest derivatives in BL}

We first illustrate the idea of taking highest derivatives. Let $\pi$ be a smooth representation of the mirabolic subgroup $M_{n+1}$. Let $\tau$ be a $G_n$-quotient of $\pi|_{G_n}$. Then, as $G_n$-representations, one can take a derivative to obtain a natural projection:
\[ f:  {}^{(j)}(\pi|_{G_n}) \twoheadrightarrow {}^{(j)}\tau .
\]
(For a precise realization on those derivatives to make sense of the projection, see Section \ref{ss taking highest derivative on mir}.) We write ${}^{\langle j \rangle}\pi :={}^{(j)}(\pi|_{G_n})$. On the other hand, one can also take 'left $j$-th derivative', denoted ${}^{(j)}\pi$, on $\pi$ as $M_{n+1}$ (close to the one for $G_{n+1}$, and for a precise description, see (\ref{eqn left derivative on mirabolic}) below), which then gives a projection:
\[ p: {}^{\langle j \rangle}\pi \twoheadrightarrow {}^{(j)} \pi .
\]
The question is when $f$ factors through $p$. 

We are interested in the branching law case and so we consider $\pi \in \mathrm{Irr}(G_{n+1})$ (then regarded as a $M_{n+1}$-representation) which is thickened (see Definition \ref{def thickened multisegments} below) and $i=\mathrm{lev}(\pi)$. For a $G_n$-quotient $\tau$ of $\pi$, if the induced map $f$ from $f': \pi \rightarrow \tau$ factors through such $p$ above, we refer this to {\it deforming the branching law $f$}. Such a deformation {\it does not happen in general}. We show that when $\tau$ takes the form $\tau' \times \sigma$ for some simple quotient $\tau'$ of $\pi^{(i)}$ (for some $i$) and some suitable choice of $\sigma \in \mathrm{Irr}^c$, such the deformation happens (see Lemma \ref{lem projection factor through}). This relies on some analysis on the layers arising from the geometric lemma (or Leibniz's rule). As a result, we constructed a branching law $\widetilde{f}: {}^-\pi \rightarrow {}^-\tau' \times \sigma$. 

As shown in Part \ref{part proof sufficiency}, it is useful to determine the layer supporting a branching law. Indeed, in Lemma \ref{lem branching bz der thickened} (in next section), one can use some analysis on Bernstein-Zelevinsky layers to show that $\mathcal L_{rBL}(f)=i$, and $\mathcal L_{lBL}(f)=\mathrm{lev}(\pi)$. It turns out that a similar argument can prove that $\mathcal L_{rBL}(\widetilde{f})=i$ in the deformed branching law. Determining $\mathcal L_{lBL}(\widetilde{f})$ needs some other works (this is also the reason we have to consider the deformation). For this, in Section \ref{ss derivative mirabolic rest}, we study the derivative on the left BZ filtration ${}_k\Lambda(\pi)$ for $\pi$, which will then be used to show that $\mathcal L_{lBL}(\widetilde{f})=\mathrm{lev}(\pi)$ (see Proposition \ref{prop transitive smallest integer} and Section \ref{ss proof of branching from bz}). 

We finally also mention that Offen \cite[Section 6]{Of20} asks similar questions on the effect of taking highest derivatives in the content of Sp-distinction.

\subsection{Taking highest derivatives for mirabolic subgroup representations} \label{ss taking highest derivative on mir}


We use the notations in Section \ref{s bz filtration}. For the purpose of exposition, it is more convenient to use the definition of BZ derivatives in the form of coinvariants, which we are going to formulate. Let 
\[ \bar{R}_i= \left\{ \begin{pmatrix} u & x \\ & I_{n+1-i} \end{pmatrix}: u \in U_i, x\in \mathrm{Mat}_{i,n+1-i}  \right\} \subset G_{n+1}.
\]
Fix a non-trivial character $\psi_i: U_i \rightarrow \mathbb C$ and extend trivially to a character $\psi_i'$ of $\bar{R}_i$. For $\pi \in \mathrm{Alg}(M_{n+1})$ and some integer $i \leq n$, we define (by abuse of notations) 
\begin{align} \label{eqn left derivative on mirabolic}
 {}^{(i)}\pi =\pi_{\bar{R}_j, \psi_j} = \delta_{\bar{R}_i}^{-1/2}\cdot \frac{\pi}{\langle u.x-\psi_i'(u)x :u \in \bar{R}_i, x\in \pi \rangle} 
\end{align}
regarded as a $M_{n+1-i}$-representation via the embedding $m \mapsto \mathrm{diag}(I_i, m)$. This definition is equivalent (up to a natural isomorphism) to the one in Section \ref{s bz filtration} and we shall use such realization in this section.

We now define a smaller subgroup $\bar{S}_i$ in $\bar{R}_i$:
\[ \bar{S}_i= \left\{ \begin{pmatrix} u & x' & 0 \\ & I_{n-i} & \\ & & 1 \end{pmatrix} : u \in U_i, x' \in \mathrm{Mat}_{i, n-i} \right\} .
\]
We similarly define:
\[  {}^{\langle i \rangle}\pi= \delta_{\bar{S}_i}^{-1/2}\cdot \frac{\pi}{\langle u.x-\psi_i'(u)x :u \in \bar{S}_i, x\in \pi \rangle} ,
\]
regarded as a $G_{n-i}$representation via the embedding $g \mapsto \mathrm{diag}(I_i, g,1)$. Note that ${}^{\langle i \rangle}\pi={}^{(i)}(\pi|_{G_n})$ (recall we embed $G_n$ to $G_{n+1}$ by $g\mapsto \mathrm{diag}(g,1)$). 

We are going to study the natural projection: ${}^{\langle i\rangle}\pi \rightarrow {}^{(i)}\pi$. In other words, that is to study how certain derivative for $G_n$ can be extended to a derivative of $G_{n+1}$. 

We first give a computation on ${}^{(j)}$-derivative on $\Gamma(\pi)$, which follows from a direct computation from Mackey theory:

\begin{lemma} \label{lem transitive layer in derivative}
Let $\pi \in \mathrm{Alg}(G_n)$. Then ${}^{(j)}\Gamma^{k+1}(\pi)\cong \Gamma^{k+1}({}^{(j)}\pi)$.
\end{lemma}

\begin{proof}
One can proceed by taking Jacquet functors in stage. Take $N$ to be the unipotent subgroup containing matrices of the form:
\[    \begin{pmatrix} I_j & * \\ & I_{k+1+n-j} \end{pmatrix} .
\]
 However, it follows from a similar computation as \cite[Lemma 4.5]{Ch21} that 
\[ \Gamma^{k+1}(\pi)_N \cong \Gamma^{k+1}(\pi_{N_{j,n-j}}) , \]
where $\Gamma^{k+1}(\pi_{N_{j,n-j}})$ is considered to take the functor $\Gamma^{k+1}$ in the second factor of $G_j\times G_{n-j}$. Then one further takes the Whittaker model on the first factor to obtain ${}^{(j)}\Gamma^{k+1}(\pi) \cong \Gamma^{k+1}({}^{(j)}\pi)$.
\end{proof}

\begin{lemma} \label{lem descending a bz layer}
Let $\pi$ be a representation of $G_m$ of finite length. Let $\kappa= \Gamma^{k+1}(\pi)$. Let $\tau$ be a simple $G_m$-quotient of $\pi$. Let $\sigma \in \mathrm{Irr}(G_k)$ good to $\nu^{1/2}\pi$. Let $j=\mathrm{lev}(\tau)$. Let $f: \Gamma^{k+1}(\pi)  \rightarrow (\nu^{1/2}\cdot\tau) \times \sigma$. As $G_{m+k}$-representations, this induces a map $\widetilde{f}:{}^{\langle j \rangle}\kappa \rightarrow \nu^{1/2}({}^-\tau) \times \sigma$. Then $\widetilde{f}$ factors through the projection map ${}^{\langle j\rangle}\kappa \rightarrow {}^{(j)}\kappa$. 
\end{lemma}

\begin{proof}
Let 
\[  Q =\left\{ \begin{pmatrix} u & x & v \\ & g & v' \\ & & 1 \end{pmatrix} : u \in U_j, x \in Mat_{j,m+k-j}, v \in F^j, v' \in F^{m+k-j}, g \in G_{m+k-j} \right\}
\]
and let 
\[ P =\left\{ \begin{pmatrix} g & v \\ & u \end{pmatrix} : g \in G_j, v \in Mat_{j,m+k+1-j}, u \in U_{m+k+1-j} \right\}.  
\]
Let $\lambda$ be the space of smooth compactly-supported functions from $PQ$ to $\pi\boxtimes \psi_{m+k+1-j}$. 

Let $\widetilde{\tau}=\nu^{1/2}\tau$. Since ${}^-\widetilde{\tau}\times \sigma$ appears in the top layer (in the filtration from the geometric lemma) of ${}^{(j)}(\widetilde{\tau} \times \sigma)$, ${}^-\widetilde{\tau} \times \sigma$ also appears in the top layer of ${}^{\langle j \rangle}\kappa$ via the functoriality of the geometric lemma. In other words, with taking the twisted Jacquet functor, we have the following commutative diagram:
\[  \xymatrix{  {}^{\langle j \rangle}\kappa                 \ar[r]  \ar[d]    &  {}^{\langle j \rangle}\lambda  \ar[d] \\
              {}^-(\widetilde{\tau}\times \sigma) \ar[r] & ({}^-\widetilde{\tau})\times \sigma
}  .
\]
Hence, we have that $\widetilde{f}$ factors as:
\[ \widetilde{f}: {}^{\langle j\rangle} \kappa \rightarrow {}^{\langle j \rangle}\lambda \rightarrow ( {}^-\widetilde{\tau})\times \sigma .
\]
Now projecting ${}^{\langle j \rangle}\kappa$ to ${}^{(j)}\kappa$, we have that
\[  {}^{( j )}\kappa= {}^{( j)}\lambda ={}^{\langle j \rangle}\lambda
\]
as vector spaces, where the first equality follows from an application of the geometric lemma (see e.g. \cite[5.9]{BZ77}), c.f. Lemma \ref{lem transitive layer in derivative} and the second follows from definitions. Thus, the map $\widetilde{f}$ factors through the projection from ${}^{\langle j\rangle}\kappa$ to ${}^{(j)}\kappa$.
\end{proof}

\begin{definition} \label{def thickened multisegments}
A multisegment $\mathfrak m$ is said to be {\it thickened} if any segment has relative length at least $2$. An irreducible representation $\pi \in \mathrm{Irr}$ is said to be {\it thickened} if $\pi \cong \langle \mathfrak m \rangle$ for some thickened multisegment $\mathfrak m$. 
\end{definition}

As also seen in \cite[Section 7]{Ch21}, the thickened case has good combinatorics and the general case needs more substantial work. 

\begin{lemma} \label{lem projection factor through}
  Let $\pi \in \mathrm{Irr}(G_{n+1})$ be thickened. Let $\tau$ be a simple quotient of $\pi^{[i]}$. Let $j=\mathrm{lev}(\pi) (=\mathrm{lev}(\tau))$. Let $\sigma \in \mathrm{Irr}^c(G_{i-1})$ be good to $\nu^{1/2}\cdot \pi$. Then the following holds:
\begin{itemize}
\item $\mathrm{Hom}_{G_n}(\pi, \tau \times \sigma) \neq 0$.
\item Let $f': {}^{\langle j \rangle}\pi \rightarrow {}^-\tau \times \sigma$ be the induced $G_{n-i}$-map from a non-zero map $f: \pi \rightarrow \tau \times \sigma$. Then $f'$ factors through the projection ${}^{\langle j\rangle}\pi$ to ${}^{(j)}\pi$. 
\end{itemize}

\end{lemma}

\begin{proof}
By a standard argument of the BZ filtration (also see \cite[Lemma 1.1]{Ch22+}), we have an isomorphism:
\[  \mathbb C\cong \mathrm{Hom}_{G_n}(\pi, \tau \times \sigma) \rightarrow \mathrm{Hom}_{G_n}(\Lambda_{i-1}(\pi), \tau \times \sigma) . \]
Let $f$ be a non-zero map in $\mathrm{Hom}_{G_n}(\pi, \tau \times \sigma)$ and let the corresponding lift in 
\[ \mathrm{Hom}_{G_n}(\Lambda_{i-1}(\pi), \tau \times \sigma)\]
be $\widetilde{f}$. Taking the twisted Jacquet functor ${}^{\langle j\rangle}$ (and composing with the projection to ${}^-\tau\times \sigma$, we then obtain a map $f': {}^{\langle j \rangle}\pi\rightarrow {}^-\tau\times \sigma$. We then obtain a (non-zero) lift $\widetilde{f}'$ on $\mathrm{Hom}_{G_{n-j}}({}^{\langle j \rangle}\Lambda_{i-1}(\pi), {}^-\tau \times \sigma)$ by the above isomorphism. Since the map $\widetilde{f}$ comes from $\mathrm{Hom}_{G_n}(\Sigma_i(\pi), \tau \times \sigma)$, Lemma \ref{lem descending a bz layer} implies that $\widetilde{f}'$ factors through the projection ${}^{\langle j\rangle}\Lambda_{i-1}(\pi)$ to ${}^{(j)}\Lambda_{i-1}(\pi)$. Thus we also have a corresponding non-zero map, denoted $g'$, in $\mathrm{Hom}_{G_n}({}^{(j)}\Lambda_{i-1}(\pi), {}^-\tau\times \sigma)$.

Let $\omega$ be the cokernel of the embedding $\Lambda_{i-1}(\pi)$ to $\pi$. 
We consider the following commutative diagram:
\[ \xymatrix{ \mathrm{Hom}_{G_{n-j}}({}^{\langle j\rangle}\pi, {}^-\tau \times \sigma) \ar[r] & \mathrm{Hom}_{G_{n-j}}({}^{\langle j\rangle}\Lambda_{i-1}(\pi), {}^-\tau \times \sigma) & \\
                \mathrm{Hom}_{G_{n-j}}({}^{(j)}\pi, {}^-\tau \times \sigma) \ar[r] \ar[u] & \mathrm{Hom}_{G_{n-j}}({}^{(j)}\Lambda_{i-1}(\pi), {}^-\tau \times \sigma) \ar[u] \ar[r] & \mathrm{Ext}^1_{G_{n-j}}({}^{(j)}\omega, {}^-\tau \times \sigma)                                                     }
\]
 The Ext-group in the diagram is zero by an Ext group computation on the layers ${}^{(j)}\Sigma_{i'}(\pi)=\Sigma_{i'}({}^{(j)}\pi)$ ($i'<i$) (here we use Lemma \ref{lem transitive layer in derivative}) and a cuspidal support comparison (see the proof of Lemma \ref{lem branching bz der thickened} for more details). This implies that there is a lift of $g'$ in the bottom left corner Hom. It follows from the commutative diagram that the lift gives the desired map.

From the commutative diagram above, it remains to show that the image of $g'$ in the top left corner agrees with $f'$. However, this would follow if we can show that $\mathrm{Hom}_{G_{n-j}}({}^{\langle j \rangle}\omega, {}^-\tau \times \sigma)=0$, where $\omega=\pi/\Lambda_{i-1}(\pi)$. To this end, from the BZ filtration, layers in $\omega$ takes the form $\pi^{[k]}\times \Pi_{n-k}$ for $k<i$. After taking the twisted Jacquet functor, those layers for ${}^{\langle j \rangle}\omega$ take the form $({}^{(j_1)}(\pi^{[k]})\times \Pi_{k-j_2})\otimes \mathfrak{Z}_{j_2}$, where $j_1, j_2$ run for all $j_1+j_2=j$ and $\mathfrak{Z}_{j_2}$ is the Bernstein center of $G_{j_2}$. Applying the second adjointness and comparing cuspidal support, we always have that:
\[  \mathrm{Hom}_{G_{n-j}}(({}^{(j_1)}(\pi^{[k]})\times \Pi_{k-j_2})\otimes \mathfrak{Z}_{j_2}, {}^-\tau \times \sigma)=0 .
\]
This implies that $\mathrm{Hom}_{G_{n-j}}({}^{\langle j_1\rangle}\omega, {}^-\tau \times \sigma)=0$ as desired.
\end{proof}

\subsection{Derivatives on another mirabolic restriction} \label{ss derivative mirabolic rest}


We now consider another mirabolic restriction and so we are considering $M_{n+1}^t$-modules. (Here $M_{n+1}^t$ denotes the transpose of $M_{n+1}$.) To avoid confusion and further abuse of notations, we shall simply use $\pi_{\bar{S}_j, \psi}$ to consider the co-invariant spaces.

Let $\Pi_i=\mathrm{ind}_{U_i}^{G_i}\psi_i$ be the Gelfand-Graev representation of $G_i$. By the induction in stages, for $\tau \in \mathrm{Alg}(G_j)$, we have ${}^i\Gamma(\tau)|_{G_{i+j-1}}\cong \Pi_{i-1}\times \tau$.

\begin{lemma} \label{lem composition factors of applying geo on mirabolic}
We regard $M_{n-j+1}^t$ as a subgroup $M_{n+1}^t$ via $m\mapsto \mathrm{diag}(I_j, m)$. Note that $\bar{S}_j$ is invariant under conjugation by $M_{n-j+1}^t$. Let $\tau \in \mathrm{Alg}(G_{n+1-i})$. Then 
\[ ({}^i\Gamma(\tau))_{\bar{S}_j, \psi_j} \]
admits a $M_{n-j+1}^t$-filtration of the form: with $j_1+j_2=j$,
\[  {}^{i-j_1}\Gamma({}^{(j_2)}\tau) \otimes \mathfrak Z_{j_1},
\]
where $\mathfrak Z_{j_1}$ is the Bernstein center $G_{j_1}$. Moreover, when restricting to $G_{n-j}$-representations, the filtration agrees with the filtration obtained by applying Leibniz's rule on ${}^{(j)}(\Pi_{i-1}\times \tau)$. 
\end{lemma}

\begin{proof}
Let $R_i$ be the subgroup containing matrices of the form $\begin{pmatrix} g & * \\ & u \end{pmatrix}$ for $g \in G_{n+1-i}, u \in U_i$. Let $w= \mathrm{diag}(\begin{pmatrix} & J_i  \\ I_{n-i} &  \end{pmatrix},1)$, where $J_i$ is the matrix with $1$ in the anti-diagonal and $0$ elsewhere. In such case, $wR_i^tw^{-1}$ contains all matrices of the form:
\[  \begin{pmatrix} u & * & \\  & g & \\ * & * & 1 \end{pmatrix} ,
\]
where $u \in U_{i-1}$ and $g \in G_{n+1-i}$. 
Then, using the element $w$, as a $G_n$-representation, we have that:
\[  \mathrm{ind}_{R_i^t}^{M^t_n} (\tau \boxtimes \psi_i) \cong \Pi_{i-1}\times \tau .
\]
Now one applies the derivative and Leibniz's rule on $\Pi_{i-1}\times \tau$ and uses ${}^{(j_1)}(\Pi_{i-1})=\Pi_{i-1-j_1}\otimes \mathfrak Z_{j_1}$, and then we obtain a filtration of the form 
\[ (\Pi_{i-1-j_1}\times {}^{(j_2)}\tau) \otimes \mathfrak Z_{j_1} \]
 with $j_1+j_2=j$. Now one imposes the action of $M_{n-j+1}^t$ and obtains such form.
\end{proof}

\begin{lemma} \label{lem submodule embedding after derivative}
Let $\pi \in \mathrm{Irr}(G_{n+1})$ be thickened. Let $\tau={}^-({}^-\pi)$ (i.e. taking the highest derivatives twice). Let $j^*=\mathrm{lev}(\pi)$. Let $m=n+1-j^*$. Then, as shown in Lemma \ref{lem composition factors of applying geo on mirabolic}, $({}_{j^*}\Sigma(\pi))_{\bar{S}_{j^*},\psi_{j^*}}$ has a submodule ${}^{j^*}\Gamma(\tau)$, as $M_m$-modules. Denote such the submodule map by $\iota$. Let $p: ({}_{j^*}\Sigma(\pi))_{\bar{S}_{j^*},\psi_{j^*}} \rightarrow {}^{\langle j^*\rangle}\pi$ be the embedding induced from the BZ filtration. Denote the projection map from  ${}^{\langle j^*\rangle}\pi$ to ${}^{(j^*)}\pi={}^-\pi$ by $q$. Then $q\circ p\circ \iota \neq 0$. 
\end{lemma}

\begin{proof}
Note that we have a surjection from $\pi_{\bar{S}_{j^*}, \psi_{j^*}}$ to ${}^-\pi$, which is $M_m^t$-equivariant. Thus we only have to observe that other composition factors of $({}_{j^*}\Sigma(\pi))_{\bar{S}_{j^*}, \psi_{j^*}}$ as well as any composition factor of $({}_j\Sigma(\pi))_{\bar{S}_{j^*},\psi_{j^*}}$ ($j<j^*$) does not take the form ${}^{j^*}\Gamma(\kappa)$ for some $\kappa \in \mathrm{Irr}$. Indeed, the last assertion follows from the description of those composition factors in Lemma \ref{lem composition factors of applying geo on mirabolic}.
\end{proof}

\subsection{Derivative-descentable quotient on another BZ filtrations}

\begin{lemma} \label{lem coming from smallest mirabolic}
Let $\omega \in \mathrm{Irr}(G_k)$. Let $\kappa \in \mathrm{Alg}(G_l)$. Let $\lambda$ be an irreducible quotient of $\kappa \times \omega$. Let $j^*=\mathrm{lev}(\omega)$. Let $q$ be the quotient map and let $q': {}^{(j^*)}(\kappa \times \omega) \rightarrow  {}^{(j^*)}\lambda$ be the induced map from $q$. Let $\iota: \kappa \times {}^-\omega \hookrightarrow {}^{(j^*)}(\kappa \times \omega) $ be the embedding from the bottom layer of Leibniz's rule. Then 
\begin{enumerate}
\item[(1)] $q'\circ \iota\neq 0$;
\item[(2)] Suppose $\kappa$ is of finite length. Suppose $\lambda$ can be written as $\lambda_1\times \lambda_2$ such that $\mathrm{csupp}(\lambda_1)\cap \mathrm{csupp}(\omega)=\emptyset$ and $\mathrm{csupp}(\lambda_1)\cap \mathrm{csupp}(\lambda_2)=\emptyset$. Then, we further have a map:
\[   s: {}^{(j^*)}(\lambda_1 \times \lambda_2)\twoheadrightarrow  \lambda_1 \times {}^{(j^*)}\lambda_2.
\]
Then $s\circ q' \circ \iota \neq 0$.
\end{enumerate}

\end{lemma}

\begin{proof}
Let $P=P_{l,k}$ and let $U=U_{k+l}$. Let $N_P^-$ be the opposite unipotent radical of $P$.

We realize $\kappa \times \omega$ as the space $C^{\infty}_c(P\setminus G_{k+l}, \kappa \boxtimes \omega)$ of smooth compactly-supported $\kappa \boxtimes \omega$-valued functions. Let $C^{\infty}_c(P\setminus PU^t, \kappa\boxtimes \omega)$ as its subspace of compactly-supported functions with support in $P\setminus PU^t$. The following composition 
\[  C^{\infty}_c(P \setminus PU^t, \kappa \boxtimes \omega) \hookrightarrow \kappa \times \omega \twoheadrightarrow \lambda
\]
is non-zero, see Bezrukavnikov-Kazhdan \cite[Section 6.2]{BK15} or see \cite[Section 8]{Ch22+c}.


One now twists the above sequence with the action from an element $w=\begin{pmatrix} & I_l \\ I_k &  \end{pmatrix}$, from $G_k \times G_l$-action to $G_l\times G_k$-action. Then we obtain a non-zero composition, after taking the Jacquet functor $w^{-1}N_P^-w=N_{P_{k,l}}$:
\[  C^{\infty}_c(P\setminus PU^tw , \kappa \boxtimes \omega) \hookrightarrow \kappa \times \omega \twoheadrightarrow \lambda .
\]
The composition takes the form:
\[    \omega \boxtimes \kappa  \rightarrow \lambda_{N_{P_{k,l}}}  .
\]
Using Jacquet functor in stages, this implies that after applying the twisted Jacquet functor ${}_{\bar{R}_{j^*}, \psi_{j^*}}$, we also have the following non-zero composition (see \cite[Section 8]{Ch22+c} for more discussions): (set $\tau=\kappa \boxtimes \omega$)
\[ \xymatrix{ C^{\infty}_c(P\setminus PU^tw , \tau)_{\bar{R}_{j^*}, \psi_{j^*}} \ar@{=}[r]  & C^{\infty}_c(P\setminus PU^tw\bar{R}_{j^*}, \tau)_{\bar{R}_{j^*}, \psi_{j^*}} \ar@{^{(}->}[r] \ar@{=}[d] & (\kappa \times \omega)_{\bar{R}_{j^*}, \psi_{j^*}} \ar@{->>}[r] \ar@{=}[d] & \lambda_{\bar{R}_{j^*}, \psi_{j^*}} \\
    &   \kappa \times {}^-\omega &  {}^{(j^*)}(\kappa\times \omega)} 
\]
and so we obtain (1). Here taking the twisted Jacquet functor ${}_{\bar{R}_{j^*}, \psi_{j^*}}$ on $\kappa\times \omega$ is the same as taking the BZ derivative.

For (2), $\lambda_1\times {}^{(j^*)}\lambda_2$ is an indecomposable component in ${}^{(j^*)}(\lambda_1\times \lambda_2)$ by the cuspidal support condition and hence gives the map $s$. Then, by a comparison of cuspidal support and (1), the image of $q \circ \iota'$ must lie in that component and so this gives (2).
\end{proof}

\begin{proposition} \label{prop transitive smallest integer} (c.f. Theorem \ref{thm simple quotient branching law} below)
Let $\pi \in \mathrm{Irr}(G_{n+1})$ be thickened. Let $j^*=\mathrm{lev}(\pi)$. Let $\tau$ be a simple quotient of $\pi^{[i]}$ for some $i$. Let $\sigma \in \mathrm{Irr}^c(G_{n-i})$ be good to $\nu^{1/2}\cdot \pi$. Suppose the quotient $G_n$-map $s$ from $\pi$ to $\tau \times \sigma$ is non-zero, when restricted to ${}_{j^*}\Sigma(\pi)$. Then the unique map from ${}^-\pi$ to ${}^-\tau \times \sigma$ is also non-zero, when restricted to ${}_{j^*}\Sigma({}^-\pi)$.
\end{proposition}

\begin{proof}
One realizes ${}_{j^*}\Sigma(\pi)$ as $\Pi_{j^*}\times {}^-\pi$ and for such the explicit realization, see the proof of Lemma \ref{lem composition factors of applying geo on mirabolic}. Then there exists a (generic, not necessarily irreducible) quotient of finite length $\lambda$ of $\Pi_{j^*}$ such that we have an induced surjection
\[  t: \Pi_{j^*}\times {}^-\pi \rightarrow \lambda \times {}^-\pi 
\]
and the given non-zero map $\Pi_{j^*}\times {}^-\pi  \rightarrow \tau \times \sigma $ factors through $t$. (To construct such the quotient, one may apply an element in the Bernstein center \cite{BD84} that annihilates $\tau \times \sigma$.)

The next step is to apply Lemma \ref{lem coming from smallest mirabolic}. We replace $\kappa$ with $\Pi_{j^*}$, and $\omega$ with ${}^-\pi$, and $\lambda$ with $\lambda \times {}^-\pi$. We shall provide the details below.

We have the following commutative diagram:
\[  \xymatrix{ 
         \lambda \times {}^-{}^-\pi  \ar@{^{(}->}[r]^{\iota'} & {}^{(j^*)}(\lambda \times {}^-\pi) \ar[rrrd]^r  &&  \\                
{}^{j^*}\Gamma({}^-{}^-\pi)  \ar@{^{(}->}[r]^{\iota} \ar@{->>}[u]^{t''}   &                   {}_{j^*}\Sigma(\pi)_{\bar{S}_{j^*}, \psi}   \ar[r]^p \ar@{->>}[u]^{t'} & \pi_{\bar{S}_{j^*}, \psi} \ar[d]^{q} \ar[r]^{s'}   & {}^{(j^*)}(\sigma \times \tau) \ar[r]^{s''} & \sigma \times {}^-\tau \\                                   
	&	     & {}^-\pi=\pi_{\bar{R}_{j^*}, \psi} \ar[urr]^{f''} & & 		           															},
\]
where 
\begin{itemize}
\item $\iota$ comes from the bottom layer of Leibniz's rule in Lemma \ref{lem submodule embedding after derivative} (also see Lemma \ref{lem composition factors of applying geo on mirabolic});
\item $t'$ is induced from $t$; 
\item $t''$ is the functorial map from $t$ and Leibniz's rule;
\item $f''$ in the last triangle is from Lemma \ref{lem projection factor through};
\item $s'$ is induced from $s$ by taking ${}_{\bar{S}_{j^*}}$;
\item $s''$ is the projection to the indecomposable component $\sigma \times {}^-\tau$ in ${}^{(j^*)}(\pi \times \tau)$;
\item $p$ is induced from the embedding in the BZ filtration;
\item $r$ comes from the above discussion;
\item ${}^-{}^-\pi$ means the highest left derivative of ${}^-\pi$.
\end{itemize}

 Now, we have:
\begin{itemize}
\item  $r \circ \iota' \neq 0$ by Lemma \ref{lem coming from smallest mirabolic}. 
\end{itemize} 
These two imply that $ f''\circ (q\circ p\circ \iota)\neq 0$ as desired. (We remark that $q \circ p \circ \iota \neq 0$ in Lemma \ref{lem submodule embedding after derivative}, can also be deduced indirectly from here. This, in particular, gives that ${}^{j^*}\Gamma({}^-{}^-\pi)$ coincides with ${}_{j^*}\Sigma({}^-\pi)$ in ${}^-\pi$.)
\end{proof}

\section{Branching law from BZ derivatives} \label{s branching from bz derivatives}

\subsection{Branching law from simple quotients of BZ derivatives}

The goal of this section is to prove the following:

\begin{theorem} \label{thm simple quotient branching law}
Let $\pi \in \mathrm{Irr}(G_{n+1})$. Let $\tau$ be a simple quotient of $\pi^{[i]}$. Let $\sigma \in \mathrm{Irr}^c(G_{i-1})$ good to $\nu^{1/2}\cdot \pi$. Then
\begin{enumerate}
\item $\mathrm{Hom}_{G_n}(\pi, \sigma \times \tau) \neq 0$; and 
\item $\mathcal L_{rBL}(\pi, \sigma \times \tau)=i$; and
\item $\mathcal L_{lBL}(\pi, \sigma\times \tau)=\mathrm{lev}(\pi)$.
\end{enumerate}
\end{theorem}

\subsection{A basic criterion for simple quotients of BZ derivatives} \label{ss basic criterion}

Before proving Theorem \ref{thm simple quotient branching law}, we introduce some basic results and more notations. For a segment $\Delta=[a,b]_{\rho}$, define:
\[  \Delta^-=[a,b-1]_{\rho}, \quad {}^-\Delta=[a+1,b]_{\rho} , \quad \Delta^0=\Delta, \quad {}^0\Delta=\Delta, 
\]
\[  \Delta^{[-]}=[a+\frac{1}{2}, b-\frac{1}{2}]_{\rho}, \quad {}^{[-]}\Delta=[a+\frac{1}{2}, b-\frac{1}{2}]_{\rho} ,
\]
and we also define $\Delta^{[0]}=[a+\frac{1}{2},b+\frac{1}{2}]_{\rho}$ and ${}^{[0]}\Delta=[a-\frac{1}{2}, b-\frac{1}{2}]_{\rho}$. 

Let $\mathfrak m=\left\{  \Delta_1, \ldots, \Delta_r\right\} \in \mathrm{Mult}$. Write $\Delta_k=[a_k, b_k]_{\rho_k}$ and let $n_k=n(\rho_k)$. Define
\[ \mathfrak m^{(i)}=\left\{ \left\{\Delta_1^{\#_1}, \ldots, \Delta_r^{\#_r} \right\}:\ \forall p,\ \#_p=- \mbox{ or } 0, \mbox{ and} \sum_{p:\ \#_p=-} n_p= i \right\}
\]
\[ {}^{(i)}\mathfrak m=\left\{ \left\{{}^{\#_1} \Delta_1, \ldots, {}^{\#_r}\Delta_r \right\}:\ \forall p,\ \#_p=- \mbox{ or } 0, \mbox{ and} \sum_{p:\ \#_p=-} n_p= i  \right\} .
\]
We similarly define the shifted version for $\mathfrak m^{[i]}$ and ${}^{[i]}\mathfrak m$ by replacing $-$ with $[-]$ and replacing $0$ by $[0]$. 

\begin{lemma} \label{lem multisegment simple q bz} \cite[Lemma 3.11]{Ch22+}
Let $\mathfrak m=\left\{ \Delta_1, \ldots, \Delta_r \right\} \in \mathrm{Mult}$. Let $\tau$ be a simple quotient of $\langle \mathfrak m \rangle^{(i)}$. Then $\tau \cong \langle \mathfrak n \rangle$ for some $\mathfrak n \in \mathfrak m^{(i)}$. 
\end{lemma}

\begin{proof}
It follows from \cite[Lemma 7.3]{Ch21} and the embedding $\langle \mathfrak m \rangle \hookrightarrow \zeta(\mathfrak m)$ in Section \ref{ss basic notations}.
\end{proof}

\subsection{Proof of Theorem \ref{thm simple quotient branching law} for thickened case} \label{ss bz for derivatives thickened}

\begin{lemma} \label{lem branching bz der thickened}
Theorem \ref{thm simple quotient branching law} holds if $\pi$ is thickened.
\end{lemma}

\begin{proof}
We use the notations in Section \ref{ss basic criterion}. We first prove (1) and (2). Let $\pi'=\sigma \times \tau$. The standard argument in BZ filtration gives that (e.g. see the argument in \cite[Proposition 2.5]{Ch21}): for all $i'<i$, and for all $k \geq 0$,
\[  \mathrm{Ext}^k_{G_{n+1-i'}}(\pi^{[i']}, {}^{(i'-1)}\pi') = 0 .
\]
Thus, a long exact sequence argument gives that:
\[  \mathrm{Hom}_{G_{n}}(\pi, \pi') \cong \mathrm{Hom}_{G_n}(\Lambda_{i-1}(\pi), \pi') .
\]
Now the latter Hom is non-zero since 
\[  \mathrm{Hom}_{G_n}(\Sigma_i(\pi), \pi') \cong \mathrm{Hom}_{G_{n-i+1}}(\pi^{[i]}, {}^{(i-1)}\pi') \cong \mathrm{Hom}_{G_{n-i+1}}(\pi^{[i]}, \tau) \neq 0. 
\]
Then a long exact sequence argument gives (1). Then, (2) also follows, for instance, by Corollary \ref{cor integer for branching law equal layer}.

Let $i^*=\mathrm{lev}(\pi)$. From (1), we would have that $\mathrm{Hom}_{M'}({}_{i^*}\Sigma(\pi), \tau \times \sigma) \neq 0$ and (3) if we can show the following claim:\\

\noindent
{\it Cliam:} For $j<i^*$, 
\[ \mathrm{Hom}_{G_n}({}_j\Sigma(\pi),  \tau \times \sigma) =0 .
\]

\noindent
{\it Proof of claim:} Suppose the Hom is non-zero for some $j<i^*$. Then, by the Frobenius reciprocity,
\[ \mathrm{Hom}_{G_{n-j+1}}({}^{[j]}\pi, \tau^{(k)}) \neq 0 ,
\]
where $k=j-i$. Now set $\mathfrak m =\left\{ \Delta_1, \ldots, \Delta_r \right\}$ for the multisegment associated to $\pi$. Then ${}^{[j]}\pi$ admits a filtration with successive quotients to be some quotients of $\theta(\zeta(\mathfrak n))^{\vee}$ for $\mathfrak n \in {}^{[j]}\mathfrak m$ by using the surjection $\theta(\zeta(\mathfrak n))^{\vee}\twoheadrightarrow \pi$ and the geometric lemma (see the proof of \cite[Lemma 7.3]{Ch21} again). 

Similarly, set 
\[   \mathfrak M= \left\{ \Delta_1^{\#_p}, \ldots, \Delta_r^{\#_p}: \forall p,\ \#_p=[--], [-], \mbox{ or } [0], \sum_{p:\ \#_p=[--]}2n_p+\sum_{q:\ \#_q=[-]}n_q=j \right\}, 
\]
where $[a,b]_{\rho}^{[--]}=[a+\frac{1}{2} ,b-\frac{3}{2}]_{\rho}$. (Note that the definition is well-defined for our situation by using the thickening condition.) There is a filtration on $\tau^{(k)}$ with successive quotients isomorphic to some submodules of $\zeta(\mathfrak n')$ for some $\mathfrak n' \in \mathfrak M$. 

Thus, by a standard argument (see e.g. \cite[Proposition 2.3]{Ch22}), the condition that $\mathrm{Hom}_{G_{n-j}}({}^{[j]}\pi, \tau^{(k)})\neq 0$ implies $\mathfrak n =\mathfrak n'$ for some $\mathfrak n \in {}^{[j]}\mathfrak m$ and $\mathfrak n' \in \mathfrak M$. Now it suffices to show that the last equation is not possible, which will be proceeded combinatorially as follows. 

Suppose $\mathfrak n=\mathfrak n'$ for some $\mathfrak n \in {}^{[j]}\mathfrak m$ and $\mathfrak n' \in \mathfrak M$. We fix any $\leq$-minimal $\rho$ in $\mathrm{csupp}(\mathfrak m)$. Let 
\[   N_c= |\mathfrak n_{a=\nu^{-1/2+c}\rho}| , \quad  N_c' =|\mathfrak n'_{a=\nu^{-1/2+c}\rho}| .
\]
It follows from the definition of $\mathfrak M$  that
\[  N_c'\leq |\mathfrak m_{a=\nu^{c-1}\rho}| 
\] 
and so $N_c \leq |\mathfrak m_{a=\nu^{c-1}\rho}|$ (by $\mathfrak n=\mathfrak n'$). 

Then, we have $N_0=0$, $N_1\leq |\mathfrak m_{a=\rho}|$, and more generally, $N_c \leq|\mathfrak m_{a=\nu^{c-1}\rho}|$. Then $N_0=0$ implies that the segments in $\mathfrak m_{a=\rho}$ are shifted and truncated to get segments in $\mathfrak n$. Hence, there are at least $|\mathfrak m_{a=\rho}|$ segments in $\mathfrak n_{a=\nu^{1/2}\rho}$ (by the thickening condition). Then this forces $N_1=|\mathfrak m_{a=\rho}|$ and then this implies that all segments in $\mathfrak m_{a=\nu\rho}$ are shifted and truncated to get segments in $\mathfrak n$. Inductively, with varying different $\leq$-minimal $\rho$,  we deduce that 
\[  \mathfrak n =\left\{ {}^{[-]}\Delta_1, \ldots, {}^{[-]}\Delta_r \right\}.\]
However, this contradicts that $j<i^*$. This proves the claim.
\end{proof}

\subsection{Proof of Theorem \ref{thm simple quotient branching law} for general case} \label{ss proof of branching from bz}

We now consider $\pi \in \mathrm{Irr}$ to be arbitrary. Let $\tau$ be a simple quotient of $\pi^{[i]}$ for some $i$. The proof for (1) and (2) is the same as the ones in Lemma \ref{lem branching bz der thickened}.

We consider (3). Let $\pi'$ be the irreducible (thickened) representation such that ${}^-\pi' \cong \pi$ and $\mathrm{lev}(\pi')=\mathrm{lev}(\pi)$. By Theorem \ref{thm deform segments} below (whose proof is independent of this section), there exists a simple quotient $\tau'$ of $\pi'{}^{[i]}$ such that ${}^-\tau' \cong \tau$ and $\mathrm{lev}(\tau')=\mathrm{lev}(\pi')$. 

Now choose a cuspidal representation $\sigma$ good to $\pi'$. We have that $\tau' \times \sigma$ is a simple quotient of $\pi'$ by Lemma \ref{lem projection factor through} (or by Lemma \ref{lem branching bz der thickened}). Now it follows from Proposition \ref{prop transitive smallest integer} that the unique quotient map from $\pi$ to $\tau \times \sigma$,  restricted to ${}_{j^*}\Sigma(\pi)$, is also non-zero. Since $j^*$ is the largest possible integer, we have that $\mathcal L_{rBL}(\pi, \tau\times \sigma)=j^*$.

\section{Deforming and thickening simple quotients of BZ derivatives} \label{s deform and thicken simple quotients}

Deforming a simple quotient of a BZ derivative is more straightforward by taking the highest BZ derivative. Thickening relies on the Zelevinsky combinatorial realization on the highest derivative \cite{Ze80} and then one does the integrals of some cuspidal representations steps by steps.

The idea is to establish some kind of commutativity of integrals and BZ derivatives. In priori, we do not have the exhaustion theorem for simple quotients for BZ derivatives at this point.

\subsection{A lemma for thickening}

We also later need the following result, which follows e.g. by an application of a result of M\'inguez \cite[Th\'eo\`eme 7.5]{Mi09} (also see Jantzen \cite{Ja07}): 

\begin{lemma} \label{lem repn after integral}
Let $\mathfrak m \in \mathrm{Mult}_{\rho}$ and let $\pi=\langle \mathfrak m \rangle$. Let $\rho \in \mathrm{Irr}^c$ such that any segment $\widetilde{\Delta}$ in $\mathfrak m$ satisfies that $a(\widetilde{\Delta})\not\cong \rho$. Let $\widetilde{\mathfrak m}$ be the multisegment associated to $I_{\rho}^k(\pi)$. Let $k$ be the number of segments $\widetilde{\Delta}$ in $\mathfrak m$ satisfying $\nu\cdot \rho \cong a(\widetilde{\Delta})$ i.e.
\[    k =|\mathfrak m_{a=\nu\cdot \rho}| .
\]
Then there is no segment $\widetilde{\Delta}$ in $\widetilde{\mathfrak m}$ such that $a(\widetilde{\Delta})\cong \nu^{-1}\rho$ i.e. $\widetilde{\mathfrak m}_{a=\nu^{-1}\rho}=\emptyset$.  
\end{lemma}

\begin{proof}
For a segment $\widetilde{\Delta}=[a,b]_{\rho}$, define ${}^+\widetilde{\Delta}=[a-1,b]_{\rho}$. Since $\mathfrak m_{a=\rho}=\emptyset$, the multisegment for $I_{\rho}^k(\pi)$ is obtained by replacing each segment $\widetilde{\Delta}$ in $\mathfrak m_{a=\nu^{-1}\rho}$ with ${}^+\widetilde{\Delta}$. As a result, there is no segment $\widetilde{\Delta}$ of the form $a(\widetilde{\Delta})\cong \nu^{-1}\rho$. 
\end{proof}

\begin{lemma} \label{lem repn after integral varepsilon}
We use the notations in Lemma \ref{lem repn after integral}. Let $\tau'$ be a simple quotient of $(I_{\rho}^k(\pi))^{(i)}$ (for some $i$). Then $\varepsilon_{\rho}(\tau')=k$.
\end{lemma}

\begin{proof}
By applying Lemmas \ref{lem repn after integral} and \ref{lem multisegment simple q bz}, one can then proceed to compute $\varepsilon_{\rho}(\tau')=k$. 
\end{proof}

\begin{lemma} \label{lem simple quotient after integral}
We use the notations in Lemma \ref{lem repn after integral}.  Let $\tau$ be a simple quotient of $\pi^{(i)}$ (for some $i$). Then $I_{\rho}^k(\tau)$ is a simple quotient of $(I_{\rho}^k(\pi))^{(i)}$.
\end{lemma}

\begin{proof}
Let $n=n(\pi)$ and let $p_a=k\cdot n(\rho)$. Let $\tau$ be an irreducible quotient of $\pi^{(i)}$.  We have surjections:
\[ \xymatrix{ (\rho^{\times k} \times \pi)^{(i)} \ar@{->>}[r] & \rho^{\times k} \times \pi^{(i)} \ar@{->>}[r] & \rho^{\times k} \times \tau \\
     (I^k_{\rho}(\pi))^{(i)} \ar@{^{(}->}[u]  & & I_{\rho}^{k}(\tau) \ar@{^{(}->}[u] \\
}.
\]
 Here the vertical injections are induced from the embeddings in Definition \ref{def integral derivative}.

For simplicity, set $\widetilde{\pi}=I_{\rho}^{k}(\pi)$ and $\widetilde{\tau}=I_{\rho}^{k}(\tau)$. Since $\varepsilon^L_{\rho}(\pi)=0$, we have
\begin{align} \label{eqn epsilon integeral 1}
  \varepsilon^L_{\rho}(\widetilde{\pi}) = k
\end{align}
(see e.g. \cite[Section 6]{Mi09}). By Lemma \ref{lem multisegment simple q bz}, we have that $\varepsilon^L_{\rho}(\tau)=0$, and so we also have:
\begin{align} \label{eqn epsilon integeral 2}
\varepsilon^L_{\rho}(\widetilde{\tau})=k .
\end{align}

On the other hand, we have a short exact sequence:
\[ \xymatrix{ 0 \ar[r] & (\widetilde{\pi})^{(i)} \ar[r] & (\rho^{\times k}\times \pi)^{(i)} \ar[r] & Q^{(i)} \ar[r] & 0 \\
    },
\]
where $Q$ is the cokernel of the embedding $\widetilde{\pi} \hookrightarrow \rho^{\times k}\times\pi$. We also observe that, any factor $\omega$ in $Q$ satisfies $\omega_{\bar{N}}$ cannot be of the form 
\[ (*)\quad \rho^{\times k} \boxtimes \omega'
\]
for some $\omega'$, and so in other words,
\[  \varepsilon^L_{\rho}(\omega) < k  . \]
 Here $\bar{N} \subset G_n$  takes all the matrices of the form $\begin{pmatrix} I_{p_a} & * \\ & I_{n-p_a} \end{pmatrix}$. 

Now we see that 
\begin{itemize}
\item[(1)] $Q^{(i)}$ cannot contain the factor $\widetilde{\tau}$ 
\end{itemize}
since $\widetilde{\tau}_{N'}$ contains a factor of the form (*).

One can similarly conclude the following:
\begin{itemize}
\item[(2)] any simple composition factor $\omega$ of the cokernel $\widetilde{\tau} \hookrightarrow \rho^{\times k} \times \tau$ has $\varepsilon_{\rho}^L(\omega)<k$.
\end{itemize}

On the other hand, we have:
\begin{itemize}
\item[(3)] any simple quotient $\omega$ of $\widetilde{\pi}^{(i)}$ has $\varepsilon_{\rho}^L(\omega)=k$ by Lemma \ref{lem repn after integral varepsilon}; and
\item[(4)] similarly, $\varepsilon_{\rho}^L(\widetilde{\tau})=k$ by (\ref{eqn epsilon integeral 2}).
\end{itemize}
Let $\mathrm{pr}$ be the composition of horizontal maps in the toppest diagram in this proof. By (1) and (4), the image of $\mathrm{pr}$ for $\widetilde{\pi}^{(i)}$ contains the factor $\widetilde{\tau}$ and hence is non-zero. Now, (2) and (3) imply that some simple quotient of $\widetilde{\pi}^{(i)}$ must be mapped to $\widetilde{\tau}$ under $\mathrm{pr}$, as desired.
\end{proof}

\subsection{Deformation of derivatives}


We now explain a construction of simple quotients or submodules of Bernstein-Zelevinsky derivatives by taking the highest BZ derivative. The key idea for Theorem \ref{thm deform segments}(2) is to repeatedly apply Lemma \ref{lem simple quotient after integral}.

\begin{theorem} \label{thm deform segments}
Let $\pi \in \mathrm{Irr}(G_n)$ be thickened.
\begin{enumerate}
\item For any irreducible submodule $\tau$ of $\pi^{(i)}$, $\tau^-$ is also an irreducible submodule of $(\pi^{-})^{(i)}$. Moreover, $\mathrm{lev}(\tau)=\mathrm{lev}(\pi)$.
\item For any irreducible submodule $\tau'$ of $(\pi^{-})^{(i)}$, there is an irreducible submodule $\tau$ of $\pi^{(i)}$ such that $\tau^-\cong \tau'$. 
\end{enumerate}
\end{theorem}

\begin{remark} \label{rmk on deformation derivatives}
\begin{enumerate}
\item The condition of thickedness guarantees that the map from isomorphism classes of irreducible submodules of $\pi^{(i)}$ to isomorphism classes of irreducible submodules of $(\pi^-)^{(i)}$ given by $\tau \mapsto \tau^-$ is an injection. This is not true if we drop the condition.
\item One also compares with Proposition \ref{level preserving strong comm}, which uses derivatives instead of BZ derivatives.
\end{enumerate}
\end{remark}

\begin{proof}
We first prove (1). Let $\tau$ be an irreducible submodule of $\pi^{(i)}$. It follows from Lemma \ref{lem multisegment simple q bz} that 
\[ \mathrm{lev}(\tau)=\mathrm{lev}(\pi) .
\]
Then
\[ \tau^{[-]}\cong {}^{[i^*]}\tau \hookrightarrow {}^{[i^*]}(\pi^{(i)}) \cong ({}^{[i^*]}\pi)^{(i)}= (\pi^{[-]})^{(i)} ,
\]
where the commutativity between left and right derivatives follows from taking Jacquet functors in stages, see e.g. \cite[Lemma 5.8]{Of20} for more details.

We now consider (2). We shall prove a quotient version of the statement. Set $\pi'={}^-\pi$. Let $\mathfrak m'$ be the multisegment such that $\pi'\cong \langle \mathfrak m' \rangle$. 

Let $\rho_1', \ldots, \rho_r'$ be all the cuspidal representations appearing in $\left\{ a(\Delta): \Delta \in \mathfrak m' \right\}$. We shall arrange the cuspidal representations such that 
\[  \rho_i' \not> \rho_j'  \]
for any $i<j$. Let $\rho_j=\nu^{-1}\cdot \rho_j'$. Let 
\[  k_j=|\left\{ \Delta \in \mathfrak m': a(\Delta) \cong \rho_j' \right\}| .
\]

Let $\tau'$ be a simple quotient of $\pi^{(i)}$ for some $i$. Then, by Lemma \ref{lem simple quotient after integral}, $I_{\rho_1}^{k_1}(\tau')$ is a simple quotient of $(I_{\rho_1}^{k_1}(\pi'))^{(i)}$. But now using Lemma \ref{lem repn after integral}, one can apply Lemma \ref{lem simple quotient after integral} again for $I_{\rho_2}^{k_2}$. Repeatedly, we have that $(I_{\rho_r}^{k_r}\circ \ldots \circ I_{\rho_1}^{k_1})(\pi')^{(i)}$ has a simple quotient isomorphic to $I_{\rho_r}^{k_r}\circ \ldots I_{\rho_1}^{k_1}(\tau')$. Note that 
\[ (I_{\rho_r}^{k_r}\circ \ldots \circ I_{\rho_1}^{k_1})(\pi')\cong \pi, \quad {}^-(I_{\rho_r}^{k_r}\circ \ldots I_{\rho_1}^{k_1}(\tau')) \cong \tau' ,
\]
where the first isomorphism follows from again \cite[Th\'eor\`eme 7.5]{Mi09} and \cite[Theorem 8.1]{Ze80} and the second isomorphism follows similarly with the additional Lemma \ref{lem multisegment simple q bz}. Now shifting by $\nu^{-1}$, we then obtain (2).
\end{proof}

\begin{remark}
Even if $\pi$ is thickened, it is in general not true that a composition factor of $\pi^{(i)}$ has the same level as $\pi$ (c.f. Remark \ref{rmk on deformation derivatives}). For example, let $\mathfrak m=\left\{ [0,1], [0,1],[2,3] \right\}$ and let $\pi=\langle \mathfrak m \rangle$. In this case, $\pi=\langle [0,1]\rangle \times \langle \left\{ [0,1],[2,3]\right\} \rangle$, which is irreducible. Then $\pi^{(1)}$ admits a filtration with two successive quotients $\langle [0] \rangle \times \langle \left\{[0,1],[2,3]\right\}\rangle$ and $\langle [0,1] \rangle \times \langle [0]\rangle \times \langle [2,3]\rangle$. The second one has a composition factor $\langle [0,3] \rangle \times \langle [0] \rangle$, which has level $2$.  
\end{remark}

\begin{corollary} \label{cor unique highest derivative to give submodule}
Let $\pi \in \mathrm{Irr}$. Let $\tau$ be a simple quotient of $\pi^{(i)}$ for some $i$. Let $\pi'$ be the irreducible (thickened) representation such that ${}^-\pi' \cong \pi$ and $\mathrm{lev}(\pi)=\mathrm{lev}(\pi')$. Let $\mathfrak h$ be the highest left derivative multisegment of $\pi$. Then there is a unique representation $\tau'$ such that $D^L_{\mathfrak h}(\tau')=\tau$. For such $\tau'$, it satisfies $\mathrm{lev}(\tau')=\mathrm{lev}(\pi)$. 
\end{corollary}

\begin{proof}
By Theorem \ref{thm deform segments}(2), there exists a simple quotient $\tau'$ of $\pi'{}^{(i)}$ such that ${}^-\tau' \cong \tau$.

We then have that
\[   \pi'{}^{(i)} \twoheadrightarrow \tau'.\]
 With a similar proof to Lemma \ref{lem right multi submulti}, we have that for all segments $\Delta$,
\[ \eta^L_{\Delta}(\tau') \leq \eta^L_{\Delta}(\pi') .\]
By using $\pi'$ is thickened and Lemma \ref{lem multisegment simple q bz},  $\mathrm{lev}(\tau')=\mathrm{lev}(\pi')$. This then implies that the above inequality has to be an equality. Hence we have $\mathfrak{hd}^L(\tau')=\mathfrak{hd}^L(\pi)$ (see \cite[Corollary 8.6]{Ch22+d} for an argument).

Now $D_{\mathfrak h}^L(\tau')\cong {}^-\tau'\cong \tau$, where the first isomorphism follows from \cite[Theorem 1.3]{Ch22+} and the second isomorphism follows from Theorem \ref{thm deform segments}(1). The uniqueness follows from the uniqueness of the operators $I_{\Delta}$, and $\mathrm{lev}(\tau')=\mathrm{lev}(\pi)$ follows from $\mathrm{lev}(\tau')=\mathrm{lev}(\pi')$, where the last equality on levels again follows from Lemma \ref{lem multisegment simple q bz}.
\end{proof}

\section{Exhaustion and relevance conditions} \label{s exhaustion relev conditions}

\subsection{Exhaustion condition}

We first define the exhaustion condition. \\

\noindent
({\bf Exhaustion Condition}) \label{thm exhaust bz derivatives} 
Let $\pi \in \mathrm{Irr}$. We say that the {\it exhaustion condition} holds for $\pi$ if for any $i$ and for any simple quotient $\tau$ of $\pi^{(i)}$, there exists a multisegment $\mathfrak m$ such that $D_{\mathfrak m}(\pi)\cong \tau$.

\begin{lemma} \label{lem producing mor exhuastion}
Suppose $\pi \in \mathrm{Irr}$ satisfies the exhaustion condition. For any $\sigma \in \mathrm{Irr}^c$ good to $\pi$, the exhaustion condition also holds for $\sigma \times \pi$.
\end{lemma}

\begin{proof}
This follows by using Leibniz's rule.
\end{proof}




\subsection{A technical consequence from the exhaustion condition}

\begin{lemma} \label{lem admissible exhaust form}
Suppose the exhaustion condition holds for some $\pi \in \mathrm{Irr}$. Let $\rho$ be a $\leq$-maximal element in $\mathrm{csupp}(\pi)$. Let $\mathfrak p=\mathfrak{mxpt}(\pi, \rho)$ (see Section \ref{ss multisegment for eta}). Let $\omega=D_{\mathfrak p}(\pi)$. Let $\mathfrak n \in \mathrm{Mult}$ be Rd-minimal to $\omega$. Let $i=l_a(\mathfrak n+\mathfrak p)$. Then $D_{\mathfrak n}(\omega)$ is also a simple quotient of $\pi^{(i)}$ if and only if $\mathfrak n+\mathfrak p$ is also minimal to $\pi$.
\end{lemma}

\begin{proof}
For the only if direction, suppose $D_{\mathfrak n}(\omega)$ is a simple quotient of a BZ derivative of $\pi$. Then the exhaustion condition implies that there exists a multisegment $\mathfrak n'$ such that $D_{\mathfrak n'}(\pi) \cong D_{\mathfrak n}(\omega)$. We choose $\mathfrak n'$ to be the minimal one. Then by Lemma \ref{lem delta reduced subset}, $\mathfrak p \subset \mathfrak n'$. Now by Lemma \ref{lem commute and minimal}, we have that $\mathfrak n'-\mathfrak p$ is minimal to $\omega$ and $D_{\mathfrak n'-\mathfrak p}(\omega)\cong D_{\mathfrak n}(\omega)$. By the uniqueness, we have $\mathfrak n=\mathfrak n'-\mathfrak p$ and so $\mathfrak n'=\mathfrak n+\mathfrak p$ is admissible to $\pi$.

For the if direction, suppose $\mathfrak n+\mathfrak p$ is minimal. Then we have that 
\[  D_{\mathfrak n+\mathfrak p}(\pi) \cong D_{\mathfrak n}\circ D_{\mathfrak p}(\pi)=D_{\mathfrak n}(\omega)\]
by Lemma \ref{lem commute and minimal}. Now we have that $D_{\mathfrak n+\mathfrak p}(\pi)$ is a simple quotient of $\pi^{(i)}$ (see \cite[Proposition 1.2]{Ch22+}). 
\end{proof}

\subsection{Relevance condition}

We now introduce another condition: \\

\noindent
({\bf Relevance condition}) Let $\pi \in \mathrm{Irr}(G_{n+1})$ and $\pi' \in \mathrm{Irr}(G_n)$ with $\mathrm{Hom}_{G_n}(\pi, \pi')\neq 0$. We say that $(\pi, \pi')$ satisfies the {\it relevance condition} if $(\pi, \pi')$ is also a relevant pair.

\begin{proposition} \label{prop exhaustion thm from branching}
Fix an integer $n$. Suppose the relevance condition holds for all pairs of the form $(\omega \times \sigma, \pi')$ satisfying
\begin{itemize}
\item $\mathrm{Hom}_{G_n}(\omega \times \sigma, \pi')\neq 0$;
\item $\omega \in \mathrm{Irr}(G_m)$ for some $m \leq n$;
\item $\pi' \in \mathrm{Irr}(G_n)$;
\item $\sigma \in \mathrm{Irr}^c(G_{n+1-m})$ good to $\omega$ and $\nu^{-1/2}\pi'$. 
\end{itemize}
 Then the exhaustion condition holds for all $\pi \in \mathrm{Irr}(G_{n+1})$. 
\end{proposition}

\begin{proof}
\noindent
{\bf Step 1: Reduction to a case in the hypothesis.} Fix an integer $i$. Let $\tau$ be a simple quotient of $\pi^{[i]}$. Then we can find $\sigma \in \mathrm{Irr}^c$ such that 
\[   \mathrm{Hom}_{G_n}(\pi, \tau\times \sigma) \neq 0
\]
and $\tau \times \sigma$ is irreducible. Let $j^*=\mathrm{lev}(\pi)$. Then $\mathcal L_{lBL}(\pi, \tau \times \sigma)=j^*$ by Theorem \ref{thm simple quotient branching law}. 

Now let $\rho$ be a $\leq$-minimal element in $\mathrm{csupp}(\pi)$. Let $\mathfrak p=\mathfrak{mxpt}^L(\pi, \rho)$. Let $\omega=D^L_{\mathfrak p}(\pi)$. Then, we have 
\[   \omega  \times \mathrm{St}(\mathfrak p) \twoheadrightarrow \pi .
\]
Thus $\mathrm{Hom}_{G_n}(\omega \times \mathrm{St}(\mathfrak p), \tau \times \sigma) \neq 0$. Then one can find $\sigma' \in \mathrm{Irr}^c$ by a refinement of the left version of Lemma \ref{lem standard trick first form} (also see Lemma \ref{lem construct one more non-zero element}) such that 
\[   \mathrm{Hom}_{G_n}(\omega \times \sigma', \tau \times \sigma)\neq 0
\]
and  
\begin{align} \label{eqn level in exhaustion}
\mathcal L_{lBL}(\omega \times \sigma', \tau \times \sigma)=j^*.
\end{align}

\ \\

\noindent
{\bf Step 2: Use the relevance condition.} Let $\widetilde{\omega}=\nu^{-1/2}\cdot \omega$. By using the relevance condition and Theorem \ref{thm symmetric property of relevant}, there exist multisegments $\mathfrak m$ and $\mathfrak n$ such that 
\[   D^L_{\mathfrak n}((\widetilde{\omega} \times \nu^{-1/2}\sigma')) \cong D^R_{\mathfrak m}(\tau \times \sigma) 
\]
and they satisfy the strong commutativity condition i.e. $(\mathfrak n, \mathfrak m, \nu^{-1/2}\cdot (\omega \times \sigma'))$ is a strongly LdRi-commutative triple by Proposition \ref{prop dual multi strong comm}.

By the cuspidal condition, we must have that $[\sigma'] \in \mathfrak n$ and $[\sigma] \in \mathfrak m$, and so we have:
\[   D^L_{\mathfrak n-[\nu^{-1/2}\cdot \sigma']}(\widetilde{\omega}) \cong D^R_{\mathfrak m-[\sigma]}(\tau) .
\]
Let $\mathfrak n'=\mathfrak n-[\nu^{-1/2}\cdot \sigma']$ and let $\mathfrak m'=\mathfrak m-[\sigma]$. By (the left version of) Proposition \ref{prop smallest integer in strong triple} and Corollary \ref{cor integer for branching law equal layer} with (\ref{eqn level in exhaustion}), we then have that 
\begin{align} \label{eqn matching level}
  l_a(\mathfrak n')=j^*-l_a(\mathfrak p) .
\end{align}
By Theorem \ref{thm unique of relevant pairs}, we can assume that $(\mathfrak n',\mathfrak m',  \omega)$ is strongly minimal LdRi-commutative.\\

\noindent
{\bf Step 3.}
\noindent
{\it Claim:} Let $\mathfrak h=\mathfrak{hd}^L(\nu^{-1/2}\cdot \pi)$. Then $\mathfrak n'=\mathfrak h-\nu^{-1/2}\cdot \mathfrak p$.

\noindent
{\it Proof of claim:} Suppose not. Then $D^L_{\mathfrak n'}(\widetilde{\omega})\not\cong {}^{[-]}\pi$ (by Lemmas \ref{lem delta reduced subset}, \ref{lem commute and minimal} and uniquness of minimality). Let $i_l^*=l_a(\mathfrak n')$ and $i_r^*=l_a(\mathfrak m')$. Then we obtain a map from ${}^{[i_l^*]}\omega$ to $\tau^{(i_r^*)}$ factoring through  $D^L_{\mathfrak n'}(\omega)$. On the other hand, we have a map from ${}^{[i_l^*]}\omega$ to $\tau^{(i_r^*)}$ factoring through ${}^{[-]}\pi$ by (\ref{eqn matching level}). This contradicts the multiplicity one of 
\[   \mathrm{Hom}_{G_{n-j^*}}({}^{[i_l^*]}\omega, \tau^{(i_r^*)}) ,
\]
which is deduced from the multiplicity one of $\mathrm{Hom}_{G_{n-j^*}}({}^{[j^*]}(\omega \times \sigma'), (\tau\times \sigma)^{(i_r^*)})$ by Corollary \ref{cor integer determining branching law and layers}. 

Thus, we must have that $D_{\mathfrak n'}(\widetilde{\omega}) \cong {}^{[-]}\pi$. By the commutativity result, we have that $\mathfrak h-\nu^{-1/2}\cdot\mathfrak p$ is Ld-minimal to $\widetilde{\omega}$ and $D^L_{\mathfrak h-\nu^{-1/2}\cdot\mathfrak p}(\widetilde{\omega})\cong {}^{[-]}\pi$. Hence, uniqueness of minimality implies that $\mathfrak n'=\mathfrak h-\nu^{-1/2}\cdot\mathfrak p$. This proves the claim. \\

\noindent
{\bf Step 4: Show $I^R_{\mathfrak m'}$ preserves levels.} We now return to the proof. Let $\widetilde{\pi}=\nu^{-1/2}\pi$. The strong LdRi-commutativity of $( \mathfrak n, \mathfrak m, \nu^{-1/2}\cdot(\omega \times \sigma'))$ now implies the strong LdRi-commutativity of $( \mathfrak n',\mathfrak m', \widetilde{\omega})$. This implies that $(\mathfrak h, \mathfrak m', \widetilde{\pi})$ is still a strongly LdRi-commutative triple by Corollary \ref{cor commute max rho}. Hence, we have 
\begin{align} \label{eqn isomorphism in exhuastion part}
 D_{\mathfrak h}^L\circ I^R_{\mathfrak m'}(\widetilde{\pi})\cong I^R_{\mathfrak m'}\circ D_{\mathfrak h}^L(\widetilde{\pi}) \cong I^R_{\mathfrak m'}\circ D_{\mathfrak n'}^L(\widetilde{\omega})\cong \tau ,
\end{align}
where the first isomorphism follows from Proposition \ref{prop strong commute imply commute}, the second isomorphism follows from Corollary \ref{cor commute minimal strong rdli comm}, and the third isomorphism follows from canceling the derivatives from integrals.

Now, the uniqueness in Corollary \ref{cor unique highest derivative to give submodule} implies that 
\begin{align} \label{eqn preserve level for integrals}
   \mathrm{lev}(I^R_{\mathfrak m'}(\nu^{-1/2}\cdot \pi))=\mathrm{lev}(\pi) .
\end{align}

\noindent
{\bf Step 5: Complete the proof using the double integral.} 
Now, by Theorem \ref{thm double integral} with (\ref{eqn preserve level for integrals}), we have a multisegment $\mathfrak m''$ such that 
\[ (I^R_{\mathfrak m''}\circ I^R_{\mathfrak m'}(\widetilde{\pi}))^- \cong \widetilde{\pi}  \]
and $\mathrm{lev}(I^R_{\mathfrak m''}\circ I^R_{\mathfrak m'}(\widetilde{\pi}))=\mathrm{lev}(\widetilde{\pi})$. Now the level preservation of $I^R_{\mathfrak m'}$ also implies the strong commutativity of $(\mathfrak h, \mathfrak m'', I^R_{\mathfrak m'}(\widetilde{\pi}))$. Thus, we have:
\begin{align*}
 I^R_{\mathfrak m''}(\tau) &  \cong I^R_{\mathfrak m''}\circ D_{\mathfrak h}^L \circ I^R_{\mathfrak m'}(\widetilde{\pi}) \\
 &\cong D^L_{\mathfrak h}(I^R_{\mathfrak m''}\circ I^R_{\mathfrak m'}(\widetilde{\pi})) \\
   & \cong {}^-(I^R_{\mathfrak m''}\circ I^R_{\mathfrak m'}(\widetilde{\pi}))  \\
	 & \cong  \nu \cdot (I^R_{\mathfrak m''}\circ I^R_{\mathfrak m'}(\widetilde{\pi}))^- \\
	 & \cong \nu^{1/2}\cdot \pi , 
\end{align*}
where the first isomorphism follows from (\ref{eqn isomorphism in exhuastion part}), the second isomorphism follows from Lemma \ref{level preserving strong comm} (from (\ref{eqn preserve level for integrals})) and Proposition \ref{prop strong commute imply commute}, and the third isomorphism follows from $\mathrm{lev}(I^R_{\mathfrak m''}\circ I^R_{\mathfrak m'}(\widetilde{\pi}))=\mathrm{lev}(\widetilde{\pi})=|\mathfrak h|$. Taking $D^R_{\mathfrak m'}$, we now have $D^R_{\mathfrak m'}(\nu^{1/2}\cdot \pi)\cong \tau$, as desired.
\end{proof}

\begin{remark}
The exhaustion theorem above also has some similarity to the proof of symmetry property in Theorem \ref{thm symmetric property of relevant}. One may regard that instead of determining the relevance for $(\pi, \tau\times \sigma)$ directly, we determine the relevance for $(\tau \times \sigma, \pi)$ in the above proof. 
\end{remark}

\section{Necessity of generalized relevance and exhaustion theorem} \label{s necessarity and exhaustion}

\subsection{Exhuastion $\Rightarrow$ necessity }


\begin{proposition} \label{prop refinement conjecture}
Let $\pi \in \mathrm{Irr}(G_{n+1})$ and let $\pi' \in \mathrm{Irr}(G_n)$. Suppose the following two conditions hold:
\begin{enumerate}
\item the exhaustion condition holds for $\pi$; and
\item the exhaustion condition holds for any irreducible representation of $G_k$ with $k \leq n$, particularly $\pi'$.
\end{enumerate}
If $\mathrm{Hom}_{G_n}(\pi, \pi')\neq 0$ and $\mathcal L_{rBL}(\pi, \pi')=i^*$, then $(\pi, \pi')$ is $i^*$-relevant. 
\end{proposition}


\begin{proof}

{\bf Step 1: Induction setup} \\

When the relative rank for $(\pi, \pi')$ is $0$, it follows from Lemma \ref{lem basic thm branching}. We now assume that the relative rank for $(\pi, \pi')$ is non-zero. By  Proposition \ref{prop smallest integer in strong triple} and Corollary \ref{cor integer for branching law equal layer}, it suffices to show the relevance (instead of $i^*$-relevance).

Let $\rho$ be a $\leq$-maximal element in $\mathrm{csupp}( \pi)\cup \mathrm{csupp}(\nu^{-1/2}\cdot\pi')$ such that $\rho$ is relevant to $\pi'$ if $\rho \in \mathrm{csupp}(\pi)$; or $\rho$ is relevant to $\pi$ otherwise. (Such $\rho$ exists since we are assuming $\mathcal{RR}(\pi, \pi')>0$.) \\

\noindent
{\bf Step 2: First case of induction and reducing to the inductive case by BZ filtration technique} \\
{\bf Case 1:} Suppose $\rho \notin \mathrm{csupp}(\nu^{-1/2}\cdot\pi')$. Let $\mathfrak p=\mathfrak{mxpt}^{R}(\pi, \rho)$. Let $\omega =D_{\mathfrak p}^R(\pi)$. This gives a surjection:
\[  \mathrm{St}(\mathfrak p)\times \omega \twoheadrightarrow \pi
\]
with the kernel denoted by $\kappa$.

Now, we have
\[ (*)\quad  \mathrm{dim}~\mathrm{Hom}_{G_n}(\mathrm{St}(\mathfrak p)\times \omega, \pi') =1 .
\]
Let $i^*=\mathcal L_{rBL}(\mathrm{St}(\mathfrak p)\times \omega, \pi')$. Thus, we also have:
\begin{itemize}
\item By Corollary \ref{cor integer determining branching law and layers}, 
\[  \mathrm{dim}~\mathrm{Hom}_{G_{n+1-i^*}}((\mathrm{St}(\mathfrak p)\times \omega)^{[i^*]}, {}^{(i^*-1)}\pi') =1 .
\]
\item By Lemma \ref{lem construct one more non-zero element},
\[ \mathrm{dim}~\mathrm{Hom}_{G_n}(\sigma \times \omega, \pi') =1 \]
for some cuspidal representation $\sigma$ good to $\nu^{1/2}\pi$ and $\pi'$, and $\mathcal L_{rBL}(\sigma \times \omega, \pi')=i^*$. 
\item Note that $\sigma \times \omega$ also satisfies the exhaustion condition by the given hypothesis and Lemma \ref{lem producing mor exhuastion}.
\item Now, by induction on the relative rank, there exist multisegments $\mathfrak m$ and $\mathfrak n$ such that
\begin{align} \label{eqn isomorphism under strong in relevance}
  D^R_{\mathfrak m}(\nu^{1/2}\cdot(\sigma \times \omega)) \cong D^L_{\mathfrak n}(\pi') .
\end{align}
and $l_a(\mathfrak m)=i^*$. By a cuspidal condition, we have that $[\sigma]\in \mathfrak m$. Let $\mathfrak m'=\mathfrak m-[\sigma]$ and $\widetilde{\omega}=\nu^{1/2}\omega$. We shall assume that $\mathfrak m'$ is minimal to $\widetilde{\omega}$ (by Theorem \ref{thm unique of relevant pairs}). \\
\end{itemize}

\noindent
{\bf Step 3: Determine the admissibility of $\mathfrak m'+\nu^{1/2}\mathfrak p$ by using information from simple quotients of BZ derivatives} \\

\noindent
{\it Claim:} Let $\widetilde{\pi}=\nu^{1/2}\pi$. We have that $\mathfrak m'+\nu^{1/2}\mathfrak p$ is admissible to $\widetilde{\pi}$. Moreover, $\mathfrak m'+\nu^{1/2}\mathfrak p$ is minimal to $\widetilde{\pi}$. \\

 For simplicity, let $\lambda= \mathrm{St}(\mathfrak p)\times \omega $. Note $D_{\mathfrak m'}(\widetilde{\omega})$ is a simple quotient of $\lambda^{[i^*]}$ (by using $\lambda^{[i^*]}$ has a quotient $\omega^{[i^*-n(\sigma)]}$ and $l_a(\mathfrak m')=i^*-n(\sigma)$ in the last bullet of Step 2) and we let $q$ be the quotient map. Suppose the first assertion (admissibility) of the claim does not hold. Then, by the exhaustion condition on $\pi$ and Lemma \ref{lem admissible exhaust form}, $D_{\mathfrak m'}(\widetilde{\omega})$ is not a simple quotient of $\pi^{[i^*]}$. Thus, $q$ restricted to $\kappa^{[i^*]}$ is non-zero. 

 On the other hand, we also have a non-zero map $\pi^{[i^*]}$ to ${}^{(i^*-1)}\pi'$ by the independence part in Corollary \ref{cor integer determining branching law and layers}. Composing with the quotient map
\[  \lambda^{[i^*]} \rightarrow \pi^{[i^*]} ,
\]
we have a map from $\lambda^{[i^*]}$ to ${}^{(i^*-1)}\pi'$ whose restriction to $\kappa^{[i^*]}$ is zero. Hence, we arrive that
\[ \mathrm{dim}~ \mathrm{Hom}_{G_{n+1-i^*}}(\lambda^{[i^*]}, {}^{(i^*-1)}\pi') \geq 2,
\]
giving a contradiction to the first bullet in Step 2. \\

We now show that $\mathfrak m'+\nu^{1/2}\mathfrak p$ is minimal to $\widetilde{\pi}$. This is similar to the argument in Lemma \ref{lem admissible exhaust form} by using Lemma \ref{lem delta reduced subset} and the uniqueness of minimality for $\mathfrak m'$. In more details, let $\mathfrak k$ be the minimal multisegment such that $D_{\mathfrak k}(\widetilde{\pi}) \cong D_{\mathfrak m'}\circ D_{\nu^{1/2}\mathfrak p}(\widetilde{\pi})$. By Lemma \ref{lem delta reduced subset}, $\mathfrak k=\mathfrak m''+\nu^{1/2}\mathfrak p$ for some multisegment $\mathfrak m''$. By Lemma \ref{lem commute and minimal}, $\mathfrak m''$ is minimal to $\widetilde{\omega}$ and by the uniqueness of minimality, $\mathfrak m''=\mathfrak m'$ as desired.
\\

\noindent
{\bf Step 4: Get back the original case by using commutativity of the minimal sequences}

We now return to the proof. The strong commutativity for $(\mathfrak m,  \mathfrak n, \nu^{1/2}\cdot (\sigma\times \omega))$ implies the strong commutativity for $(\mathfrak m', \mathfrak n, \nu^{1/2}\cdot \omega)$. By Step 3, we have $\mathfrak m'+\nu^{1/2}\mathfrak p$ is also minimal to $\widetilde{\pi}$. Thus, we have that $(\mathfrak m'+\nu^{1/2}\mathfrak p, \mathfrak n, \widetilde{\pi})$ is also a minimal strongly RdLi-commutative triple by Corollary \ref{cor commute max rho}. Now, 
\[   D_{\mathfrak m'+\nu^{1/2}\mathfrak p}(\widetilde{\pi})\cong D_{\nu^{1/2}\mathfrak p}\circ D_{\mathfrak m'}(\widetilde{\pi})\cong D_{\mathfrak m'} \circ D_{\nu^{1/2}\mathfrak p}(\widetilde{\pi})\cong D^L_{\mathfrak n}(\pi') , \]
where the second isomorphism follows from Lemma \ref{lem commute and minimal} and the last isomorphism follows from (\ref{eqn isomorphism under strong in relevance}). This shows the relevance for $(\pi, \pi')$. \\

\noindent
{\bf Step 5: The second case and using the symmetry of relevance}

{\bf Case 2:} Suppose $\rho \in \mathrm{csupp}(\pi')$. Then $ \rho \notin \mathrm{csupp}(\nu^{-1/2} \cdot \pi)$ by the maximality. The argument is very similar to Case 1, and so we will sketch some main steps:
\begin{itemize}
 \item Let $\mathfrak q=\mathfrak{mx}(\pi',\rho)$. Then
\[ \mathrm{dim}~\mathrm{Hom}_{G_n}(\pi,  D_{\mathfrak q}(\pi') \times \mathrm{St}(\mathfrak q))= 1. \]
 \item By the right version of Lemma \ref{lem construct one more non-zero element}, there exists a $\sigma \in \mathrm{Irr}^c$ good to $\nu^{-1/2}\cdot \pi$ and $\pi'$ such that 
\[ \mathrm{dim}~ \mathrm{Hom}_{G_n}(\pi,   D_{\mathfrak q}(\pi') \times \sigma) =1
\]
\item By induction on the relative rank, $(\pi, D_{\mathfrak q}(\pi') \times \sigma)$ is relevant and so is $(\pi, D_{\mathfrak q}(\pi'))$. Then, with Theorem \ref{thm symmetric property of relevant}, there exist multisegments $\mathfrak m$ and $\mathfrak n'$ such that:
\[   D^L_{\mathfrak m}(\nu^{-1/2}\cdot \pi) \cong D^R_{\mathfrak n'}\circ D^R_{\mathfrak q}(\pi') 
\]
and $(\mathfrak n', \mathfrak m, D^R_{\mathfrak q}(\pi'))$ is a strongly RdLi-commutative triple. We shall choose $\mathfrak m$ ito be Li-minimal to $D^R_{\mathfrak q}(\pi')$ and choose $\mathfrak n'$ to be Rd-minimal to $D^R_{\mathfrak q}(\pi')$. 
\item As in the argument of above claim (which uses the exhaustion condition for $\pi'$, and Lemma \ref{lem admissible exhaust form}, and one considers simple submodules of $\pi'^{(i^*)}$), we have $\mathfrak n'+\mathfrak q$ is Rd-admissible to $\pi'$. By Lemma \ref{lem delta reduced subset}, uniqueness of minimality implies that $\mathfrak n'+\mathfrak q$ is also Rd-minimal to $\pi'$. 
\item Now, the relevance of $(\pi', \pi)$ follows by a commutation using the minimal sequence in Step 4.
\item This concludes the relevance for $(\pi, \pi')$ by the symmetry in Theorem \ref{thm symmetric property of relevant}.
\end{itemize}
\end{proof}

\subsection{Proof of necessity and exhaustion} \label{ss proof of sufficiency exh}

\begin{theorem} \label{thm refine brancing}
Let $\pi \in \mathrm{Irr}(G_{n+1})$ and let $\pi' \in \mathrm{Irr}(G_n)$. If $\mathrm{Hom}_{G_n}(\pi, \pi')\neq 0$ and $\mathcal L_{rBL}(\pi, \pi')=i^*$, then $(\pi, \pi')$ is $i^*$-relevant. 
\end{theorem}

\begin{theorem} \label{thm exhaustion thm}
The exhaustion condition holds for all irreducible representations of $G_n$.
\end{theorem}

\noindent
{\it Proof of Theorems \ref{thm refine brancing} and  \ref{thm exhaustion thm}.} We shall prove by induction on $n$ (for $G_n$). When $n=0,1$, the statements are clear. 

 Suppose Theorems \ref{thm refine brancing} and \ref{thm exhaustion thm} hold for some $n>0$. Then, for any $\omega \in \mathrm{Irr}(G_m)$ ($m\leq n$) and any $\sigma \in \mathrm{Irr}^c(G_{n+1-m})$ good to $\omega$,  Lemma \ref{lem producing mor exhuastion} implies that the exhaustion condition holds for $\sigma \times \omega$. The assumption of course also implies the exhaustion condition holds for $\pi' \in \mathrm{Irr}(G_n)$. Then, Proposition \ref{prop refinement conjecture} implies that the relevance condition holds for any such pairs $(\sigma \times \omega, \pi')$. 

The above verifies conditions in Proposition \ref{prop exhaustion thm from branching} and so we have Theorem \ref{thm exhaustion thm} holds for $n+1$. This, with Proposition \ref{prop refinement conjecture} again, in turn gives that Theorem \ref{thm refine brancing} holds for $n+1$. \qed



\part{Appendices}
\section{Appendix A: Specialization map under Bernstein-Zelevinsky filtrations}

\subsection{Some adjointness}
We use the notations in Section \ref{s bz filtration}. We have the following adjointness:
\[  \mathrm{Hom}_{M_n}(\Phi^+(\pi), \tau) \cong \mathrm{Hom}_{M_{n-1}}(\pi, \Phi^-(\tau)) ,
\]
Given a map $\Phi^+(\pi) \stackrel{f}{\rightarrow} \tau$, the adjunction map is given by \cite[Proposition 5.12]{BZ76} (and its proof):
\[    \pi \cong \Phi^-\circ \Phi^+(\pi)  \stackrel{\Phi^-(f)}{\longrightarrow} \Phi^-(\tau) 
\]
On the other hand, given a map $\pi \stackrel{h}{\longrightarrow} \Phi^-(\tau)$, the adjunction map is given by:
\[   \Phi^+(\pi) \stackrel{\Phi^+(h)}{\longrightarrow} \Phi^+\circ \Phi^-(\tau) \hookrightarrow \tau .
\]
Moreover, the embedding $\Phi^+\circ \Phi^-(\tau)$ to $\tau$ is adjoint to the identity morphism from $\Phi^-(\tau)$ to $\Phi^-(\tau)$.

\subsection{Deformation of representations}

Let $P$ be a standard parabolic subgroup of $G_n$ with Levi decomposition $LN$. Identify $L$ with $G_{n_1}\times \ldots \times G_{n_r}$ via $\mathrm{diag}(g_1, \ldots, g_r)\mapsto (g_1, \ldots, g_r)$. Let $\sigma \in \mathrm{Alg}(L)$. For indeterminates $u_1, \ldots, u_r$, define $\pi_u=\mathrm{Ind}_P^{G_n}\sigma$ to be the space of functions from $G_n$ to $\mathbb C[q^{\pm u_1}, \ldots, q^{\pm u_r}]$ satisfying:
\[   f(pg)=\nu(g_1)^{u_1}\ldots \nu(g_r)^{u_r}p.f(g) ,
\]
where $p=\mathrm{diag}(g_1, \ldots, g_r)v$ with $g_i \in G_{n_i}$ and $v$ in the unipotent radical. Let $u^*=(u_1^*, \ldots, u_r^*) \in \mathbb C^r$. Define $\mathrm{sp}: \mathbb{C}[q^{\pm u_1}, \ldots, q^{\pm u_r}] \rightarrow \mathbb C$ by evaluating $u_i=u_i^*$ for all $i$. This induces a $G_n$-representation map, denoted by $\widetilde{\mathrm{sp}}$, from $\pi_u$ to $\pi_{u^*}$ given by:
\[  \widetilde{\mathrm{sp}}(f)(g)=\mathrm{sp}\circ f(g) .
\]

\begin{lemma} \label{lem surjective of sp}
We use the notations above. Then  $\widetilde{\mathrm{sp}}$ is surjective.
\end{lemma}

\begin{proof}
See e.g. \cite[Section 3.1]{CPS17}. Indeed, let $K_0=\mathrm{GL}_n(\mathcal O)$ and we have $G_n=PK_0$. Then, for each $f \in \pi_{u^*}$, $\widetilde{f}$ is determined by its values $\widetilde{f}(k)$ for $k \in K_0$. This similarly holds for $f \in \pi_{u^*}$. Thus, for any $f \in \pi_{u^*}$, one defines $\widetilde{f} \in \pi_u$ such that $\widetilde{f}(k)=f(k)$ for $k \in K_0$. One checks it is well-defined from definitions and then $\mathrm{sp}(\widetilde{f})=f$. This shows surjectivity. 
\end{proof}

\subsection{Specialization for Bernstein-Zelevinsky layers}

\begin{lemma} \label{lem commutative with specialization}
Let $\tau_i=(\Phi^-)^i(\pi_u)$. Similarly, let $\tau^*_i=(\Phi^-)^i(\pi_{u^*})$. Let $s_i=(\Phi^-)^i(\widetilde{\mathrm{sp}})$ and let $s_{i+1}=(\Phi^-)^{i+1}(\widetilde{\mathrm{sp}})$, where $\widetilde{\mathrm{sp}}$ is the specialization map defined in the previous section. Then the following diagram commutes:
\[  \xymatrix{        \Phi^+(\tau_{i+1}) \ar[r]^{\iota} \ar[d]^{\Phi^+(s_{i+1})} & \tau_i \ar[d]^{s_i}  \\
                      \Phi^+(\tau_{i+1}^*) \ar[r]^{\iota^*}       & \tau_i^*   },
\]
where $\iota$ (resp. $\iota^*$) is the adjunction map to the identity morphism from $\tau_{i+1}=\Phi^-(\tau_i)$ to $\Phi^-(\tau_i)$ (resp. from $\tau_{i+1}^*=\Phi^-(\tau_i^*)$ to $\Phi^-(\tau_i^*)$).
\end{lemma}

\begin{proof}

We first have:
\[ \Phi^-(s_i\circ \iota)=\Phi^-(s_i)\circ \Phi^-(\iota)=\Phi^-(s_i)\circ \mathrm{id}=s_{i+1},
\]
where the first equation follows from the functoriality of $\Phi^-$, the second equation follows from $\Phi^-(\iota)=\mathrm{id}$ and the third one follows from the definitions of $s_i$ and $s_{i+1}$.

Similarly, we have
\[  \Phi^-(\iota^*\circ \Phi^+(s_{i+1}))=\Phi^-(\iota^*)\circ (\Phi^-\circ \Phi^+)(s_{i+1})=\mathrm{Id}\circ s_{i+1}=s_{i+1} .
\]
Thus, taking back the adjointness, we have the commutative diagram.
\end{proof}

\begin{proposition} \label{prop preserve bz filtration}
Denote the Bernstein-Zelevinsky filtration for $\pi_u$ by
\[                         0 \subset     \lambda_n    \subset  \ldots  \subset  \lambda_1  \subset \lambda_0=\pi_u
\]
and the Bernstein-Zelevinsky filtration for $\pi_{u^*}$ by
\[   0 \subset \lambda_n^* \subset \ldots \subset \lambda_1^*\subset \lambda_0^*=\pi_{u^*} .
\]
Here $\lambda_i$ is the submodule given by the natural embedding from $\Lambda_i(\pi_u)$ to $\pi_u$, and $\lambda_i^*$ is the submodule given by the natural embedding from $\Lambda_i(\pi_{u^*})$ to $\pi_{u^*}$. Then $\widetilde{\mathrm{sp}}(\lambda_i)=\lambda_i^*$.
\end{proposition}

\begin{proof}
We fix a $i$. Recall that we have the identification:
\[  (\Phi^+)^{i}\circ  (\Phi^-)^{i}(\pi_u) \cong \lambda_i , \quad  (\Phi^+)^{i}\circ  (\Phi^-)^{i}(\pi_{u^*}) \cong \lambda_i^*
\]
Let $t_i=(\Phi^+)^{i}\circ (\Phi^-)^{i}(\mathrm{sp}_{u^*})$. By Lemma \ref{lem commutative with specialization}, we inductively have the following commutative diagram:
\[   \xymatrix{ (\Phi^+)^{i}\circ  (\Phi^-)^{i}(\pi_u) \ar[d]^{t_i}  \ar[r] &      \lambda_i   \ar[d]^{\widetilde{\mathrm{sp}}} \ar@{^{(}->}[r] & \lambda_0 \ar[d]^{\widetilde{\mathrm{sp}}} \\
              (\Phi^+)^{i}\circ  (\Phi^-)^{i}(\pi_{u^*})  \ar[r]        &         \lambda_i^*       \ar@{^{(}->}[r]           & \lambda_0^*   },
\]
where the two horizontal maps are the above isomorphisms. Now, note that $t_i$ is surjective by Lemma \ref{lem surjective of sp} and the exactness of the functors. Thus $\widetilde{\mathrm{sp}}$ on $\lambda_{i}$ is also surjective by using the above commutative diagram.
\end{proof}

\begin{remark} \label{rmk special parameters step by step}
Instead of specializing all the parameters $u_1, \ldots, u_r$ at once, one can also specialize the parameters step-by-step. The above result still holds.
\end{remark}

\section{Appendix B: A short proof for the asymmetry property of left-right BZ derivatives}

With the development of new tools, we give a shorter and more conceptual proof for a result proved in \cite{Ch21}. This also illustrates some powerfulness of those new techniques.

\begin{theorem} \cite{Ch21}
Let $\pi \in \mathrm{Irr}(G_{n+1})$. Let $i^*=\mathrm{lev}(\pi)$. For $i<i^*$ with $\pi^{[i]}\neq 0$ and ${}^{[i]}\pi\neq 0$,
\[  \mathrm{Hom}_{G_{n+1-i}}( \mathrm{cosoc}(\pi^{[i]}), \mathrm{cosoc}({}^{[i]}\pi))=0 .
\]
Here $\mathrm{cosoc}(\pi^{[i]})$ and $\mathrm{cosoc}({}^{[i]}\pi)$ denote the cosocles of $\pi^{[i]}$ and ${}^{[i]}\pi$ respectively.
\end{theorem}

\begin{proof}
Let $\tau$ be a simple quotient of $\pi^{[i]}$ for some $i<i^*$. Then, by Theorem \ref{thm simple quotient branching law}, we can choose a cuspidal representation $\sigma$ of $G_{i-1}$ good to $\nu^{1/2}\pi$ and $\tau$ so that $\mathrm{Hom}_{G_n}(\pi, \tau \times \sigma)\neq 0$ and $\mathcal L_{lBL}(\pi, \tau \times \sigma)=i^*$. By Corollary \ref{cor integer for branching law equal layer}, $\mathrm{Hom}_{G_{n+1-i}}({}^{[i]}\pi, (\tau\times \sigma)^{(i-1)}) = 0$. A cuspidal condition constraint then also gives that 
\[  \mathrm{Hom}_{G_{n+1-i}}({}^{[i]}\pi, \tau)= 0. 
\]
\end{proof}





\end{document}